\documentclass[12pt,english,a4paper]{smfart}

\usepackage[utf8]{inputenc}
\usepackage[T1]{fontenc}
\usepackage{lmodern}
\usepackage{fixcmex}
\usepackage{smfthm}
\usepackage[headings]{fullpage}
\usepackage{amssymb}
\usepackage[all]{xy}

\let\cal\mathcal

\let\hat\widehat
\let\tilde\widetilde
\let\phi\varphi

\let\epsilon\varepsilon

\newcommand\Q{{\bf Q}} 
\newcommand\Z{{\bf Z}}
\newcommand\C{{\bf C}}
\newcommand\N{{\bf N}}

\newcommand\A{{\bf A}}
\newcommand\E{{\bf E}}
\newcommand\B{{\bf B}}
\newcommand\G{{\mathcal G}}
\newcommand\F{{\bf F}}

\renewcommand\i{{\bf i}}
\newcommand\D{{\bf D}}
\newcommand\M{{\bf M}}
\renewcommand\c{{\bf c}}

\newcommand\id{{\mathrm{id}}}

\newcommand\Emb{{\mathrm{Emb}}}

\newcommand\Mat{{\mathrm{Mat}}}
\newcommand\Fil{{\mathrm{Fil}}}

\newcommand\Lie{{\mathrm{Lie}}}
\newcommand\pa{{\mathrm{pa}}}

\newcommand\an{{\mathrm{an}}}
\newcommand\crochu{{\{\{u\}\}}}

\newcommand\brax{{\langle \langle t_\pi \rangle \rangle}}
\newcommand\Sen{{\mathrm{Sen}}}
\newcommand\dR{{\mathrm{dR}}}
\newcommand\ra{{\rightarrow}}
\renewcommand\Ref{{\cal{F}}}
\newcommand\dif{{\mathrm{Dif}}}

\renewcommand{\O}{{{\cal O}}}

\newcommand\Zp{{\Z_p}}
\newcommand\Qp{{\Q_p}}
\newcommand\Cp{{\C_p}}

\newcommand\Qpbar{{\overline{\Q}_p}}

\newcommand\At{{\widetilde{\bf{A}}}}
\newcommand\Atplus{{\widetilde{\bf{A}}^+}}
\newcommand\Bt{{\widetilde{\bf{B}}}}
\newcommand\Btplus{{\widetilde{\bf{B}}^+}}
\newcommand\Et{{\widetilde{\bf{E}}}}
\newcommand\Etplus{{\widetilde{\bf{E}}^+}}

\newcommand\Bdr{{\bf{B}_{\mathrm{dR}}}}
\newcommand\Bdrplus{{\bf{B}_{\mathrm{dR}}^+}}

\newcommand\Btrigplus{{\Bt_{\mathrm{rig}}^+}}

\newcommand\Gal{{\mathrm{Gal}}}
\newcommand\Hom{{\mathrm{Hom}}}
\newcommand\GL{{\mathrm{GL}}}
\newcommand\Frac{{\mathrm{Frac}}}

\newcommand\LT{{\mathrm{LT}}}
\newcommand\la{{\mathrm{la}}}
\newcommand\cbf{{\bf{c}}}
\newcommand\kbf{{\bf{k}}}

\newcommand\cycl{{\mathrm{cycl}}}
\newcommand\crys{{\mathrm{crys}}}
\newcommand\st{{\mathrm{st}}}
\newcommand\rig{{\mathrm{rig}}}

\newcommand\BEK{{\B_{|K}^{\otimes E}}}
\newcommand\unr{{\mathrm{unr}}}
\newcommand\tri{{\mathrm{tri}}}

\newtheorem{question}{Question}

\numberwithin{equation}{section}

\author{Léo Poyeton}
\address{Institut de Mathématiques de Bordeaux}
\email{leo.poyeton@math.u-bordeaux.fr}
\urladdr{https://www.math.u-bordeaux.fr/~lpoyeton/}

\date{\today}

\title{Locally analytic vectors and rings of periods}

\begin{document}

\begin{abstract}
In this paper, we try to extend Berger's and Colmez's point of view, using locally analytic vectors in order to generalize classical cyclotomic theory, in higher rings of periods. We also explain how the formalism of locally analytic vectors recovers the ring $\B_{\Sen}$ of Colmez, and extends to Sen theory in the de Rham case, and to classical $(\phi,\Gamma)$-modules theory. We explain what happens when we try to generalize constructions of $(\phi,\Gamma)$-modules to arbitrary infinitely ramified $p$-adic Lie extensions, and provide a conjecture on the structure of the locally analytic vectors in the corresponding rings. We also highlight the fact that the situation should be very different, depending on wether the $p$-adic Lie extension ``contains a cyclotomic extension'' or not. Finally, we explain how some of these constructions may be related to the construction of a ring of trianguline periods. 
\end{abstract}



\maketitle

\tableofcontents

\section*{Introduction}

Let $p$ be a prime, and let $K$ be a finite extension of $\Q_p$. We fix $\Qpbar = \overline{K}$ an algebraic closure of $K$, and we let $\G_K = \Gal(\overline{K}/K)$ be its absolute Galois group. 

A classical idea in $p$-adic Hodge theory in order to study $p$-adic representations of $\G_K$ is to use an intermediate extension $K_\infty/K$ such that $K_\infty/K$ is nice enough but such that it contains ``most of the ramification'' of $\Qpbar/K$, so that $\Qpbar/K_\infty$ is almost étale in the sense of Faltings (which is the same as saying that the $p$-adic completion of $K_\infty$ is perfectoid). The main example of such an extension is the cyclotomic extension $K(\mu_{p^\infty})$ of $K$, which has been thoroughly used in $p$-adic Hodge theory, notably in Sen theory and $(\phi,\Gamma)$-modules theory. 

In some sense, Kummer extensions are simpler than the cyclotomic extension, and work from Breuil \cite{breuil1998schemas} and Kisin \cite{KisinFiso} show that Kummer extensions are very useful in order to study semistable representations. However, Kummer extensions are never Galois and this implies that we usually have to replace them by their Galois closure which increases the difficulty of the situation. Lubin-Tate extensions attached to uniformizers of $K$, of which the cyclotomic extension when $K=\Qp$ is a particular case, trivialize local class field theory and thus seem particularly useful in order to extend the $p$-adic Langlands correspondence to $\GL_2(K)$ (see for example \cite{KR09,FX13,Ber14MultiLa} for work in this direction). More generally, the interesting framework should be the one of infinitely ramified Galois extensions whose Galois group is a $p$-adic Lie group, with potential applications in Iwasawa theory \cite{venjakob2003iwasawa}. 

Let $V$ be a $p$-adic representation of $\G_K$, and let $K_\infty = K(\mu_{p^\infty})$, $H_K= \Gal(\Qpbar/K_\infty)$ and $\Gamma_K = \Gal(K_\infty/K)$. Recall that the cyclotomic character $\chi_\cycl : \Gamma_K \to \Z_p^\times$ identifies $\Gamma_K$ with an open subgroup of $\Z_p^\times$. Since $\Qpbar/K_\infty$ is almost étale, $(V \otimes_\Qp \Cp)^{H_K} \otimes_{\hat{K_\infty}}\Cp \simeq V \otimes_\Qp \Cp$, so that the study of the $\Cp$-representation $V \otimes_\Qp \Cp$ is reduced to the one of $(V \otimes_\Qp \Cp)^{H_K}$. The idea of Sen to study such a representation \cite{sen1980continuous} is to consider the subspace $D_{\Sen}(V)$ of $K$-finite vectors, which are elements of $(V \otimes_\Qp \Cp)^{H_K}$ which belong to finite dimensional sub-$K$-vector spaces stable by $\Gamma_K$. This is a sub-$K_\infty$-vector space of $(V \otimes_\Qp \Cp)^{H_K}$, and Sen proved that $D_{\Sen}(V) \otimes_{K_\infty}\hat{K_\infty} \simeq (V \otimes_\Qp \Cp)^{H_K}$. 

If $K_\infty$ is any infinitely ramified $p$-adic Lie extension $K_\infty/K$, and if $V$ is a $\Qp$-representation of $\G_K$, then since $\Qpbar/K_\infty$ is almost étale, we still have an isomorphism $(V \otimes_\Qp \Cp)^{H_K} \otimes_{\hat{K_\infty}}\Cp \simeq V \otimes_\Qp \Cp$, but if the dimension of $\Gamma_K = \Gal(K_\infty/K)$ as a $p$-adic Lie group is greater or equal to $2$, then the space of $K$-finite vectors of this semilinear $\hat{K_\infty}$-representation of $\Gamma_K$ is no longer suitable, as shown by \cite[Prop. 1.5]{Ber14SenLa}.

In order to generalize Sen theory to any infinitely ramified $p$-adic Lie extension $K_\infty/K$, Berger and Colmez suggested to replace the space of $K$-finite vectors and the use of normalized Tate's traces maps (which no longer exist in general \cite{fourquaux2009applications}) by the space of locally analytic vectors, which are elements $x$ such that the orbit map $g \mapsto g(x)$ is a locally analytic function on $\Gamma_K$. This gives a decompletion of $(V \otimes_\Qp \Cp)^{\Gal(\Qpbar/K_\infty)}$ into a $\hat{K_\infty}^{\la}$-vector space of dimension $\dim_{\Qp}V$, but in general $\hat{K_\infty}^{\la}$ strictly contains $K_\infty$. 

Recall that the strategy developped by Fontaine (see \cite{fontaine1994representations}) to study $p$-adic representations of $\G_K$ is to construct some $p$-adic rings of periods $B$, which are topological $\Qp$-algebras endowed with an action of $\G_K$ and additional structures such that if $V$ is a $p$-adic representation of $\G_K$, then the $B^{\G_K}$-module $D_B(V) := (B \otimes_{\Qp}V)^{\G_K}$ is endowed with the structures coming from those on $B$, and such that the functor $V \mapsto D_B(V)$ gives some interesting invariants attached to $V$. For Fontaine's strategy to work, one requires that these rings of periods $B$ are $\G_K$-regular in the sense of \cite[1.4.1]{fontaine1994representations} (this implies in particular that $B^{\G_K}$ is a field). We then say that a $p$-adic representation $V$ of $\G_K$ of dimension $d$ is $B$-admissible if $B \otimes_{\Qp} V \simeq B^d$ as $B$-representations. The strategy of Fontaine then consists of classifying $p$-adic representations according to the rings of periods for which they are admissible. In the case where $V$ is admissible, $D_B(V)$ can usually be used to recover $V$, or at least $V_{|\G_L}$ for some finite extension $L$ of $K$.

Colmez has constructed in \cite{colmez1994resultat} a ring of periods $\B_{\Sen}$ which recovers Sen's theory in the cyclotomic setting. Precisely, he defines $\B_{\Sen}^n$ as the set of power series in the variable $u$ over $\Cp$, with radius of convergence $\geq p^{-n}$, and endows it with an action of $\Gal(\Qpbar/K(\mu_{p^n}))$ by $g(u)=u+\log \chi_\cycl(g)$ (this makes sense since $\log\chi_{\cycl}(g) \in p^n\mathcal{O}_K$ if $g \in \G_{K_n}$). He then shows that $(\B_{\Sen}^n)^{\G_{K(\mu_{p^n})}}=K(\mu_{p^n})$ and that $K_\infty \otimes_{K_n}(\B_{\Sen}^n \otimes_{\Qp}V)^{\G_{K(\mu_{p^n})}}$ is isomorphic to $\D_{\Sen}(V)$ for $n$ big enough.

One other key ingredient in the study of $p$-adic representations of $\G_K$ is the theory of $(\phi,\Gamma_K)$-modules, which provides an equivalence of categories $V \mapsto D(V)$ between the category of all $p$-adic representations of $\G_K$ and the category of étale $(\phi,\Gamma_K)$-modules. In Fontaine's theory, $(\phi,\Gamma_K)$-modules are finite dimensional vector spaces, defined over a $2$-dimensional local ring $\B_K$ and endowed with semilinear actions of a Frobenius $\phi$ and of $\Gamma_K$ which commutes one to another.

One variant of the theory, which has been used with many useful applications, is the theory of $(\phi,\Gamma_K)$-modules over the Robba ring $\B_{\rig,K}^\dagger$. The theorem of Cherbonnier-Colmez \cite{cherbonnier1998representations} shows that the category of étale $(\phi,\Gamma_K)$-modules over $\B_K$ actually embedds into the category of $(\phi,\Gamma_K)$-modules over $\B_{\rig,K}^\dagger$ of slope $0$, and the slope filtration theorem of Kedlaya \cite{slopes} shows that this is an equivalence of categories. 

One interesting feature of the Robba ring is that it can be used as a bridge between the classical theory of $(\phi,\Gamma_K)$-modules and $p$-adic Hodge theory, as its elements can be embedded inside $\Bdrplus$. In particular, Berger has shown \cite{Ber02} how to recover the invariants attached to a $p$-adic representation $V$ in $p$-adic Hodge theory from its $(\phi,\Gamma_K)$-module on the Robba ring. 

Kisin and Ren have defined Lubin-Tate $(\phi_q,\Gamma_K)$-modules \cite{KR09} and proved that the category of Lubin-Tate étale $(\phi_q,\Gamma_K)$-modules is equivalent to the one of $\Qp$-representations, but unfortunately a result from Fourquaux and Xie \cite{FX13} shows that those $(\phi_q,\Gamma_K)$-modules are usually not overconvergent. Results from Berger \cite{berger2012multivariable} \cite{Ber14MultiLa} suggest that the right objects to consider are once again the locally analytic vectors inside some higher rings of periods.

In this paper, we try to understand what happens if we use the point of view of Berger-Colmez of locally analytic vectors in ``higher rings of periods''.

Our first remark, which follows from the formalism of locally analytic vectors, is that Colmez's construction of $\B_{\Sen}$ can be generalized to construct rings of periods which ``compute the cyclotomic theory''. More precisely, if $\B$ is a $\Qp$-Banach (or Fréchet) ring endowed with an action of $\G_K$, such that the functor $V \mapsto \D_{\B}(V)^{\la}:=(\B \otimes_{\Qp}V)^{\Gal(\Qpbar/K(\mu_{p^\infty})),\Gamma_K-\la}$ gives interesting invariants of $V$, where $V$ is a $p$-adic representation of $\G_K$, then the ring $\mathcal{C}^{\la}(\Gamma_K,\B)_1$, the stalk at the identity of the sheaf of locally analytic functions on $\Gamma_K$ with coefficients in $\B$, ``computes'' the functor $V \mapsto \D_{\B}(V)^{\la}$, in the sense that 

$$(\mathcal{C}^{\la}(\Gamma_K,\B)_1\otimes_{\Qp}V)^{\G_K} \simeq ((\B \otimes_{\Qp}V)^{H_K})^{\Gamma_K-\la}.$$

In particular, this allows us to provide constructions recovering cyclotomic $(\phi,\Gamma)$-modules and cyclotomic Sen theory for $\Bdrplus$-representations in this spirit, which extend to the $F$-analytic Lubin-Tate case as remarked by Berger \cite[\S 8,9,10]{Ber14MultiLa} and Porat \cite[\S 3]{Porat}.

In order to generalize $(\phi,\Gamma_K)$-modules theory to any infinitely ramified $p$-adic Lie extension, one would like to understand the structure of the rings $(\Bt^I)^{H_K,\Gamma_K-\la}$, where the rings $\Bt^I$ are some higher rings of periods which are properly defined in \S 1. For the theory to behave well and indeed generalize, we should expect that $(\Bt^I)^{H_K,\Gamma_K-\la}$ can be interpreted as a ring of power series in $d$ variables, where $d$ is the dimension of $\Gamma_K$ as a $p$-adic Lie group. We expect that if $K_\infty$ contains a twist by an unramified character of the cyclotomic extension, in which case we say that the extension $K_\infty/K$ contains a cyclotomic extension, then the theory does generalize and  the rings $(\Bt^I)^{H_K,\la}$ can be interpreted as rings of power series in $d$ variables:

\begin{conj}
If $K_\infty/K$ contains a cyclotomic extension, then the rings $(\Bt_K^I)^{\Gamma_n-\an}$ can be interpreted as rings of power series in $d$ variables, with some convergence condition.

More precisely, we expect that, for $n \gg 0$, there exist $d$ elements $x_{1,n},\ldots,x_{d,n}$ in $(\Bt_K^I)^{\Gamma_n-\an}$ such that 
$(\Bt_K^I)^{\Gamma_n-\an}$ is the set of power series $\sum_{\i = (i_1,\dots,i_d) \in \N^d}a_{\i}x_{j,n}^{i_j}$ in the variables $(x_{i,n})_{i \in \{1,\cdots,d\}}$ with coefficients in $K$ such that the series $\sum_{\i = (i_1,\dots,i_d) \in \N^d}a_{\i}x_{j,n}^{i_j}$ converge in $(\Bt_K^I)^{\Gamma_n-\an}$.
\end{conj}

It was proven by Berger in \cite[Thm. 4.4]{Ber14MultiLa} that this conjecture holds when $K_\infty/K$ is a Lubin-Tate extension. In the particular case of the cyclotomic extension, Berger's result shows that $(\Bt_K^I)^{\Gamma_n-\an}$ is a ring of power series in one variable with coefficients in $K$, such that the series converge on some annulus depending only on $n$ and $I$. Moreover, this variable is, up to some power of the Frobenius, exactly the one used in the construction of cyclotomic $(\phi,\Gamma)$-modules. In particular, just as in \cite{Ber14SenLa}, we notice that locally analytic vectors applied to the cyclotomic setting recover the classical theory. 

In this paper, we are able to generalize Berger's result and to prove our conjecture in a particular case:

\begin{theo}
Let $K_\infty/K$ be an infinitely ramified $p$-adic Lie extension which is a successive extension of $\Z_p$-extensions and contains a cyclotomic extension. Then the conjecture above is true for $K_\infty/K$.
\end{theo} 

The fact that we expect the need to contain a cyclotomic extension follows from the following, which shows that for $p$-adic Lie extensions which do not contain a cyclotomic extension, the situation looks different:

\begin{theo}
Let $K_\infty/\Q_{p^2}$ be the anticyclotomic extension, where $\Q_{p^2}$ is the unramified extension of $\Qp$ of degree $2$. Then the rings $(\Bt^I)^{H_K,\Gamma_K-\la}$ are equal to $\Q_{p^2}$ if $0 \in I$.
\end{theo}

If $W$ is a Fréchet representation of a $p$-adic Lie group, the space of locally analytic vectors $W^\la$ can be defined but is too small in general to be able to recover $W$ from $W^{\la}$. We provide in this paper computations of locally analytic vectors for Robba rings in the $F$-analytic Lubin-Tate case, which highlights this fact. We also show that taking $F$-locally analytic vectors in the $(\phi_q,\Gamma_K)$-modules on Robba rings recovers modules defined by Colmez in \cite{colmez2014serie} through different methods:

\begin{theo}
Let $V$ be an $F$-analytic representation of $\G_K$, and let $\D_{\rig}^\dagger(V)$ be its attached $(\phi_q,\Gamma_K)$-module over the Robba ring $\B_{\rig,K}^\dagger$. We have the following:
\begin{itemize}
\item $(\Bt_{\rig,K}^{\dagger})^{\Gamma_K-\la} = (\B_{\rig,K}^\dagger)^{\Gamma_K-\la} = K\langle \langle t_\pi \rangle \rangle$ ;
\item $\D_{\rig}^\dagger(V)^{\Gamma_K-\la}=\cap_{n \geq 0}\phi^n(\D_{\rig}^\dagger(V))$ and is a free $K\langle \langle t_\pi \rangle \rangle$-module of rank $\leq \dim_{\Qp}V$ ;
where $t_\pi$ is the ``Lubin-Tate analog of $t$'' and $K\langle \langle T \rangle \rangle$ denote the set of power series in $T$ with coefficients in $K$ and infinite radius of convergence. 
\end{itemize}
\end{theo}

This theorem alongside theorem 3.23 of \cite{colmez2014serie} show that in general the rank of $\D_{\rig}^\dagger(V)^\la$ as a $K\langle \langle t_\pi \rangle \rangle$-module is strictly smaller than $\dim_{\Qp}V$ and is thus too small to recover $\D_{\rig}^\dagger(V)$. 

Finally, we highlight the fact that thinking of the rings $\mathcal{C}^{\la}(\Gamma_K,\B)_1$ as rings of periods could have applications in order to define rings of periods for trianguline representations: a trianguline representation is a representation such that its attached $(\phi,\Gamma_K)$-module on the Robba  ring is a successive extension of rank $1$ $(\phi,\Gamma_K)$-modules, but that does not mean that the corresponding representation itself is a successive extension of rank $1$ representations, because the $(\phi,\Gamma_K)$-modules of rank $1$ that appear in the decomposition do not need to be étale. Trianguline representations are assumed to be related to representations coming from global geometric objects (see for example \cite{emerton2009p} and \cite{Kisin2003}) and for example the representations attached to overconvergent modular forms of finite slope are trianguline.

In order to better understand and parametrize trianguline representations, it would make sense to construct a ring which would be to trianguline representations what $\B_{\crys}$ is to crystalline representations, and we try to offer candidate rings for that purpose. 

Note that, since unramified representations of $\G_K$ are crystalline and thus trianguline, we expect such rings to contain $\Q_p^{\mathrm{unr}}$. Moreover, when we talk about rings of periods $\B$, we expect that the corresponding modules attached to $p$-adic representations $V$ and defined by $\D_{\B}(V) = (\B \otimes_{\Qp}V)^{\G_K}$ satisfy at least the three following properties: 
\begin{enumerate}
\item $\B$ is reduced;
\item for any $p$-adic representation $V$ of $\G_K$, $\D_{\B}(V)$ is a free $\B^{\G_K}$-module;
\item the $\B$-linear map $\alpha_V: \B \otimes_{\B^{\G_K}}\D_{\B}(V) \longrightarrow \B \otimes_{\Qp}V$ deduced from the inclusion $\D_{\B}(V) \subset \B\otimes_{\Qp}V$ by extending the scalars to $\B$, is injective.
\end{enumerate}

The reason why one would have to define several rings is the following:

\begin{prop}
There is no ring of periods $\B$ satisfying those three properties and containing $\Q_p^{\mathrm{unr}}$ such that for any finite extension $K$ of $\Qp$, $\B$ is a trianguline period ring for $\G_K$.
\end{prop}

Therefore, our ring of trianguline periods of $\G_K$ has to be dependent on $K$. In the case $K=\Qp$, since every rank $1$ representation is trianguline, our ring has to contain every $\exp(\alpha \log t)$ with $\alpha \in E$, a field of coefficients. In particular, the ring $\B_{\tri,\Qp}^\an$ we define is to $\Btrigplus$ what the ring $\B_{\Sen}$ introduced in \cite{colmez1994resultat} is to $\Cp$: $\B_{\tri,\Qp}^\an$ is the ring $\mathcal{C}^{\la}(\Gamma_K,\Btrigplus)_1$, which is also the inductive limit of the rings $\mathcal{C}^{\an}(\Gamma_{K_n},\Btrigplus)$. Proposition \ref{prop F-la for Robba} shows that $(\B_{\tri,\Qp}^n)^{\G_{K_n}} = \Qp\langle \langle t\rangle \rangle$, the set of power series in $t$ with infinite radius of convergence. The module $\D_{\tri,\Qp}^{\an}(V)$ is therefore a module over $\Qp\langle \langle t \rangle \rangle$ and is also endowed with a Frobenius $\phi$ coming from the one on $\Btrigplus$ and an operator $\nabla$ coming from the action of the Lie algebra of $\Gamma_K$ and commutes with the action of $\phi$. We then extend these constructions to the $F$-analytic case, constructing a ring $\B_{\tri,K}^{\an}$ in the same fashion, and we extend Fontaine's classical formalism of admissibility to take this setting into account.

Generalizing the notion of refinements of $p$-adic representations \cite{mazur2000theme} \cite{bellaichechenevier} to our setting, we prove the following:

\begin{theo}
\label{theo intro Btadm implies triang}
Let $V$ be an $F$-analytic representation of $\G_K$ which is $\B_{\tri,K}^{\an}$-admissible. Then $V$ is trianguline. 
\end{theo}

While the ring $\B_{\tri,K}^\an$ is too small to contain the periods of all $F$-analytic trianguline representations of $\G_{K}$, we could adapt our constructions to ``add a log to our ring'', which would cover the semistable periods, but we would still be missing the ``nongeometric'' periods of trianguline representations, which appear in item $(ii)$ of theorem 3.23 of \cite{colmez2014serie}. It is not yet clear how many periods one would have to add to $\B_{\tri,K}^\an$ to get a ring of trianguline periods.

\section*{Structure of the paper}
The first section of the paper recalls the theory of classical rings of periods and the theory of $(\phi,\Gamma)$-modules and the rings it involves. The second section recalls the theory of locally and pro-analytic vectors. In \S 3, we recall the main results from \cite{Ber14SenLa}. We explain in \S 4 how this framework recovers classical Sen theory for $\Bdrplus$-representations, and we compute what $(\B_{\dR}^+)^{H_K,\la}$ looks like in some particular cases with emphasis on the Lubin-Tate one. In section 5 we explain how $(\phi,\Gamma)$-modules theory is recovered through our framework. In \S 6, we explain what we expect to happen in general when trying to generalize $(\phi,\Gamma)$-modules theory by using locally analytic vectors, prove the particular case of the conjecture and highlight some problems which may arise in the anticyclotomic case. The computations of locally analytic vectors in Robba rings is done in \S 7. Finally, \S 8 is devoted to the applications to trianguline representations and towards a construction of rings of trianguline periods.

\section{Classical $p$-adic rings of periods and $(\phi,\Gamma)$-modules}
\subsection{Fontaine's strategy and some rings of periods}
Let $p$ be a prime, let $K$ be a finite extension of $\Q_p$ and let $\G_K = \Gal(\overline{K}/K)$ be its absolute Galois group. Let $k$ be the residual field ok $K$ and let $F=W(k)[1/p]$ be the maximal unramified extension of $\Qp$ inside $K$. Let $\C_p$ be the $p$-adic completion of $\overline{K}$. Let $F_\infty=\Qp(\mu_{p^\infty})$ be the cyclotomic extension of $\Qp$. For $n \geq 1$ let $K_n = K(\mu_{p^n})$ be the extension of $K$ generated by the $p^n$-th roots of unity, and let $K_\infty= \bigcup_{n \geq 1}K(\mu_{p^n})=K \cdot F_\infty$ be the cyclotomic extension of $K$. Let $H_{\Qp} = \Gal(\Qpbar/F_\infty)$ and $\Gamma_{\Qp} = \Gal(F_\infty/\Qp)$. Let $H_K = \Gal(\overline{K}/K_\infty)$ and $\Gamma_K = \Gal(K_\infty/K)$. Recall that the cyclotomic character $\chi_\cycl : \G_K \to \Z_p^\times$ factors through $\Gamma_K$ and identifies it with an open subset of $\Z_p^\times$. We also let $K_0$ denote the maximal unramified extension of $\Qp$ inside $K_\infty$. 

Recall that the strategy developped by Fontaine (see \cite{fontaine1994representations}) to study $p$-adic representations of $\G_K$ is to construct some $p$-adic rings of periods $\B$, which are topological $\Qp$-algebras endowed with an action of $\G_K$ and additional structures such that if $V$ is a $p$-adic representation of $\G_K$, then the $\B^{\G_K}$-module $\D_{\B}(V) := (\B \otimes_{\Qp}V)^{\G_K}$ is endowed with the structures coming from those on $\B$, and such that the functor $\B \mapsto \D_{\B}(V)$ gives some interesting invariants attached to $V$. In Fontaine's original setting, one requires that these rings of periods $\B$ are $\G_K$-regular in the sense of \cite[1.4.1]{fontaine1994representations} (this implies in particular that $\B^{\G_K}$ is a field). We then say that a $p$-adic representation $V$ of $\G_K$ of dimension $d$ is $\B$-admissible if $\B \otimes_{\Qp} V \simeq \B^d$ as $\B$-representations. The strategy of Fontaine then consists of classifying $p$-adic representations according to the rings of periods for which they are admissible. In the case where $V$ is admissible, $\D_{\B}(V)$ can usually be used to recover $V$, or at least $V_{|\G_L}$ for some finite extension $L$ of $K$.

We now recall the construction of some rings of periods. 

Let $\Etplus = \varprojlim\limits_{x \mapsto x^p}\mathcal{O}_{\Cp} = \left\{(x^{(0)},\ldots) \in \mathcal{O}_{\Cp}^\N : (x^{(n+1)})^p=x^{(n)}\right\}$ and recall \cite[Thm. 4.1.2]{Win83} that this ring is naturally endowed with a ring structure which makes it a perfect ring of characteristic $p$ which is complete for the valuation $v_{\E}$ defined by $v_\E(x) = v_p(x^{(0)})$. Let $\Et$ be its field of fractions and note that it is algebraically closed. We denote by $\phi$ the absolute Frobenius $x \mapsto x^p$ on $\Etplus$ and $\Et$. The action of $\G_{\Qp}$ on $\mathcal{O}_{\Cp}$ induces a continuous action of $\G_{\Qp}$ on $\Et$. 

Choose a sequence $\epsilon=(\epsilon^{n})_{n \in \N} \in \Etplus$ of compatible $p^n$-th roots of unity (with $\epsilon^{(1)} \neq 1$). Let $\overline{v} = \epsilon-1 \in \Etplus$ and let $\E_{\Qp}:=\F_p(\!(\overline{v})\!) \subset \Et$. Let $\E = \E_{\Qp}^{sep}$ be the separable closure of $\E_{\Qp}$ inside $\Et$. The field $\E_{\Qp}$ is left invariant by the action of $H_{\Qp}$ so that we have a morphism $H_{\Qp} \to \Gal(\E/\E_{\Qp})$. By \cite[Thm. 3.2.2]{Win83}, it is actually an isomorphism. We also let $\E_K=\E^{H_K}$. Note that $\Gamma_K$ acts on $\E_K$, and that the action of $\G_{\Qp}$ on $\overline{v}$ is given by $g(\overline{v})=(1+\overline{v})^{\chi_\cycl(g)}-1$. 

Let $\At = W(\Et)$ and let $\Atplus = W(\Etplus)$. We also let $\Bt = \Frac(\At) = \At[1/p]$ and $\Btplus = \Atplus[1/p]$. By functoriality of Witt vectors, the action of $\G_{\Qp}$ extends to an action on $\At$ and $\Bt$ that commutes with the Frobenius $\phi$. If $L$ is a finite extension of $\Qp$, we let $\Bt_L = \Bt^{H_L}$ and $\At_L = \At^{H_L}$, where $H_L=\Gal(\Qpbar/L(\mu_{p^\infty}))$. 

Note that any element $x$ of $\Atplus$ can be written as $x=\sum_{k \geq 0}p^k[x_k]$ where the $x_k$ belong to $\Etplus$ and $[\cdot]$ denotes the Teichmüller lift. Recall \cite[1.5.1]{fontaine1994corps} that we have a surjective morphism of rings $\theta : \Atplus \ra \mathcal{O}_{\Cp}$ given by $\theta(x) = \sum_{k \geq 0}p^kx_k^{(0)}$ and whose kernel is a principal maximal ideal of $\Atplus$. This morphism $\theta$ naturally extends to $\Btplus$ to a surjective morphism that we still denote by $\theta : \Btplus \ra \C_p$. For $m \in \N$, we let $\B_m$ be the ring $\Btplus/{\ker(\theta)^m\Btplus}$ and we endow it with the structure of a $p$-adic Banach ring by taking the image of $\Atplus$ as its ring of integers. We let $\Bdrplus = \varprojlim\limits_{m \in \N}\B_m$ be the completion of $\Btplus$ for the $\ker(\theta)$-adic topology and we endow it with the Fréchet topology of the projective limit. By construction, $\theta$ extends to a continuous morphism $\theta: \Bdrplus \to \C_p$ and the action of $\G_{\Qp}$ on $\Btplus$ extends by continuity to a continuous action on $\Bdrplus$. We let $\Bdr$ be the fraction field of $\Bdrplus$. The power series defining $\log[\epsilon]$ converges in $\Bdrplus$ to an element $t$ that generates the maximal ideal $\ker(\theta : \Bdrplus \to \Cp)$ of $\Bdrplus$, so that $\Bdr = \Bdrplus[1/t]$. Note that the action of $\G_{\Qp}$ on $t$ is given by $g(t) = \chi_{\cycl}(g)\cdot t$. We endow $\Bdr$ with a filtration by setting $\Fil^i\Bdr = t^i\Bdrplus$. We call representations that are $\Bdr$-admissible ``de Rham representations''. 

Fontaine has also defined several other rings of periods, among which $\B_{\crys}$ and $\B_{\st}$, in order to study $p$-adic representations. Recall that $\B_\crys$ is endowed with a Frobenius $\phi$, $\B_\st$ contains $\B_\crys$, is endowed with a Frobenius $\phi$ and a monodromy operator $N$ such that $\B_\crys = \B_\st^{N=0}$, and $\Bdr$ is a field endowed with a filtration $\{\Fil^i\Bdr\}_{i \in \Z}$ and such that there is an injective map $\B_{\st} \rightarrow \Bdr$. Moreover, these rings all contain the element $t$, and there exist rings $\B_\crys^+$ and $\B_\st^+$ such that $\B_\crys=\B_\crys^+[1/t]$ and $\B_\st = \B_\st^+[1/t]$. Representations that are $\B_\crys$-admissible and $\B_\st$-admissible are respectively called crystalline and semi-stable representations. The relations between those rings imply that crystalline representations are semi-stable and that semi-stable representations are de Rham. We do not recall the proper definitions of $\B_\crys$ and $\B_\st$ as they are not needed in this note. 

\subsection{Cyclotomic $(\phi,\Gamma)$-modules}
Let us now recall briefly the theory of $(\phi,\Gamma)$-modules and some of the rings involved in the theory. Let $v = [\epsilon]-1$. Let $\A_{\Qp}$ be the $p$-adic completion of $\Zp(\!(v)\!)$ inside $\At$. This is a discrete valuation ring with residue field $\E_{\Qp}$. Since
$$\phi(v) = (1+v)^p-1 \quad \textrm{and} \quad g(v)=(1+v)^{\chi_\cycl(g)}-1 \textrm{ if } g \in \G_{\Qp},$$
the ring $\A_{\Qp}$ and its field of fractions $\B_{\Qp} = \A_{\Qp}[1/p]$ are both stable by $\phi$ and $\G_{\Qp}$. 

For $r > 0$, we define $\Bt^{\dagger,r}$ the subset of overconvergent elements of ``radius'' $r$ of $\Bt$, by

$$\Bt^{\dagger,r}=\left\{x = \sum_{n \ll -\infty}p^n[x_n] \textrm{ such that } \lim\limits_{k \to +\infty}v_{\E}(x_k)+\frac{pr}{p-1}k =+\infty \right\}$$
and we let $\Bt^\dagger = \bigcup_{r > 0}\Bt^{\dagger,r}$ be the subset of all overconvergent elements of $\Bt$. 

Let $\B_{\Qp}^{\dagger,r}$ be the subset of $\B_{\Qp}$ given by
$$\B_{\Qp}^{\dagger,r}=\left\{\sum_{i \in \Z}a_iv^i, a_i \in \Qp \textrm{ such that the } a_i \textrm{ are bounded and } \lim\limits_{i \to - \infty}v_p(a_i)+i\frac{pr}{p-1} = +\infty \right\},$$
and note that $\B_{\Qp}^{\dagger,r} = \B_{\Qp} \cap \Bt^{\dagger,r}$.  

Let $\B_{\Qp}^\dagger = \bigcup_{r > 0}\B_{\Qp}^{\dagger,r}$. By \S 2 of \cite{matsuda1995local}, this is a Henselian field, and its residue ring is still $\E_{\Qp}$. Since $\B_{\Qp}^\dagger$ is Henselian, there exists a finite unramified extension $\B_K^\dagger/\B_{\Qp}^\dagger$ inside $\Bt$, of degree $f$ and whose residue field is $\E_K$. Therefore, there exists $r(K) > 0$ and elements $x_1,\ldots,x_f$ in $\B_K^{\dagger,r(K)}$ such that $\B_K^{\dagger,s} = \bigoplus_{i=1}^f \B_{\Qp}^{\dagger,s}\cdot x_i$ for all $s \geq r(K)$. We let $\B_K$ be the $p$-adic completion of $\B_K^\dagger$ and we let $\A_K$ be its ring of integers for the $p$-adic valuation. One can show that $\B_K$ is a subfield of $\Bt$ stable under the action of $\phi$ and $\Gamma_K$ (see for example \cite[Prop. 6.1]{colmez2008espaces}). Let $\A$ be the $p$-adic completion of $\bigcup_{K/\Qp}\A_K$, taken over all the finite extensions $K/\Qp$. Let $\B = \A[1/p]$. Note that $\A$ is a complete discrete valuation ring whose field of fractions is $\B$ and with residue field $\E$. Once again, both $\A$ and $\B$ are stable by $\phi$ and $\G_{\Qp}$. Moreover, we have $\A^{H_K}=\A_K$ and $\B_K=\B^{H_K}$, so that $\A_K$ is a complete discrete valuation ring with residue field $\E_K$ and fraction field $\B_K = \A_K[1/p]$. If $L$ is a finite extension of $K$, then $\B_L/\B_K$ is an unramified extension of degree $[L_\infty:K_\infty]$ and if $L/K$ is Galois then so is $\B_L/\B_K$, and we have the following isomorphisms: $\Gal(\Bt_L/\Bt_K) = \Gal(\B_L/\B_K) = \Gal(\E_L/\E_K) = \Gal(L_\infty/K_\infty) = H_K/H_L$.

\begin{defi}
If $K$ is a finite extension of $\Qp$, a $(\phi,\Gamma_K)$-module $D$ on $\A_K$ (resp. $\B_K$) is an $\A_K$-module of finite rank (resp. a finite dimensional $\B_K$-vector space) endowed with semilinear actions of $\Gamma_K$ and $\phi$ that commute one to another.   

It is said to be étale if $1 \otimes \phi: \phi^*D \to D$ is an isomorphism (resp. if there exists a basis of $D$ such that $\Mat(\phi) \in \GL_d(\A_K)$).
\end{defi}

If $K$ is a finite extension of $\Qp$ and if $V$ is a $p$-adic representation of $\G_K$, we set 
$$D(V) = (\B \otimes_{\Qp}V)^{H_K}.$$
Note that $D(V)$ is a $(\phi,\Gamma_K)$-module. Moreover, if $V$ is a $p$-adic representation of $\G_K$, then $D(V)$ is étale and $(\B \otimes_{\B_K}D(V))^{\phi=1}$ is canonically isomorphic to $V$ (see \cite[Prop. 1.2.6]{Fon90}). The functors $V \mapsto D(V)$ and $D \mapsto (\B \otimes_{\B_K}D)^{\phi=1}$ then induce an equivalence of tannakian categories between $p$-adic representations of $\G_K$ and étale $(\phi,\Gamma_K)$-modules.

For $r > 0$, we define a valuation $V(\cdot,r)$ on $\Btplus[1/[\overline{v}]]$ by setting
$$V(x,r) = \inf_{k \in \Z}(k+\frac{p-1}{pr}v_{\E}(x_k))$$
for $x = \sum_{k \gg - \infty}p^k[x_k]$. 

For $r=0$, we let $V(\cdot,0)$ be the $p$-adic valuation on $\Btplus[1/[\overline{v}]]$. 

If $I$ is a closed subinterval of $[0;+\infty[$, $I \neq [0,0]$, we let $V(x,I) = \inf_{r \in I, r \neq 0}V(x,r)$ (one can take a look at remark 2.1.9 of \cite{GP18} to understand why we avoid defining $V(\cdot,0)$). We then define the ring $\Bt^I$ as the completion of $\Btplus[1/[\overline{v}]]$ for the valuation $V(\cdot,I)$ if $0 \not \in I$, and as the completion of $\Btplus$ for $V(\cdot,I)$ if $I=[0;r]$. We will write $\Bt_{\mathrm{rig}}^{\dagger,r}$ for $\Bt^{[r,+\infty[}$ and $\Bt_{\mathrm{rig}}^+$ for $\Bt^{[0,+\infty[}$. We also define $\Bt_{\mathrm{rig}}^\dagger = \bigcup_{r \geq 0}\Bt_{\mathrm{rig}}^{\dagger,r}$. 

Let $I$ be a subinterval of $[0,+\infty[$ which is either a subinterval of $]1,+\infty[$ or of the form $[0,r]$, for some $r > 0$. Let $f(Y) = \sum_{k \in \Z}a_kY^k$ be a power series with $a_k \in F$ and such that $v_p(a_k)+k/\rho \to +\infty$ when $|k| \to + \infty$ for all $\rho \in I$. The series $f(v)$ converges in $\Bt^I$ and we let $\B_{\Qp}^I$ denote the set of all $f(\pi)$ with $f$ as above. It is a subring of $\Bt_{\Qp}^I$. 

We also write $\B_{\mathrm{rig},\Qp}^{\dagger,r}$ for $\B_{\Qp}^{[r;+\infty[}$. It is a subring of $\B_{\Qp}^{[r;s]}$ for all $s \geq r$ and note that the set of all $f(v) \in \B_{\mathrm{rig},\Qp}^{\dagger,r}$ such that the sequence $(a_k)_{k \in \Z}$ is bounded is exactly the ring $\B_{\Qp}^{\dagger,r}$. Let $\B_{\Qp}^{\dagger}=\cup_{r \gg 0}\B_{\Qp}^{\dagger,r}$. 

Recall that, for $K$ a finite extension of $\Qp$, there exists a separable extension $\E_K/\E_{\Qp}$ of degree $f=[K_\infty:F_\infty]$ and an attached unramified extension $\B_K^{\dagger}/\B_{\Qp}^{\dagger}$ of degree $f$ with residue field $\E_K$, so that there exists $r(K) > 0$ and elements $x_1,\cdots x_f \in \B_K^{\dagger,r(K)}$ such that $\B_K^{\dagger,s}= \bigoplus_{i=1}^f\B_{\Qp}^{\dagger,s}\cdot x_i$ for all $s \geq r(K)$. If $r(K) \leq \min(I)$, we let $\B_{K}^I$ be the completion of $\B_K^{\dagger,r(K)}$ for $V(\cdot,I)$, so that $\B_K^I=\oplus_{i=1}^f\B_{\Qp}^I\cdot x_i$. 

We actually have a better description of the rings $\B_{\rig,K}^{\dagger,r}$ in general:

\begin{prop}
Let $K$ be a finite extension of $\Qp$. 
\begin{enumerate}
\item There exists $v_K \in \A_K^{\dagger,r(K)}$ whose image modulo $p$ is a uniformizer of $\E_K$ and such that, for $r \geq r(K)$, every element $x \in \B_K^{\dagger,r}$ can be written as $x = \sum_{k \in \Z}a_kv_K^k$, where $a_k \in F'=\Q_p^{\mathrm{unr}} \cap K_\infty$, and the power series $\sum_{k \in \Z}a_kT^k$ is holomorphic and bounded on $\left\{p^{-1/e_Kr} \leq |T| < 1 \right\}$.
\item Let $\mathcal{H}^{\alpha}_{F'}(T)$ be the set of power series $\sum_{k\in \Z}a_kT^k$ where $a_k \in F'$ and such that, for all $\rho \in [\alpha;1[, \lim\limits_{k \to \pm \infty}|a_k|\rho^k=0$ and let $\alpha_K^r = p^{-1/e_Kr}$. Then the map $\mathcal{H}^{\alpha}_{F'}(T) \rightarrow \B_{\rig,K}^{\dagger,r}$ sending $f$ to $f(v_K)$ is an isomorphism.
\end{enumerate}
\end{prop} 
\begin{proof}
The first item is proved in \cite[Prop. 7.5]{colmez2008espaces} and the second one in \cite[Prop. 7.6]{colmez2008espaces}. Be careful that the notations for the rings and the normalizations of the valuations used in Colmez's paper are a bit different than ours. 
\end{proof}

The following theorem is the main result of \cite{cherbonnier1998representations} and shows that every étale $(\phi,\Gamma_K)$-module is the base change to $\B_K$ of an overconvergent module:
\begin{theo}
\label{theo cherbonniercolmez}
If $D$ is an étale $(\phi,\Gamma_K)$-module, then the set of free sub-$\B_K^\dagger$-modules of finite type stable by $\phi$ and $\Gamma_K$ admits a bigger element $D^\dagger$ and one has $D= \B_K \otimes_{\B_K^{\dagger}}D^\dagger$. 
\end{theo}

In particular, if $V$ is a $p$-adic representation of $\G_K$, then there exists an étale $(\phi,\Gamma_K)$-module over $\B_K^\dagger$ which we will denote by $\D^\dagger(V)$ and such that $D(V) = \B_K \otimes_{\B_K^{\dagger}}\D^\dagger(V)$. We let $\D_{\rig}^\dagger(V) = \B_{\rig,K}^\dagger \otimes_{\B_K^{\dagger}}\D^\dagger(V)$. 

If $E$ if a finite extension of $\Qp$, we can make the following definition:
\begin{defi}
A $(\phi,\Gamma_K)$-module over $E \otimes_{\Qp}\B_{\rig,K}^\dagger$ is a finite module $D$ over $E \otimes_{\Qp}\B_{\rig,K}^\dagger$, equipped with a semi-linear Frobenius $\phi_D$ and a continuous semi-linear action of $\Gamma_K$ such that $D$ is free as a $\B_{\rig,K}^\dagger$-module, $\id \otimes \phi_D: \B_{\rig,K}^\dagger\otimes_{\phi,\B_{\rig,K}^\dagger}D \ra D$ is an isomorphism and that the actions of $\phi_D$ and $\Gamma_K$ commute. 
\end{defi}

By \cite[Lemm. 1.30]{Nakapieds}, a $(\phi,\Gamma_K)$-module over $E \otimes_{\Qp}\B_{\rig,K}^\dagger$ is free as an $E \otimes_{\Qp}\B_{\rig,K}^\dagger$-module. We say that a $(\phi,\Gamma_K)$-module over $E \otimes_{\Qp}\B_{\rig,K}^\dagger$ is étale if its underlying $\phi$-module over $\B_{\rig,K}^\dagger$ is étale.

\subsection{Lubin-Tate $(\phi,\Gamma)$-modules}
\label{phigammaLT}

We now recall the theory of $(\phi,\Gamma)$-modules in the Lubin-Tate setting. We let $F$ be a finite extension of $\Qp$, $\pi$ a uniformizer of $\mathcal{O}_F$ and $\LT$ be a Lubin-Tate formal $\mathcal{O}_F$-module attached to the uniformizer $\pi$ of $\mathcal{O}_F$. Let $q$ be the cardinal of the residue field of $F$ and let $h$ be such that $q=p^h$. Let $F_0 = F \cap \Q_p^{\unr}$. We let $F_n$ denote the extension of $F$ generated by the points of $\pi^n$-torsion of $\LT$ for $n \geq 1$, and $F_\infty = \bigcup_{n \geq 1}F_n$. We let $\Gamma_F = \Gal(F_\infty/F)$ and $H_F = \Gal(F_\infty/F)$. By Lubin-Tate's theory \cite[Thm. 2]{LT65}, the Lubin-Tate character $\chi_\pi : \G_F \to \mathcal{O}_F^\times$ induces an isomorphism $\Gamma_F \simeq \mathcal{O}_F^\times$. For $a \in \mathcal{O}_F$, we let $[a](T)$ denote the power series that corresponds to the multiplication by $a$ map on $\LT$. Let $v_0 =0$ and for each $n \geq 1$, let $v_n \in \Qpbar$ be such that $[\pi](v_{n})=v_{n-1}$, with $v_1 \neq 0$.

Recall that we defined rings $\At, \Atplus, \At^I$ and $\Bt, \Btplus, \Bt^I$ previously, and in what follows we will keep the same notations for those rings tensored over $F_0$ (resp. $\mathcal{O}_{F_0}$ in the case of $\At$ and $\At^I$), by $F$ (resp. $\mathcal{O}_F$). We let $\phi_q = \phi^{\circ h}$ and we let $r_k = p^{kh-1}(p-1)$ for $k \geq 1$. 

Recall that by \cite[\S 9.2]{Col02}, there exists $v \in \Atplus$ whose image in $\Etplus$ is $(v_0,v_1,\cdots)$, where $\Etplus = \varprojlim\limits_{x \mapsto x^q}\mathcal{O}_{\Cp}/\pi$ (by \cite[Prop. 4.3.1]{BrinonConrad}, this is the same ring $\Etplus$ as before) and such that $g(v) = [\chi_\pi(g)](v)$ and $\phi_q(v) = [\pi](v)$. We also let $t_\pi= \log_{\LT}(v) \in \Btrigplus$ so that $g(t_\pi)=\chi_\pi(g)\cdot t_\pi$ and $\phi_q(t_\pi)=\pi t_\pi$. Note that when $F=\Qp$ and $\pi=p$, this is exactly the classical $t$ of $p$-adic Hodge theory. 

For $\rho > 0$, let $\rho' = \rho \cdot e \cdot p/(p-1)\cdot (q-1)/q$, where $e$ is the ramification index of $F/\Qp$. Let $I$ be a subinterval of $[0,+\infty[$ which is either a subinterval of $]1,+\infty[$ or of the form $[0,r]$, for some $r > 0$. Let $f(Y) = \sum_{k \in \Z}a_kY^k$ be a power series with $a_k \in F$ and such that $v_p(a_k)+k/\rho' \to +\infty$ when $|k| \to + \infty$ for all $\rho \in I$. The series $f(v)$ converges in $\Bt^I$ and we let $\B_F^I$ denote the set of all $f(v)$ with $f$ as above. It is a subring of $\Bt_F^I$. We also write $\B_{\mathrm{rig},F}^{\dagger,r}$ for $\B_F^{[r;+\infty[}$.

We let $\A_F$ denote the $p$-adic completion of $\mathcal{O}_F(\!(v)\!)$ inside $\At$, and we let $\B_F = \A_F[1/p]$. As in the cyclotomic case, to any extension $L/F$ finite, there corresponds extensions $\A_L/\A_F$ and $\B_L/\B_F$, of degree $[L_\infty : F_\infty]$ where $L_\infty = L \cdot F_\infty$, equipped with actions of $\phi_q$ and $\Gamma_L := \Gal(L_\infty/L)$. As in the cyclotomic case, we fix once and for all a finite extension $K$ of $F$ and will apply the theory to $p$-adic representations of $\G_K = \Gal(\Qpbar/K)$. There is also a theory of $(\phi_q,\Gamma_K)$-modules over $\B_K$, which are finite dimensional $\B_K$ vector spaces endowed with commuting semilinear actions of $\Gamma_K$ and $\phi_q$.  Once again, such a $(\phi_q,\Gamma_K)$-module is said to be étale if there exists a basis in which $\Mat(\phi_q)$ belongs to $\GL_d(\A_K)$. By specializing Fontaine's constructions \cite[A.1.2.6 and A.3.4.3]{Fon90}, Kisin and Ren prove the following, which is \cite[Thm. 1.6]{KR09}:

\begin{theo}
There is a tannakian equivalence of categories between $F$-linear representations of $\G_K$ and étale $(\phi_q,\Gamma_K)$-modules over $\B_K$.
\end{theo}

However, unlike in the cyclotomic case, these $(\phi_q,\Gamma_K)$ modules are rarely overconvergent. Berger showed in \cite{Ber14MultiLa} that the right subcategory of representations corresponding to overconvergent $(\phi_q,\Gamma_K)$-modules was the one of $F$-analytic representations (note however that there are representations which are not $F$-analytic but whose attached $(\phi_q,\Gamma_K)$-module is overconvergent). An $E$-representation $V$ of $\G_K$, where $E \supset F^{\Gal}$, is said to be $F$-analytic if for any $\tau \in \Emb(E,\Qpbar)$, $\tau \neq \id$, the semilinear $\Cp$-representation $\Cp \otimes^{\tau}V$ is trivial. In that case, theorem 10.4 of \cite{Ber14MultiLa} shows that one can attach to $V$ an étale $F$-analytic $(\phi_q,\Gamma_K)$-module $\D_{\rig}^{\dagger}(V)$ on $\B_{\rig,K}^\dagger$, which means that the operator $\frac{\log g}{\log\chi_\pi(g)}$ is $F$-linear on $\D_{\rig}^{\dagger}(V)$. Note that, when $F=\Qp$, every representation of $\G_{K}$ is $\Qp$-analytic.

For an $F$-analytic character $\delta : K^\times \ra E^\times$, we let $w(\delta)$ denote its weight, which is defined by $w(\delta) = \delta'(1)$.

\begin{lemm}
\label{lemm rank 1 phigammaLT}
Let $\D$ be a rank $1$ $F$-analytic $(\phi_q,\Gamma_K)$-module over $E \otimes_K \B_{\rig,K}^\dagger$. Then there exists $\delta$ an $F$-analytic character $K^\times \ra E^\times$ and a basis $e$ of $\D$ in which $g(e) = \delta(\chi_\pi(g))\cdot e$ and $\phi_q(e) = \delta(\pi)\cdot e$.
\end{lemm}
\begin{proof}
This is the same as in \cite[Prop. 3.1]{colmez2008representations}, using \cite[Thm. 10.4]{Ber14MultiLa}.
\end{proof}

\section{Locally and pro-analytic vectors}
Here, we recall some of the theory of locally- and pro-analytic vectors, following the presentation of Emerton in \cite{emerton2004locally} and of Berger in \cite{Ber14MultiLa}.

Let $G$ be a $p$-adic Lie group, and let $W$ be a $\Qp$-Banach representation of $G$. Let $H$ be an open subgroup of $G$ such that there exists coordinates $c_1,\cdots,c_d : H \to \Zp$ giving rise to an analytic bijection $\cbf : H \to \Z_p^d$. We say that $w \in W$ is an $H$-analytic vector if there exists a sequence $\left\{w_{\kbf}\right\}_{\kbf \in \N^d}$ such that $w_{\kbf} \rightarrow 0$ in $W$ and such that $g(w) = \sum_{\kbf \in \N^d}\cbf(g)^{\kbf}w_{\kbf}$ for all $g \in H$. We let $W^{H-\an}$ be the space of $H$-analytic vectors. This space injects into $\cal{C}^{\an}(H,W)$, the space of all analytic functions $f : H \to W$.  Note that $\cal{C}^{\an}(H,W)$ is a Banach space equipped with its usual Banach norm, so that we can endow $W^{H-\an}$ with the induced norm, that we will denote by $||\cdot ||_H$. With this definition, we have $||w||_H = \sup_{\kbf \in \N^d}||w_{\kbf}||$ and $(W^{H-\an},||\cdot||_H)$ is a Banach space.

The space $\cal{C}^{\an}(H,W)$ is endowed with an action of $H \times H \times H$, given by
\[
((g_1,g_2,g_3)\cdot f)(g) = g_1 \cdot f(g_2^{-1}gg_3)
\]
and one can recover $W^{H-\an}$ as the closed subspace of $\cal{C}^{\an}(H,W)$ of its $\Delta_{1,2}(H)$-invariants,  where $\Delta_{1,2} : H \to H \times H \times H$ denotes the map $g \mapsto (g,g,1)$ (we refer the reader to \cite[§3.3]{emerton2004locally} for more details).

We say that a vector $w$ of $W$ is locally analytic if there exists an open subgroup $H$ as above such that $w \in W^{H-\an}$. Let $W^{\la}$ be the space of such vectors, so that $W^{\la} = \bigcup_{H}W^{H-\an}$, where $H$ runs through a sequence of open subgroups of $G$. The space $W^{\la}$ is naturally endowed with the inductive limit topology, so that it is an LB space. 

\begin{lemm}
\label{ringla}
If $W$ is a ring  such that $|\!|xy|\!| \leq |\!|x|\!| \cdot |\!|y|\!|$ for $x,y \in W$, then
\begin{enumerate}
  \item $W^{H-\an}$ is a ring, and $|\!|xy|\!|_H \leq|\!|x|\!|_H \cdot |\!|y|\!|_H$ if $x,y \in W^{H-\an}$;
  \item if $w \in W^\times \cap W^{\la}$, then $1/w \in W^{\la}$. In particular, if $W$ is a field, then  $W^{\la}$ is also a field.
\end{enumerate}
\end{lemm}
\begin{proof}
See \cite[Lemm. 2.5]{Ber14SenLa}.
\end{proof}

It is often useful to choose a specific fundamental system of open neighborhoods of $G$: let $G_0$ be a compact open subgroup of $G$ which is $p$-valued and saturated (see \cite[\S 26 and 27]{schneider2011p} for the definition and proof of existence), with coordinates $\c$, and set $G_{n}=G^{p^{n}}=\left\{ g^{p^{n}}:g\in G_{0}\right\}$ for $n \geq 0$.

These are subgroups (\cite[Remark 26.9]{schneider2011p}) which have induced
coordinates $\c:G_{n}\xrightarrow{\sim}(p^{n}\Z_{p})^{d}$.
The normalization is such that for $w\in W^{G_{n}-\an}$ we can write
\[
g(w)=\sum_{\mathbf{k\in}\mathbb{N}^{d}}c(g)^{\mathbf{k}}w_{\mathbf{k}}
\]
for $g\in G_{n}$ and $\left\{ w_{\mathbf{k}}\right\} _{\mathbf{k}\in\N^{d}}$
with $p^{n\left|\mathbf{k}\right|}w_{\mathbf{k}}\rightarrow0$, and
the Banach norm is given by
\[
|\!|w|\!|_{G_{n}-\an}=\sup_{\mathbf{k}}|\!|p^{n\mathbf{k}}w_{\mathbf{k}}|\!|.
\]
It is easy to check if $w\in W^{G_{n}-\an}$ then $|\!|w|\!|_{G_{m}-\an}\leq|\!|w|\!|_{G_{m+1}-\an}$
for $m\geq n$ and $|\!|w|\!|_{G_{m}-\an}=|\!|w|\!|$
for $m\gg n$ (see \cite[Lemme 2.4]{Ber14SenLa}).

Let $W$ be a Fréchet space whose topology is defined by a sequence $\left\{p_i\right\}_{i \geq 1}$ of seminorms. Let $W_i$ be the Hausdorff completion of $W$ at $p_i$, so that $W = \varprojlim\limits_{i \geq 1}W_i$. The space $W^{\la}$ can be defined but as stated in \cite{Ber14MultiLa} and as will be explained in \S 7, this space is too small in general for what we are interested in, and so we make the following definition, following \cite[Def. 2.3]{Ber14MultiLa}:

\begin{defi}
If $W = \varprojlim\limits_{i \geq 1}W_i$ is a Fréchet representation of $G$, then we say that a vector $w \in W$ is pro-analytic if its image $\pi_i(w)$ in $W_i$ is locally analytic for all $i$. We let $W^{\pa}$ denote the set of all pro-analytic vectors of $W$, so that $W^{\pa} = \varprojlim\limits{i \geq 1}W_i^{\la}$. 
\end{defi}

We extend the definition of $W^{\la}$ and $W^{\pa}$ for LB and LF spaces respectively. 

\begin{prop}
\label{lainla and painpa}
Let $G$ be a $p$-adic Lie group which is a uniform pro-$p$-group, let $B$ be a Banach $G$-ring and let $W$ be a free $B$-module of finite rank, equipped with a compatible $G$-action. If the $B$-module $W$ has a basis $w_1,\ldots,w_d$ in which $g \mapsto \Mat(g)$ is a globally analytic function $G \to \GL_d(B) \subset M_d(B)$, then
\begin{enumerate}
\item $W^{H-\an} = \bigoplus_{j=1}^dB^{H-\an}\cdot w_j$ if $H$ is a subgroup of $G$;
\item $W^{\la} = \bigoplus_{j=1}^dB^{\la}\cdot w_j$.
\end{enumerate}
Let $G$ be a $p$-adic Lie group, let $B$ be a Fréchet $G$-ring and let $W$ be a free $B$-module of finite rank, equipped with a compatible $G$-action. If the $B$-module $W$ has a basis $w_1,\ldots,w_d$ in which $g \mapsto \Mat(g)$ is a pro-analytic function $G \to \GL_d(B) \subset M_d(B)$, then
$$W^{\pa} = \bigoplus_{j=1}^dB^{\pa}\cdot w_j.$$
\end{prop}
\begin{proof}
The part for Banach ring is proven in \cite[Prop. 2.3]{Ber14SenLa} and the one for Fréchet rings is proven in \cite[Prop. 2.4]{Ber14MultiLa}.
\end{proof}

Note that the map $\log \chi_\pi : \Gamma_K \to \O_F$ induces isomorphisms $\Gamma_n \simeq \pi^n\O_F$ for $n \gg 0$, and endows $\Gamma_K$ with an $\O_F$-analytic structure as a $p$-adic Lie group. 

If $W$ is an $F$-linear Banach representation of $\Gamma_K = \Gal(K_\infty/K)$, and if $n \geq 1$, we say that $w \in W$ is $F$-analytic on $\Gamma_n = \Gal(K_\infty/K_n)$ if there exists a sequence $\{w_k\}_{k \geq 0}$ of elements of $W$ such that $\pi^{nk}w_k \rightarrow 0$ such that $g(w) = \sum_{k \geq 0}\log \chi_{\pi}(g)^kw_k$ for all $g \in \Gamma_n$. This means that $w$ is a $\Gamma_n$-analytic vector, with $\Gamma_n$ viewed as a $p$-adic Lie group defined over $\O_F$ instead of $\Zp$. 

If $W$ is a locally analytic representation of $\Gamma_K$, we can define operators $\nabla_\tau : W \to W$, for $\tau \in \Sigma_F:=\Emb(F,\Qpbar)$ in the following way, as in \cite[\S 2]{Ber14MultiLa}. 
\begin{defi}\label{op nabla tau}
Let $L$ be a field that contains $F^{\Gal}$. If $\tau \in \Sigma_F$, then we have the derivative in the direction $\tau$, which is an element $\nabla_\tau \in L \otimes_{\Qp} \mathrm{Lie}(\Gamma_F)$. The $L$-vector space $\Hom_{\Qp}(F,L)$ is generated by the elements of $\Sigma_F$. If $W$ is an $L$-linear Banach representation of $\Gamma_F$ and if $w$ is a $\Qp$-locally analytic element of $W$ and $g \in \Gamma_F$, then there exists elements $\{ \nabla_\tau \}_{\tau \in \Sigma_F}$ of $F^{\Gal} \otimes_{\Qp} \Lie(\Gamma_F)$ such that we can write 
$$ \log g (w) = \sum_{\tau \in \Sigma_F} \tau(\log \chi_{\pi}(g)) \cdot \nabla_\tau(w). $$
\end{defi}

In particular, there exist $m \gg 0$ and elements $\{w_\kbf \}_{\kbf \in \N^{\Sigma_F}}$ such that if $g \in \Gamma_m$, then $g (w) = \sum_{\kbf \in \N^{\Sigma_F}} \log \chi_{\pi}(g)^\kbf w_\kbf$, where $\log \chi_{\pi}(g)^\kbf = \prod_{\tau \in \Sigma_F} \tau \circ \log \chi_{\pi}(g)^{k_\tau}$. We have $\nabla_\tau(w) = w_{\mathbf{1}_\tau}$ where $\mathbf{1}_\tau$ is the $\Sigma_F$-tuple whose entries are $0$ except the $\tau$-th one which is $1$. If $\kbf \in\N^{\Sigma_K}$, and if we set  $\nabla^\kbf(w) =  \prod_{\tau \in \Sigma_F} \nabla_{\tau}^{k_\tau} (w)$, then $w_\kbf = \nabla^\kbf(w)/\kbf!$. 

\begin{rema}
\label{remark nabla Fla}
If $w$ is an $F$-analytic element of $W$, so that there exists a sequence $\{w_k\}_{k \geq 0}$ of elements of $W$ such that $\pi^{nk}w_k \rightarrow 0$ such that $g(w) = \sum_{k \geq 0}\log \chi_{\pi}(g)^kw_k$ for all $g \in \Gamma_n$, with $n \gg 0$, then $w_k = \nabla_\id^k(w)/k!$. 
\end{rema}

The standard notations for the set of $F$-analytic elements of $W$ is $W^{\Gamma_n-\an,F-\la}$, following the notations from \cite[\S 2]{Ber14MultiLa}. These are the $\Gamma_n$-analytic vectors when we treat $\Gamma_n$ as a $p$-adic Lie group over $\O_F$ instead of $\Zp$. Since in this article we almost always consider the former case and in order to improve the readability and reduce the need for additional notations, we will still denote by $W^{\Gamma_n-\an}$ the set of $\Gamma_n$ $F$-analytic vectors of $W$ in the Lubin-Tate setting. We also let $W^{\la} = \bigcup_{n \geq 1}W^{\Gamma_n-\an}$. The main advantage of this formalism is that the cyclotomic case is exactly the particular case in the Lubin-Tate setting where $F=\Qp$ and $\pi=p$, so that the statements made in the ``$F$-analytic Lubin-Tate setting'' also contain the statements regarding the cyclotomic case.

In the rare cases where we will consider $\Qp$-locally analytic vectors of $W$ in the Lubin-Tate setting, we'll write $W^{\Qp-\la}$ for the set of $\Qp$-locally analytic vectors of $W$. We also use the same formalism for pro-analytic vectors. 

\section{Sen theory by Berger-Colmez}

Recall that to a $p$-adic representation $V$ of $\G_K$, one can attach the $K_\infty$-vector space $\D_{\Sen}(V)$ which is the set of elements of $W=(\Cp \otimes_{\Qp}V)^{H_K}$ which belong to some finite dimensional $K$-vector subspace of $W$ which is stable by $\Gamma_K$. The $K_\infty$-vector space $D_{\Sen}$ comes equipped with an action of the Lie algebra of $\Gamma_K$ and admits a canonical generator $\nabla = \lim_{\gamma \to 1}\frac{\gamma-1}{\chi_\cycl(\gamma)-1}$ which is the operator of Sen, usually denoted by $\Theta_{\Sen}$ and whose eigenvalues are called the generalized Hodge-Tate weights of the representation $V$.

Colmez has constructed in \cite{colmez1994resultat} a ring $\B_\Sen$ as follows:

\begin{defi}
Let $u$ be a variable and $\B_{\Sen}^n=\Cp\crochu_n$ be the set of power series $\sum_{k \geq 0}a_ku^k$ with coefficients in $\Cp$ such that the series $\sum_{k \geq 0}(p^n)^ka_k$ converges in $\Cp$ and equip it with the natural topology and with an action of $\Gal(\overline{K}/K(\mu_{p^n}))$ by setting 
$$g(\sum_{k \geq 0}a_ku^k) = \sum_{k \geq 0}g(a_k)(u+\log\chi_{\cycl}(g))^k.$$
Note that this makes sense since $\log\chi_{\cycl}(g) \in p^n\Z_p$ if $g \in \Gal(\overline{K}/K(\mu_{p^n}))$. Let $\B_{\Sen} = \bigcup_{n \geq 0}\B_{\Sen}^n$, endowed with the inductive limit topology.

We let $\nabla_u$ denote the $\Cp$-linear operator on $\B_{\Sen}$ given by $\nabla_u = -\frac{d}{du}$. 
\end{defi}

We now recall the following properties (for more details, see \cite{colmez1994resultat} and \cite[\S 2.2]{Ber14SenLa}):
\begin{prop}
\label{prop Bsen}~
\begin{enumerate}
\item We have $(\B_\Sen^n)^{\G_{K_n}} = K_n$ ;
\item if $V$ is a $p$-adic representation of $\G_K$ and if $n$ is an integer, let $\D_{\Sen,n}'(V):=(\B_\Sen^n\otimes_{\Qp}V)^{\G_{K_n}}$ equipped with the operator $\nabla_u$ induced by the operator $\nabla_u$ on $\B_\Sen^n$ (meaning that $(\nabla_u)_{D_{\Sen,n}'(V)}$ acts by $\nabla_u \otimes 1$ on $\B_\Sen^n\otimes_{\Qp}V$) and let $\D_{\Sen}'(V):=\bigcup_{n \geq 0}\D_{\Sen,n}'(V)$. Every element $\delta$ of $\D_{\Sen,n}'(V)$ can be written as $\delta^{(0)}+\delta^{(1)}u+\cdots$ where the $\delta^{(i)}$ belong to $\Cp \otimes_{\Qp}V$. Then the map $\delta \mapsto \delta^{(0)}$ induces an isomorphism of $K_\infty$-vector spaces between $\D_{\Sen}'(V)$ and $\D_{\Sen}(V)$, and of $K_n$-vector spaces between $\D_{\Sen,n}'(V)$ and $\D_{\Sen,n}(V)$ for $n \gg 0$. Moreover, the image of $\nabla_u$ by this isomorphism is $\Theta_\Sen$.
\end{enumerate}
\end{prop}
\begin{proof}
Item (i) is \cite[Thm. 2 (i)]{colmez1994resultat}. For item (ii), see \cite[Thm. 2 (ii)]{colmez1994resultat} and \cite[Prop. 2.8]{Ber14SenLa}.
\end{proof}

When $K_\infty/K$ is any $p$-adic Lie extension with Galois group $\Gamma_K$ (such that $\dim \Gamma_K \geq 2$ or such that $K_\infty/K$ is almost totally ramified), Berger and Colmez offer to replace classical Sen theory with the theory of locally analytic vectors, by considering the locally analytic vectors of semilinear $\hat{K_\infty}$-representations of $\Gamma_K$:

\begin{theo}
If $W$ is a $\hat{K_\infty}$-semilinear representation of $\Gamma_K$, then the map
$$\hat{K_\infty}\otimes_{\hat{K_\infty}^\la} W^\la \ra W$$
is an isomorphism. Moreover, if $K_\infty/K$ is the cyclotomic extension of $K$, and if $W=(\Cp \otimes_{\Qp}V)^{H_K}$ then $W^{\Gamma_n-\an}=\D_{\Sen,n}(V)$.
\end{theo}
\begin{proof}
The main claim is theorem 3.4 of \cite{Ber14SenLa}, and the particular case for the cyclotomic extension follows from remark 3.3 of ibid.
\end{proof}

We also have in general a nice description of the structure of $\hat{K_\infty}^\la$: if $K_\infty/K$ is a $p$-adic Lie extension with Galois group $\Gamma_K=\Gal(K_\infty/K)$, then by \cite[Thm. 6.1 and Rem. 6.2 (ii)]{Ber14SenLa}, if $L$ is a subfield of $\Cp$, containing $K_\infty(\mu_{p^m})$ for $m \gg 0$, then $L\hat{\otimes}_{K_n}\hat{K_\infty}^{\Gamma_n-\an}$ is isomorphic to the set $L{\{\{X_1,\ldots,X_{d-1}\}\}}_n$ of power series with coefficients in $L$ and radius of convergence $\geq p^{-n}$, where $K_n = K_\infty \cap \hat{K_\infty}^{\Gamma_n-\an}$. 

Note that in the cyclotomic case, the map $\log\chi_{\cycl} : \Gamma_K \to \Z_p$ induces isomorphisms between $\cal{C}^{\an}(\Gamma_n,\Cp)$ and $\B_\Sen^n$, and by taking the inductive limit, between $\cal{C}^{\la}(\Gamma_K,\Cp)_1$, the stalk at the identity of the sheaf of locally analytic functions on $\Gamma_K$ with coefficients in $\Cp$ and $\B_\Sen$, so that this formalism generalizes the construction of $\B_\Sen$ of Colmez. More generally, we make the following definition: 

\begin{defi}
\label{defi Bcrochu}
For any topological $\Qp$-algebra $\B$ which is an LF or LB space, equipped with a continuous action of $\G_K$, we denote by $\cal{C}^{\la}(\Gamma_K,\B)_1$ the stalk at the identity of the sheaf of locally analytic functions on $\Gamma_K$ with coefficients in $\B$, which is the inductive limit of the rings $\cal{C}^{\an}(\Gamma_n,\B)$, endowed with the inductive limit topology.
\end{defi}

Using the rings $\cal{C}^{\an}(\Gamma_n,\Qp)$, we can recover Sen theory (and its generalization by Berger and Colmez):

\begin{prop}
Let $K_\infty/K$ be a $p$-adic Lie extension with Galois group $\Gamma_K=\Gal(K_\infty/K)$ and let $V$ be a $p$-adic representation of $\G_K$. Then we have 
$$((\Cp \otimes_{\Qp}V)^{H_K})^{\Gamma_K-\la} = \bigcup_{n \geq 1}(\cal{C}^{\an}(\Gamma_n,\Cp)\otimes_{\Qp}V)^{\G_{K_n}},$$
where $\G_{K_n}$ acts on $\cal{C}^{\an}(\Gamma_n,\Cp)\otimes_{\Qp}V \simeq \cal{C}^{\an}(\Gamma_n,V)$ through the $\Delta_{1,2}$ map defined in \S 2.
\end{prop}

Note that this proposition is the consequence of the following proposition which is itself an immediate consequence of the definition of locally analytic vectors, as stated in \cite[§3.3]{emerton2004locally}:

\begin{prop}
\label{prop Bcrochu iso locana}
Let $\B$ be a $\Qp$-algebra which is an LF or LB space, endowed with an action of $\G_K$, and let $V$ be a $\Qp$-representation of $\G_K$. Then for any $n \geq 0$, we have
$$(\cal{C}^{\an}(\Gamma_n,\Qp)\hat{\otimes}_{\Qp}\B)^{\G_{K_n}} \simeq (\B^{H_K})^{\Gamma_n-\an}$$
and
$$((\cal{C}^{\an}(\Gamma_n,\Qp)\hat{\otimes}_{\Qp}\B)\otimes_{\Qp}V)^{\G_{K_n}} \simeq (\B \otimes_{\Qp}V)^{H_K,\Gamma_n-\an}.$$

Moreover, we have 
$$\bigcup_{n \geq 1}(\cal{C}^{\an}(\Gamma_n,\Qp)\hat{\otimes}_{\Qp}\B)^{\G_{K_n}} \simeq (\B^{H_K})^{\Gamma_K-\la}$$
and
$$\bigcup_{n \geq 1}((\cal{C}^{\an}(\Gamma_n,\Qp)\hat{\otimes}_{\Qp}\B))\otimes_{\Qp}V)^{\G_{K_n}} \simeq (\B \otimes_{\Qp}V)^{H_K,\Gamma_K-\la}.$$
\end{prop}
\begin{proof}
For LB spaces, this is tautological, since $H_K$ acts trivially on $\cal{C}^{\an}(\Gamma_n,\Qp)$ and since the set of $\Gamma_n$-locally analytic vectors of $W:=(\B \otimes_{\Qp}V)^{H_K}$ is by definition the subset of elements of $\cal{C}^{\an}(\Gamma_n,W)=\cal{C}^{\an}(\Gamma_n,\Qp)\hat{\otimes}_{\Qp}W$ which are invariant by the action given by $\Delta_{1,2}$ following the notations of \S 2.

For LF spaces, the proof is almost the same because the set of $\Gamma_n$-analytic vectors of $W$ is still the subset of elements of $\cal{C}^{\an}(\Gamma_n,\Qp)\hat{\otimes}_{\Qp}W$ which are invariant by the action given by $\Delta_{1,2}$ by \cite[Coro. 3.4.5]{emerton2004locally}.

The last two isomorphisms follow by taking the inductive limit. 
\end{proof}

In particular, if $\B$ is a topological $\Qp$-algebra which is an LF or LB space, endowed with an action of $\G_K$ such that for $V$ a $p$-adic representation of $\G_K$, $((B \otimes_{\Qp}V)^{H_K})^{\Gamma_K-\la}$ is related to some module attached to $V$ which appears in $p$-adic Hodge theory (e.g. its $(\phi,\Gamma)$-modules), then $\cal{C}^{\an}(\Gamma_n,\Qp)\hat{\otimes}_{\Qp}\B$ can be thought of as a ring of periods that computes those modules.

\section{de Rham computations}

In this section we compute locally analytic vectors and pro-analytic vectors in $\Bdrplus$, both in the cyclotomic case and in the Lubin-Tate case, and we explain how to recover the module $\D_{\dif}^+(V)$ attached to a $p$-adic representation $V$ thanks to the use of the locally analytic vectors. The fact that locally analytic vectors are able to recover $\D_{\dif}^+(V)$ has already been proven in \cite[\S 6.1]{PoratFF} but here we will also use proposition \ref{prop Bcrochu iso locana} to produce a ring of periods which ``computes'' the functor $\D_{\dif}^+$. 

\subsection{Computations in $\Bdrplus$}

We let $\B_{\dR,K}^+$ and $\B_{\dR,K}$ denote respectively $\B_{\dR}^{H_K}$ and $\B_{\dR}^{H_K}$. Recall (cf. \cite[\S 2]{Ber02}) that there is a natural injective, $\G_K$-equivariant map $\Btrigplus \ra \Bdrplus$, which sends $t_\pi$ to a generator of $\ker(\theta)$ in $\Bdrplus$ and we still denote the image of $t_\pi$ through this map by $t_\pi$. For $\tau \in \Sigma_F$, we have a corresponding element $t_\tau \in \Btrigplus$ defined in \cite[\S 5]{Ber14MultiLa} such that $g(t_\tau) = \tau(\chi_\pi(g))\cdot t_\tau$, and we still denote the image of $t_\tau$ through the map $\Btrigplus \ra \Bdrplus$ by $t_\tau$. Note that $t_\tau \in (\B_{\dR,F^{\Gal}}^+)^\times$ if $\tau \neq \id$ (see for example item 2 of \cite[Prop. 3.4]{BMTensTriang}). We let $\partial_\id = \frac{1}{t_\pi}\nabla_\id$. 

\begin{lemm}\label{lemm partial stab bdrplus}
We have $\partial_\id((\B_{\dR,K}^+)^{\pa}) \subset (\B_{\dR,K}^+)^{\pa}$.
\end{lemm}
\begin{proof}
Let $x \in (\B_{\dR,K}^+)^{\pa}$. Then $\theta(x) \in \hat{K_\infty}^{\la}$. Since $\nabla_\id = 0$ on $\hat{K_\infty}^{\la}$, we get that $\nabla_\id \circ \theta (x) = 0 = \theta \circ \nabla_\id(x)$ so that $\nabla_\id(x) \in t_\pi\Bdrplus$. Therefore, $\partial_\id(x) \in \Bdrplus$. Since $t_\pi$ is a pro-analytic vector of $\B_{\dR,K}$ and since $x \in (\B_{\dR,K}^+)^{\pa}$, we obtain $\partial_\id(x) \in (\B_{\dR,K})^{\pa}$. In order to conclude, we need to prove that $(\B_{\dR,K}^+)^{\pa} = (\B_{\dR,K})^{\pa} \cap \B_{\dR,K}^+$. But this is straightforward, because if $x$ is a pro-analytic vector of $(\B_{\dR,K})^{\pa}$ which belongs to $\B_{\dR,K}^+$ then the $\frac{\nabla_{\id}^k(x)}{k!}$ belong to $\B_{\dR,K}^+$ and thus $x \in (\B_{\dR,K}^+)^{\pa}$ by remark \ref{remark nabla Fla}.
\end{proof}

\begin{lemm}
We have $(\B_{\dR,K})^{\pa} = K_\infty(\!(t_\pi)\!)$ and $(\B_{\dR,K}^+)^{\pa} = K_\infty[\![t_\pi]\!]$.
\end{lemm}
\begin{proof}
See \cite[Prop. 2.6]{Porat}.
\end{proof}

We let $(\B_{\dR,K}^+)^{\Sigma_0-\pa}$ denote the set of $\Qp$-pro-analytic vectors of $\B_{\dR,K}^+$ which are killed by $\nabla_\id$. 

\begin{prop}\label{prop struc Bdrpa}
We have $(\B_{\dR,K}^+)^{\pa} = \left\{ \sum_{k \geq 0}a_kt_\pi^k, a_k \in (\B_{\dR,K}^+)^{\Sigma_0-\pa}\right\}$. 
\end{prop}
\begin{proof}
Let $x \in (\B_{\dR,K}^+)^{\pa}$. For $i \geq 0$, we let $x_i = \frac{1}{i!}\sum_{k \geq 0}(-1)^k\frac{\partial_\id^{i+k}(x)}{k!}t_\pi^k$. By lemma \ref{lemm partial stab bdrplus}, we have that for any $i,k \geq 0$, $\partial_\id^{i+k}(x)$ belongs to $(\B_{\dR,K}^+)^{\pa}$ so that the sum $\frac{1}{i!}\sum_{k \geq 0}(-1)^k\frac{\partial_\id^{i+k}(x)}{k!}t_\pi^k$ converges in $(\B_{\dR,K}^+)^{\pa}$ to an element $x_i$ such that $\partial_\id(x_i)=0$.

The sum $\sum_{i \geq 0}x_it_\pi^i$ converges in $(\B_{\dR,K}^+)^{\pa}$ and a simple computation shows that $x=\sum_{i \geq 0}x_it_\pi^i$.

Conversely, it is easy to check that if $(a_k)_{k \geq 0}$ is a sequence of elements of $(\B_{\dR,K}^+)^{\Sigma_0-\pa}$, the sum $\sum_{k \geq 0}a_kt_\pi^k$ converges to an element of $(\B_{\dR,K}^+)^{\pa}$.
\end{proof}

\begin{lemm}\label{lemm div tp implies 0}
Let $x \in (\B_{\dR,K}^+)^{\Sigma_0-\pa}$ such that $t_\pi | x$ in $\Bdrplus$. Then $x=0$.
\end{lemm}
\begin{proof}
Let $x \in (\B_{\dR,F}^+)^{\Sigma_0-\pa}$ such that $t_\pi | x$, and assume that $x \neq 0$. We can therefore write $x = t_\pi^k\alpha$ with $k \geq 1$, $\alpha \in \Bdrplus$ and $t_\pi$ does not divide $\alpha$ in $\Bdrplus$. Moreover, since $t_\pi$ is pro-analytic for the action of $\Gamma_K$, we get that $\alpha$ is pro-analytic for the action of $\Gamma_K$. 

By proposition \ref{prop struc Bdrpa}, we can write $\alpha = \sum_{j \geq 0}a_jt_\pi^j$ where the $a_j$ are elements of $(\B_{\dR,K}^+)^{\pa}$ killed by $\nabla_\id$. The fact that $x$ is killed by $\nabla_\id$ translates into
$$\sum_{j \geq 0}(k+j)a_jt_\pi^{k+j} = 0.$$
Applying $\partial_\id^k$ to this equality and reducing mod $t_\pi$, we obtain that $a_0 = 0 \mod t_\pi$ and thus $t_\pi | \alpha$, which is not possible.
\end{proof}

\begin{coro}\label{coro inj map}
For any $N \geq 1$, the map $\theta_N: (\B_{\dR,K}^+)^{\Sigma_0-\pa} \ra (\B_{\dR,K}^+/t_\pi^N\B_{\dR,K}^+)^{\Sigma_0-\pa}$ is injective.
\end{coro}

Note that $(\B_{\dR,K}^+/t_\pi\B_{\dR,K}^+)^{\Sigma_0-\pa} = \hat{K_\infty}^{\Sigma_0-\la}$ and that $\hat{K_\infty}^{\Sigma_0-\la} = \hat{K_\infty}^{\Qp-\la}$ by \cite[Prop. 2.10]{Ber14MultiLa}. Note that this also implies that for any $m \geq 0$, the natural map $\theta: (\B_{\dR,K}^+)^{\Qp-\Gamma_m-\an,\Sigma_0-\pa} \ra \hat{K_\infty}^{\Qp-\Gamma_m-\an}$  is injective. By the same argument as in the proof of the surjectivity in \cite[Thm. 6.2]{PoratFF}, the map $\theta : (\B_{\dR,K}^+)^{\Qp-\pa} \ra \hat{K_\infty}^{\Qp-\la}$ is surjective. In particular, using proposition \ref{prop struc Bdrpa}, we get the following ``description'' of $(\B_{\dR,K}^+)^{\Qp-\pa}$:

\begin{prop}
\label{prop struc complete Bdrpa}
The natural map $x \in (\B_{\dR,K}^+)^{\Qp-\pa} \mapsto \sum_{i \geq 0}\theta(x_i)t_\pi^i$, where $x_i = \frac{1}{i!}\sum_{k \geq 0}(-1)^k\frac{\partial_\id^{i+k}(x)}{k!}t_\pi^k$, induces a $\Gamma_K$-equivariant isomorphism from $(\B_{\dR,K}^+)^{\Qp-\pa}$ to $\hat{K_\infty}^{\Qp-\la}[\![t_\pi]\!]$.
\end{prop}
\begin{proof}
We already know from the above that the map $x \in (\B_{\dR,K}^+)^{\Qp-\pa} \mapsto \sum_{i \geq 0}\theta(x_i)t_\pi^i$ is injective. To prove that it is surjective, recall that the map $\theta: (\B_{\dR,K}^+)^{\Qp-\pa} \ra \hat{K_\infty}^{\Qp-\la}$ is surjective. If $y \in \hat{K_\infty}^{\Qp-\la}$, let $x \in (\B_{\dR,K}^+)^{\Qp-\pa}$ such that $\theta(x)=y$. One can write $x=\sum_{i \geq 0}x_it_\pi^i$ with $\partial_\id(x_i)=0$ for all $i$, and thus $x_0$ satisfies $\theta(x_0)=\theta(x)=y$ and $\partial(x_0)=0$, so that the map above is injective.
\end{proof}

\begin{rema}
We have $\B_{\dR,K}^+ \simeq \hat{K_\infty}[\![t_\pi]\!]$ noncanonically but this isomorphism is not $\Gamma_K$-equivariant. However, taking only the $\Qp$-pro-analytic vectors on both sides gives us a canonical isomorphism which is $\Gamma_K$-equivariant.
\end{rema}

\subsection{The modules $\D_\dif^+(V)$}

When $K_\infty/K$ is the cyclotomic extension of $K$, Fontaine has proven in \cite{fontaine2004arithmetique} that the set of sub-$K_\infty[\![t]\!]$-modules free of finite type of $(\Bdrplus\otimes_{\Qp}V)^{H_K}$ and stable by the action of $\Gamma_K$ admits a maximal element, usually denoted by $\D_{\dif}^+(V)$, and which is such that $\Bdrplus \otimes_{K_\infty[\![t]\!]}\D_{\dif}^+(V) = \Bdrplus \otimes_{\Qp}V$. 

If $\gamma \in \Gamma_K$ is close enough to $1$, then the power series defining $\log(\gamma)$ converges as a power series of $\Qp$-linear operators of $\D_{\dif}^+(V)$, and the operator $\nabla_V = \frac{\log(\gamma)}{\log(\chi_\cycl(\gamma))}$ does not depend on the choice of $\gamma$ and satisfies the Leibniz rule $\nabla_V(\lambda\cdot x)= \lambda \nabla_V(x)+\nabla(\lambda)x$ for all $\lambda \in K_\infty[\![t]\!]$ and $x \in \D_{\dif}^+(V)$. The map $\theta : \Bdrplus \to \Cp$ induces a surjective morphism of modules with connexions $(\D_\dif^+(V),\nabla_V) \ra (\D_{\Sen}(V),\Theta_V)$ (see for example \cite[\S 5.3]{Ber02}). 

The map $\iota_n : \Bt_{\rig}^{\dagger,r_n} \ra \Bdrplus$ sends $\B_K^{\dagger,r_n}$ into $K_n[\![t]\!] \subset \Bdrplus$ and $\D^{\dagger,r_n}(V)$ in a sub-$K_n[\![t]\!]$-module of $\D_{\dif}^+(V)$, and we let $\D_{\dif,n}^+(V):= K_n[\![t]\!] \otimes_{\iota_n(\B_K^{\dagger,r_n})}\iota_n(\D^{\dagger,r_n}(V))$. Proposition 5.7 of \cite{Ber02} shows that $\D_{\dif}^+(V) = K_\infty[\![t]\!]\otimes_{K_n[\![t]\!]}\D_{\dif,n}^+(V)$ for $n \gg 0$. 

The fact that one could retrieve the modules $\D_{\dif}^+(V)$ and $\D_{\dif,n}^+(V)$ using the theory of locally analytic vectors had already been noticed by Berger and Colmez \cite[Rem. 3.3]{Ber14SenLa} and proven by Porat in \cite[Prop. 3.3]{Porat} and \cite[Thm. 6.2]{PoratFF} but we now explain how this incorporates into the setting laid out at the end of \S 3. 

Note that $\Bdrplus$, endowed with its natural topology, is not a Banach ring but a Fréchet ring, and as Berger points out in \cite{Ber14MultiLa}, locally analytic vectors in the setting of Fréchet spaces usually have to be replaced with the weaker notion of pro-analytic vectors, because the resulting objects are too small in general. However, in the setting of $\Bdrplus$ and $\D_\dif^+(V)$, locally analytic vectors are actually sufficient to recover the theory.

\begin{lemm}
\label{lemma Fla and Fpa in Bdrplus}
We have $(\B_{\dR,K}^+)^{\Gamma_n-\an} = K_n[\![t_\pi]\!]$, $(\B_{\dR,K}^+)^{\la} = \bigcup_nK_n[\![t_\pi]\!]$.
\end{lemm}
\begin{proof}
The second equality follows directly from the first one. For the first equality, take $x \in (\B_{\dR,K}^+)^{\Gamma_n-\an}$. We have $\theta(x) \in \hat{K_\infty}^{\Gamma_n-\an} = K_n$ by \cite[Coro. 4.8]{Ber14SenLa}, so that we can write $x = x_0+t_{\pi}y$, with $x_0 \in K_n$ and $y \in \B_{\dR,K}^+$, and one checks that $y$ is $\Gamma_n$-analytic because $x$, $x_0$ and $t_\pi$ are. By induction, $x \in K_n[\![t_\pi]\!]$. Because $K_n \subset (\B_{\dR,K}^+)^{\Gamma_n-\an}$ and because $t_\pi$ is $\Gamma_0$-analytic, we have $K_n[\![t_\pi]\!] \subset (\B_{\dR,K}^+)^{\Gamma_n-\an}$, which finishes the proof.
\end{proof}

\begin{prop}
\label{D_difn=Gamma_nan}
For $n \gg 0$, we have $\D_{\dif,n}^+(V) = ((\Bdrplus \otimes_{\Qp}V)^{H_K})^{\Gamma_n-\an}$.
\end{prop}
\begin{proof}
Since $(\B_{\dR,K}^+)^{\Gamma_n-\an} = K_n[\![t]\!]$, it suffices to prove that the elements of $\iota_n(\D^{\dagger,r_n}(V))$ are $\Gamma_n$-analytic for $n \gg 0$. 

Let $m \geq 0$ be such that $\D^{\dagger,r_m}(V)$ has the right dimension, and let $e_1,\cdots,e_d$ be a basis of $\D^{\dagger,r_m}(V)$. We can see the elements of $\D^{\dagger,r_m}(V)$ as elements of $\D^{[r_m;r_m]}(V)$. By \S 2.1 of \cite{KR09}, these elements are $\Gamma_n$-analytic for $n \gg m$ big enough. A direct consequence of lemma 2.2 of \cite{Ber14SenLa} shows that if we let $u_i=\phi^{n-m}(e_i)$, $1 \leq i \leq d$, then the $u_i$ are $\Gamma_n$-analytic as elements of $\D^{[r_n;r_n]}(V)$, and we know that it is a basis of $\D^{\dagger,r_n}(V)$ (since $\phi^*(\D^\dagger(V)) \simeq \D^{\dagger}(V)$) and thus of $\D^{[r_n;r_n]}(V)$. Therefore, $(\iota_n(u_1),\cdots,\iota_n(u_d))$ generates $\D_{\dif,n}^+(V)$, and forms a basis of $\Gamma_n$-analytic elements of $\D_{\dif,n}^+(V)$.
\end{proof}

\begin{prop}
\label{Ddif=Ddrrpa}
We have $\D_{\dif}^+(V) = ((\Bdrplus \otimes_{\Qp}V)^{H_K})^{\pa}$.
\end{prop}
\begin{proof}
This is proposition 3.3 of \cite{Porat} and also follows from the previous proposition: we know that $\D_{\dif}^+(V) = K_\infty[\![t]\!]\otimes_{K_n[\![t]\!]}\D_{\dif,n}^+(V)$. Rewriting this using lemma \ref{lemma Fla and Fpa in Bdrplus} and proposition \ref{D_difn=Gamma_nan}, we get:
$$\D_{\dif}^+(V) = (\B_{\dR,K}^+)^{\pa}\otimes_{(\B_{\dR,K}^+)^{\Gamma_n-\an}}((\Bdrplus \otimes_{\Qp}V)^{H_K})^{\Gamma_n-\an}$$
so that $\D_{\dif}^+(V) = ((\Bdrplus \otimes_{\Qp}V)^{H_K})^{\pa}$ by proposition \ref{lainla and painpa}. 
\end{proof}

We can use this result to generalize the theory to the Lubin-Tate case, as it has been done in \cite[\S 3]{Porat}: we define $\D_{\dif}^+(V)$ by the formula $\D_{\dif}^+(V) := ((\Bdrplus \otimes_{\Qp}V)^{H_K})^{\pa}$, so that our object is indeed the classical $\D_{\dif}^+$ when $F=\Qp$ by proposition \ref{Ddif=Ddrrpa}, and if $V$ is an $F$-representation of $\G_K$ then $\D_{\dif}^+(V)$ is a free $K_\infty[\![t_\pi]\!]$-module of rank $\dim_F(V)$, and the natural map $\Bdrplus \otimes_{K_\infty[\![t_\pi]\!]}\D_{\dif}^+(V) \ra \B_{\dR}^+\otimes_{F}V$ is an isomorphism.

Note that when $V$ is an $F$-analytic representation of $\G_K$, one can define a module $\D_{\dif,n}^+(V)$ in the $F$-analytic Lubin-Tate case using the same arguments as given in the proof of proposition \ref{D_difn=Gamma_nan}: we can define $\D_{\dif,n}^+(V)$ by $\D_{\dif,n}^+(V) := ((\Bdrplus \otimes_{\Qp}V)^{H_K})^{\Gamma_n-\an}$ for $n \gg 0$, and we have $\D_{\dif}^+(V)=K_\infty[\![t_\pi]\!]\otimes_{K_n[\![t_\pi]\!]}\D_{\dif,n}^+(V)$ for $n \gg 0$. 

In particular, the rings $\cal{C}^{\an}(\Gamma_n,\Bdrplus)$ and $\cal{C}^{\la}(\Gamma_K,\Bdrplus)_1$ allow us to compute the modules $\D_{\dif}^+(V)$ in the spirit of Fontaine's strategy. Moreover, this shows that every $p$-adic representation of $\G_K$ is ``$\cal{C}^{\la}(\Gamma_K,\Bdrplus)_1$-admissible''.

In general, when $K_\infty/K$ is any $p$-adic Lie extension, one could define a module $\D_{\dif,K}^+(V)$ in the same manner, taking the pro-analytic vectors of $(\B_{\dR}^+)^{H_K}$ for the action of $\Gamma_K$. The fact that this module has the same dimension as $\dim_{\Qp}V$ follows from an unpublished result of Porat, and one could show in that case that the ring $\varprojlim\limits_k \varinjlim\limits_n(\cal{C}^{\an}(\Gamma_n,\Qp)\hat{\otimes}_{\Qp}\Bdrplus/t^k)$ computes the said module.

\section{$(\phi,\Gamma)$-modules}
Computations made by Berger in \cite[\S 4, \S 8]{Ber14MultiLa} show that classical cyclotomic $(\phi,\Gamma)$-modules over the Robba ring $\B_{\rig,K}^\dagger$ can be recovered by using pro-analytic vectors, and theorem 10.1 of ibid shows that this can be extended to $F$-analytic representations. If $V$ is an $F$-analytic representation of $\G_K$ then we can attach a $(\phi_q,\Gamma_K)$-module which we will denote $\D_{\rig}^\dagger(V)$.

Given an $F$-analytic $E$-representation $V$ of $\G_K$ , we let $\tilde{\D}_{\rig,K}^{\dagger,r}(V)=(\Bt_{\rig,K}^{\dagger,r}\otimes_{E}V)^{H_K}$. For $r > 0$ and $n \geq 1$ we let $\B_{\rig,K,n}^{\dagger,r} = \phi_q^{-n}(\B_{\rig,K,n}^{\dagger,q^nr})$, and we let $\B_{\rig,K,\infty}^{\dagger,r} = \bigcup_{n \geq 1}\B_{\rig,K,\infty}^{\dagger,r}$ in $\Bt_{\rig,K}^{\dagger,r}$, and $\B_{\rig,K,n}^{\dagger}=\bigcup_{r > 0}\B_{\rig,K,n}^{\dagger,r}$, $\B_{\rig,K,\infty}^{\dagger}=\bigcup_{r > 0}\B_{\rig,K,\infty}^{\dagger,r}$.

\begin{prop}
\label{prop phi gamma from locana berger}
We have
\begin{enumerate}
\item  $(\Bt_{\rig,K}^{\dagger,r_k})^{\pa} = \B_{\rig,K,\infty}^{\dagger,r_k}$;
\item $\tilde{\D}_{\rig,K}^{\dagger}(V)^\pa = \B_{\rig,K,\infty}^{\dagger}\otimes_{\B_{\rig,K}^{\dagger}}\D_{\rig,K}^{\dagger}(V)$;
\item if $\D$ is a $(\phi_q,\Gamma_K)$-module over $\B_{\rig,K}^{\dagger}$ such that $\tilde{\D}_{\rig,K}^{\dagger}(V)^\pa = \B_{\rig,K,\infty}^{\dagger}\otimes_{\B_{\rig,K}^{\dagger}}\D$ then $\D=\D_{\rig,K}^{\dagger,r}(V)$.
\end{enumerate}
\end{prop}
\begin{proof}
The first item is item 3 of theorem 4.4 of \cite{Ber14MultiLa}. The second item is item 2 of theorem 8.1 of \cite{Ber14MultiLa} in the cyclotomic setting, or follows from the proof of theorem 10.1 of ibid. in the Lubin-Tate setting. For the last item, let $M$ denote the base change matrix and $P_1,P_2$ denote the matrices of $\phi$ on $\D$, $\D_{\rig,K}^{\dagger,r}(V)$ respectively. There exists $n \gg 0$ such that $M \in \GL_d(\B_{\rig,K,n}^\dagger)$, and the equation $M = P_2^{-1}\phi_q(M)P_1$ implies that $M \in \GL_d(\B_{\rig,K}^\dagger)$.
\end{proof}

In particular, taking the pro-analytic vectors of $\tilde{\D}_{\rig,K}^{\dagger,r}(V)$ allows us to recover the $(\phi_q,\Gamma_K)$-module $\D_{\rig,K}^{\dagger,r}(V)$, either in the cyclotomic setting or in the $F$-analytic Lubin-Tate setting. 

As in the constructions for $\Bdr$ and $\Cp$, the rings $\cal{C}^{\la}(\Gamma_n,\B)_1$, for $\B$ an LB or LF space, are not endowed with an action of $\Gamma_K$ but only with an action of its Lie algebra, so that if $V$ is a $p$-adic representation of $\G_K$, the module $(\cal{C}^{\la}(\Gamma_n,\B)_1\otimes V)^{H_K}$ is only endowed with an operator $\nabla$ coming from the infinitesimal action of $\Gamma_K$, and from the operator $\nabla_\id$, which we still denote by $\nabla$, in the Lubin-Tate setting. In particular, the constructions laid out in this subsection can only allow us to recover the $(\phi,\nabla)$-module (or $F$-analytic $(\phi_q,\nabla)$-module) attached to a representation $V$. 

\begin{prop}
Let $V$ be an $F$-analytic representation of $\G_K$ and let $r > 0$. The collection $(\bigcup_n(\cal{C}^{\an}(\Gamma_n,\Bt^I)\otimes_{\Qp}V)^{\G_{K_n}})_{\min(I) \geq r}$ equipped with natural transition maps $\bigcup_n(\cal{C}^{\an}(\Gamma_n,\Bt^I)\otimes_{\Qp}V)^{\G_{K_n}} \ra \bigcup_n(\cal{C}^{\an}(\Gamma_n,\Bt^J) \otimes_{\Qp}V)^{\G_{K_n}}$ when $J \subset I$, and Frobenius maps $\phi_q : \bigcup_n(\cal{C}^{\an}(\Gamma_n,\Bt^I)\otimes_{\Qp}V)^{\G_{K_n}} \ra \bigcup_n(\cal{C}^{\an}(\Gamma_n,\Bt^{qI})\otimes_{\Qp}V)^{\G_{K_n}}$ defines a $(\phi_q,\nabla)$-module $\tilde{\D}$ over $\varprojlim\limits_I(\cal{C}^{\an}(\Gamma_n,\Bt^I))^{\G_{K_n}} \simeq (\Bt_{\rig,K}^{\dagger,r})^{\pa}$, and we have $\tilde{\D} \simeq \tilde{\D}_{\rig}^{\dagger,r}(V)^{\pa}$ as $(\phi_q,\nabla)$-modules.

Moreover, there exists a $(\phi_q,\nabla)$-module $\D$ on $\B_{\rig,K}^\dagger$ inside $\tilde{\D}$ such that $\tilde{\D} = (\Bt_{\rig,K}^{\dagger})^{\pa}\otimes_{\B_{\rig,K}^\dagger}\D$, and if $\D'$ is a $(\phi_q,\nabla)$-module on $\B_{\rig,K}^{\dagger}$ such that $\tilde{\D}_{\rig,K}^{\dagger}(V)^\pa = \B_{\rig,K,\infty}^{\dagger}\otimes_{\B_{\rig,K}^{\dagger}}\D'$ then $\D=\D'$. In particular, $\D \simeq \D_{\rig,K}^\dagger(V)$.
\end{prop}
\begin{proof}
The first part of the proposition follows from the definition of pro-analytic vectors, proposition \ref{prop Bcrochu iso locana} and proposition \ref{prop phi gamma from locana berger}.

Since $(\Bt_{\rig,K}^{\dagger,r})^{\pa} = \B_{\rig,K,\infty}^{\dagger,r}$, there exist elements $v_1,\cdots,v_d$ of $\tilde{\D}$ and $n \gg 0$ such that $\D:=\oplus_{i=1}^d\B_{\rig,K}^{\dagger,p^nr}\cdot\phi^n(v_i)$ generates $\tilde{\D}$. The unicity of $\D$ follows from the same argument as in the proof of the last item of proposition \ref{prop phi gamma from locana berger}.
\end{proof}

\section{Generalization to other $p$-adic Lie extensions}
\subsection{General results when $K_\infty$ contains a cyclotomic extension}
In what follows, $K$ is a finite extension of $\Qp$, $K_\infty/K$ is an infinite Galois $p$-adic Lie extension, with $\dim \Gamma_K \geq 2$, and such that $K_\infty$ contains a cyclotomic extension, in the sense that there exists an unramified character $\eta : \G_K \ra \Z_p^\times$ such that $K_\infty \cap \Qpbar^{\eta \chi_{\cycl}}$ is infinitely ramified. We let $K_\infty^\eta$ denote the extension $K_\infty \cap \Qpbar^{\eta \chi_{\cycl}}$. 

Recall that $K_\infty^\eta/K$ is the extension of $K$ attached to $\eta\chi_{\mathrm{cycl}}$. Let $\Gamma_{K,\eta} = \Gal(K_\infty^\eta/K)$ and $H_{K,\eta}=\Gal(\Qpbar/K_\infty^\eta)$. Let $\B_{K,\eta}^\dagger$, $\B_{K,\eta}^I$ and $\B_{\mathrm{rig},K,\eta}^\dagger$ be as in \cite[\S 8]{Ber14MultiLa}. By the same arguments as in \cite[\S 8]{Ber14MultiLa}, there is an equivalence of categories between étale $(\phi,\Gamma_{K,\eta})$-modules over $E \otimes_{\Qp}\B_{\mathrm{rig},K,\eta}^\dagger$ (it is also true over $E \otimes_{\Qp}\B_{K,\eta}^\dagger$) and $E$-representations of $\G_K$. 

If $V$ is a $p$-adic representation of $\G_K$, we let $\D_\eta^\dagger(V):=\bigcup_{r \gg 0}\D_\eta^{\dagger,r}(V)$, where $\D_\eta^{\dagger,r}(V):=(\B_\eta^{\dagger,r}\otimes_{\Qp}V)^{H_{K,\eta}}$. Let $\D_{\eta}^{[r;s]}$ and $\D_{\rig,\eta}^{\dagger,r}(V)$ denote the various completions of $\D_\eta^{\dagger,r}(V)$. We let $\tilde{\D}_\eta^{[r;s]}(V)=(\Bt^{[r;s]}\otimes_{\Qp}V)^{H_{K,\eta}}$ and $\tilde{\D}_{\rig,\eta}^{\dagger,r}(V)=(\Bt_{\rig}^{\dagger,r}\otimes_{\Qp}V)^{H_{K,\eta}}$. By the variant of the Cherbonnier-Colmez theorem for twisted cyclotomic extensions, we have that $\tilde{\D}_\eta^{[r;s]}(V)=\Bt_{K,\eta}^{[r;s]}\otimes_{\B_{K,\eta}^{[r;s]}}\D_{\eta}^{[r;s]}$ and $\tilde{\D}_{\rig,\eta}^{\dagger,r}(V)=\Bt_{\rig,K,\eta}^{\dagger,r}\otimes_{\B_{\rig,K,\eta}^{\dagger,r}}\D_{\rig,K,\eta}^{\dagger,r}(V)$.

\begin{lemm}
\label{lemma phigammamod gives locana}
Let $r \geq 0$ be such that $\D^{\dagger,r}_\eta(V)$ is free of rank $\dim_{\Qp}(V)$ as a $\B_{\rig,K,\eta}^{\dagger,r}$-module, and let $s \geq r$. Then the elements of $\D_{\eta}^{[r;s]}(V)$ are locally analytic for the action of $\Gal(K_\infty^{\eta}/K)$.
\end{lemm}
\begin{proof}
See the proof of \cite[Thm. 8.1]{Ber14MultiLa}.
\end{proof}

\begin{coro}
\label{coro construire vla}
Let $V$ be a $p$-adic representation of $\G_K$ which factors through $\Gamma_K$, and let $r > 0$ be such that $\D^{\dagger,r}_\eta(V)$ is free of rank $\dim_{\Qp}(V)$ as a $\B_{\rig,K,\eta}^{\dagger,r}$-module, so that $\B_\eta^{\dagger,r}\otimes_{\Qp}V \simeq \B_\eta^{\dagger,r}\otimes_{\B_{K,\eta}^{\dagger,r}}\D_{\eta}^{\dagger,r}(V)$. Then the coefficients in $\B_\eta^{\dagger,r}$ of a base change matrix between $V$ and $\D_{\eta}^{\dagger,r}(V)$ belong to $(\B_{K_{\eta}}^{[r;s]})^{\la}$ for any $s \geq r$. 
\end{coro}
\begin{proof}
Let $V$ be such a $p$-adic representation of $\G_K$. Since $V$ factors through $\Gamma_K$ and by Cartan's theorem (see for example \cite[Prop. 3.6.10]{emerton2004locally}), the elements of $V=V^{H_K}$ are locally analytic vectors for the action of $\Gamma_K$. Now, we have 
$$\tilde{\D}_{\eta}^{[r;s]}(V)^{\la}=(\Bt_{K,\eta}^{[r;s]}\otimes_{\Qp}V)^{\la}$$
since $V$ factors through $\Gamma_K$, and thus
\begin{equation}\label{equationla1}
\tilde{\D}_{\eta}^{[r;s]}(V)^{\la}=(\Bt_{K,\eta}^{[r;s]})^{\la}\otimes_{\Qp}V
\end{equation}
by proposition \ref{lainla and painpa}.

Since $K_\infty$ contains $K_\infty^{\eta}$, the elements of $\D_{\eta}^{[r;s]}(V)$ are locally analytic for the action of $\Gamma_K$, as they are locally analytic for the action of $\Gamma_{K,\eta}$ by \ref{lemma phigammamod gives locana} and invariant by the action of $\Gal(K_\infty/K_{\infty}^\eta)$ (which is just an other way of saying that the action of $\Gamma_K$ on $\D_{\eta}^{[r;s]}(V)$ factor through a locally analytic action of $\Gamma_K^{\eta}$).

By proposition \ref{lainla and painpa}, this implies that
\begin{equation}\label{equationla2}
\tilde{\D}_{K}^{[r;s]}(V)^{\la}=(\Bt_{K}^{[r;s]})^{\la}\otimes_{\Bt_{K,\eta}^{[r;s]}}\D_{\eta}^{[r;s]}.
\end{equation}
Putting equations \ref{equationla1} and \ref{equationla2} together, we obtain that
$$(\Bt_{K}^{[r;s]})^{\la}\otimes_{\B_{K,\eta}^{[r;s]}}\D_{\eta}^{[r;s]}=(\Bt_K^{[r;s]})^{\la}\otimes_{\Qp}V.$$
In particular, this implies that the coefficients of the base change matrix in $\GL_d(\B_\eta^{\dagger,r})$ belong to $(\Bt_{K,\eta}^{[r;s]})^{\la}$, and thus to $(\B_{K,\eta}^{[r;s]})^{\la}$. 
\end{proof}

This corollary will prove very useful in order to produce locally analytic vectors for $\Gamma_K$ in the rings $(\Bt_K^I)$. 

\begin{rema}
Note that the fact that $K_\infty$ contains $K_\infty^{\eta}$ is crucial for the proof of corollary \ref{coro construire vla} to work.
\end{rema}

\subsection{Higher locally analytic vectors}
Let $V$ be a Banach representation of a $p$-adic Lie group $G$, and assume that $G$ is small in the sense of \cite[\S 2.1]{PoratFF}, so that the set of $G$-analytic vectors of $V$ is well defined. 

The functor $V \mapsto V^{G-\an}$ is left exact, and following \S 2.2 of \cite{pan2022locally}, \cite{RJRC21} and \cite[\S 2.3]{PoratFF}, we define right derived functors for $i\geq0$:
\[
\mathrm{R}_{G-\an}^{i}(V)=H^{i}(G,V\widehat{\otimes}_{\Q_{p}}\mathcal{C}^{\an}(G,\Q_{p})),
\]
where we consider continuous cohomology on the right hand side.

If $G$ is a compact $p$-adic Lie group (without the smallness assumption) with subgroups $\left\{ G_{n}\right\} _{n\geq1}$
as in $\S 2$, taking the colimit over
$n$, there are right derived functors for $V\mapsto V^{G-\la}$ given
by
\[
\mathrm{R}_{G-\la}^{i}(V)=\varinjlim_{n}\mathrm{R}_{G_{n}-\an}^{i}(V)=\varinjlim_{n}H^{i}(G_{n},V\widehat{\otimes}_{\Q_{p}}\mathcal{C}^{\an}(G_{n},\Q_{p})).
\]

Following \cite[\S 2.3]{PoratFF}, we call these groups the higher locally analytic vectors of $V$.

Note that if
\[
0\rightarrow V\rightarrow W\rightarrow X\rightarrow0
\]
is a short exact sequence of $G$-Banach spaces, we then have a long exact sequence
\[
0\rightarrow V^{\la}\rightarrow W^{\la}\rightarrow X^{\la}\rightarrow\mathrm{R}_{G-\la}^{1}\left(V\right)\rightarrow\mathrm{R}_{G-\la}^{1}\left(W\right)\rightarrow\mathrm{R}_{G-\la}^{1}\left(X\right)\rightarrow \ldots
\]

The fact that the functor $V \mapsto V^\la$ is exact is thus equivalent to the vanishing of the higher locally analytic vectors $R^i_{G-\la}(V)$ for $i \geq 1$. 

We now prove several results and recall some results from \cite{PoratFF} regarding locally analytic vectors attached to $p$-adic Lie extensions in the rings $\Bt^I$. We let $K_\infty/K$ be a general $p$-adic Lie extension with Galois group $\Gamma_K$, and we let $H_K = \Gal(\Qpbar/K_\infty)$. We let $\Bt_K^I= (\Bt^I)^{H_K}$.

We recall the following result, which is item (ii) of corollary 5.6 of \cite{PoratFF}:

\begin{prop}
\label{prop R1vanish cyclo}
If $I$ is a compact subinterval of $[\frac{p}{p-1};+\infty[$, if $K_\infty$ contains a cyclotomic extension $K_\infty^\eta$, and if $M$ is a finite free $\Bt_K^I$-semilinear representation of $\Gamma_K$ then the higher analytic vectors $R^i_{\la}(M)$ are zero for $i \geq 1$.
\end{prop}

\begin{lemm}
\label{lemma ker theta}
Let $I =[r_k;r_\ell]$. For any $m \in [k;\ell]$ integer, the kernel of $\theta \circ \phi_q^{-m} : \At_K^I \ra \O_{\hat{K_\infty}}$ is a principal ideal.
\end{lemm}
\begin{proof}
The same proof as in \cite[Prop. 4.3.7]{Win83} (we are in the same setting since $K_\infty/K$ is strictly arithmetically profinite by the main theorem of \cite{sen1972ramification}) shows that the kernel of $\theta : \At_K^+ \ra \O_{\hat{K_\infty}}$ is principal, generated by some element $y$ such that $v_{\E}(\overline{y})=v_p(\pi_K)$, where $\pi_K$ is a uniformizer of $\O_K$. 

The same proof as for item 1. of lemma 3.2 of \cite{Ber14MultiLa} shows that $\frac{\phi_q^m(y)}{\pi_K}$ is then a generator of the kernel of $\theta \circ \phi_q^{-m} : \At_K^I \ra \O_{\hat{K_\infty}}$.
\end{proof}

\begin{coro}
\label{coro theta surj on locana}
Let $I =[r_k;r_\ell]$. If $K_\infty$ contains a cyclotomic extension $K_\infty^\eta$ then for any $m \in [k;\ell]$ integer, the map $\theta \circ \phi_q^{-m} : (\Bt_K^I)^{\la} \ra \hat{K_\infty}^{\la}$ is surjective.
\end{coro}
\begin{proof}
We have an exact sequence 
$$0 \ra \ker(\theta : \Bt_K^I \ra \hat{K_\infty})\ra \Bt_K^I \ra \hat{K_\infty} \ra 0$$
which gives rise to the exact sequence 
$$0 \ra (\ker(\theta : \Bt_K^I \ra \hat{K_\infty}))^{\la} \ra (\Bt_K^I)^{\la} \ra (\hat{K_\infty})^{\la} \ra R^1_{\la}(\ker(\theta : \Bt_K^I \ra \hat{K_\infty})) \ra \ldots$$
By lemma \ref{lemma ker theta}, the kernel of $\theta \circ \phi_q^{-m} : \At_K^I \ra \O_{\hat{K_\infty}}$ is a principal ideal so that it gives rise to a one dimensional $\Bt_K^I$-semilinear representation of $\Gamma_K$, and thus by proposition \ref{prop R1vanish cyclo} we have that $R^1_{\la}(\ker(\theta : \Bt_K^I \ra \hat{K_\infty}))=0$, so that we get the exact sequence 
$$0 \ra (\ker(\theta : \Bt_K^I \ra \hat{K_\infty}))^{\la} \ra (\Bt_K^I)^{\la} \ra (\hat{K_\infty})^{\la} \ra 0,$$
and thus the map $\theta \circ \phi_q^{-m} : (\Bt_K^I)^{\la} \ra \hat{K_\infty}^{\la}$ is surjective.
\end{proof}

Recall that by theorem 6.2 of \cite{Ber14SenLa}, $\hat{K_\infty}^{\la}$ is a ring of power series in $d-1$ variables. Since in the case we consider $K_\infty$ contains a cyclotomic extension, lemma \ref{lemma ker theta} shows that $\ker(\theta \circ \phi_q^{-m})$ is a principal ideal generated by a locally analytic vector of $\Bt_K^I$, and corollary \ref{coro theta surj on locana} shows that the map $\theta \circ \phi_q^{-m} : (\Bt_K^I)^{\la} \ra \hat{K_\infty}^{\la}$ is surjective. This (and the computations of section \ref{subsection particular case}) makes us think that the following conjecture should hold:

\begin{conj}
\label{conjecture structure good case}
If $K_\infty/K$ contains a cyclotomic extension, then for $n \gg 0$, there exist $d$ elements $x_{1,n},\ldots,x_{d,n}$ in $(\Bt_K^I)^{\Gamma_n-\an}$ such that 
$(\Bt_K^I)^{\Gamma_n-\an}$ is the set of power series $\sum_{\i = (i_1,\dots,i_d) \in \N^d}a_{\i}x_{j,n}^{i_j}$ in the variables $(x_{i,n})_{i \in \{1,\cdots,d\}}$ with coefficients in $K$ such that the series $\sum_{\i = (i_1,\dots,i_d) \in \N^d}a_{\i}x_{j,n}^{i_j}$ converge in $(\Bt_K^I)^{\Gamma_n-\an}$.
\end{conj}

\subsection{The anticyclotomic case}
\label{bad lie}
Berger and Colmez have proven in \cite{Ber14SenLa} that the theory of locally analytic vectors is the right object to consider in order to generalize classical Sen theory to arbitrary $p$-adic Lie extensions. With that in mind, and considering the results above that show that in the cyclotomic (and in the $F$-analytic Lubin-Tate) case one recovers classical $(\phi,\Gamma)$-modules theory, it seems reasonable to assume that the theory of locally analytic vectors is the right object to consider in order to generalize $(\phi,\Gamma)$-modules to arbitrary $p$-adic Lie extensions. 

It has already been noticed that, even in the Lubin-Tate case, ``one dimensional $(\phi_q,\Gamma_K)$-modules'' do not behave well \cite{FX13} and that the kind of objects one should consider are multivariable Lubin-Tate $(\phi_q,\Gamma_K)$-modules \cite{berger2012multivariable} which arise from locally analytic vectors  \cite{Ber14MultiLa}.

Therefore, in general, one should expect to use that theory for arbitrary $p$-adic Lie extensions to get a theory of $(\phi_q,\Gamma_K)$-modules over $(\Bt_{\rig,K}^{\dagger})^\pa$, and such that the functor $V \mapsto ((V \otimes_{\Qp}\Bt_{\rig})^{H_K})^{\pa}$ is a faithfully exact functor. We now give some insight as to why such a generalization does not seem to be true in general, using the anticyclotomic extension as a potential counterexample. 

Let $F/\Qp$ be the unramified extension of $\Qp$ of degree $2$. We take $\pi$ to be equal to $p$ in our Lubin-Tate setting. We let $\sigma$ denote the Frobenius on $F$. By \cite[\S 5]{Ber14MultiLa} the element $y_\sigma=\phi(u) \in \Atplus$ is such that $g(y_\sigma) = [\chi_p(g)]^\sigma(y_\sigma)$ for $g \in \G_F$. Since $[p](T) \in \Z_p[T]$, the series $\log_{\LT}(T)$ and $\exp_{\LT}(T)$ have all their coefficients in $\Qp$, so that $t_\sigma = \phi(t_p)= \log_{\LT}(y_\sigma)$. 

Let $F_{\mathrm{cycl}}=F(\mu_{p^\infty})$ denote the cyclotomic extension of $F$. We let $F_{\mathrm{ac}}$ be the anticyclotomic extension of $F$: it is the unique $\Z_p$ extension of $F$, Galois over $\Qp$, which is pro-dihedral: the Frobenius $\sigma$ of $\Gal(F/\Qp)$ acts on $\Gal(F_{\mathrm{ac}}/F)$ by inversion. It is linearly disjoint from $F_{\cycl}$ over $F$, and the compositum $F_{\mathrm{cycl}}\cdot F_{\mathrm{ac}}$ is equal to the Lubin-Tate extension $F_p^{\LT}$ attached to $p$ by local class field theory. The anticyclotomic extension is then the subfield of $F_\infty$ fixed under $G_\sigma:=\{g \in \Gal(F_{\infty}/F)~: \chi_p(g) = \sigma(\chi_p(g))\}$, and the cyclotomic extension of $F$ is the one fixed by $G:=\{g \in \Gal(F_{\infty}/F)~: \chi_p(g) = (\sigma(\chi_p(g)))^{-1}\}$. Note that $G \simeq \Gal(F_{ac}/F)$ as $F_\infty/F$ is abelian, and we still write $G$ for the Galois group of $F_{ac}/F$. We let $H_{F,ac}$ denote the group $\Gal(\Qpbar/F_{\mathrm{ac}})$, and if $B$ is a ring of periods we let $B_{F,\mathrm{ac}}$ denote $B^{H_{F,ac}}$. We write $t_1$ for $t_p$ and $t_2$ for $t_\sigma$.

\begin{prop}
\label{plgt BI Bdr cas anticyclo}
We have $(\B_{\dR,F,ac}^+)^{\pa} = F_{\mathrm{ac}}[\![\frac{t_1}{t_2}]\!]$.
\end{prop}
\begin{proof}
Clearly, if $z \in  F_{\mathrm{ac}}[\![\frac{t_1}{t_2}]\!]$, then the corresponding power series converges to an element of $\Bdrplus$ which is invariant by $H_F$ and pro-analytic for the action of $\Gamma_F$.

Now if $z \in (\B_{\dR,F,ac}^+)^{\pa}$, we have $\theta(z) \in \hat{F_{\mathrm{ac}}}^{\la} = F_{\mathrm{ac}}$ by \cite[Thm. 3.2]{Ber14SenLa}. We can therefore write $z= \theta(z)+t_1\cdot z'$ in $\Bdrplus$. Since  $\frac{t_1}{t_2}$ belongs to $(\B_{\dR,F,ac}^+)^{\pa} \cap \Fil^1\Bdr$, we can write $z=\theta(z)+\frac{t_1}{t_2}z_2$ with $z_2$ in  $(\B_{\dR,F,ac}^+)^{\pa}$. Now we can do the same thing for $z_2$, and doing this inductively gives us the result.
\end{proof}

If $I$ is big enough, so that the corresponding annulus contains a zero of $t_1$ and $t_2$, then the localization map at the zero of $t_1$ gives an embedding $(\Bt_{F,ac}^I)^{\la} \ra (\B_{\dR,F,ac}^+)^{\pa}=F_{\mathrm{ac}}[\![\frac{t_1}{t_2}]\!]$, and it seems difficult for an element in $\Bt_{F,ac}^I$ to have an ``essential singularity at a zero of $t_2$'', even if it's after a localization at a zero of $t_1$. Moreover, it is easy to prove that the image of $(\B_{\dR,F,ac}^+)^{\pa}$ in $\Bdrplus$ does not intersect $K_\infty[\frac{t_1}{t_2}] \setminus F$ as soon as $I$ is such that the corresponding annulus contains a zero of $t_2$. It seems therefore reasonable to expect that $(\B_{\dR,F,ac}^+)^{\pa}=F$, even though we do not have a proof of that statement, except for $I$ a subinterval of $[0,\infty[$ containing $0$, which we are now going to prove.

We let $P(T)=[p](T) = T^q+pT$. 

\begin{lemm}
We have $P^{\circ k}(\phi_q^{-k}(u^p)) \ra y_\sigma$ in $\Atplus$ when $k \ra +\infty$ for the $p$-adic topology.
\end{lemm}
\begin{proof}
Let $s_k := P^{\circ k}(\phi_q^{-k}(u^p)) \in \Atplus$, for all $k \geq 0$. We therefore have $s_0=u^p$, and $s_{k+1} = \phi_q^{-1}(P(s_k))$. 

Let us assume that $s_{k}-s_{k-1}$ belongs to $p^b\Atplus$, with $b \geq 1$.

Then we have $s_{k+1}= \phi_q^{-1}(P(s_k))$, and we can write
$$P(s_k) = P(s_{k-1})+\sum_{j = 1}^qP^{(j)}(s_{k-1})\frac{(s_k-s_{k-1})^j}{j!}.$$
Since $b \geq 1$ and since $P^{(j)}(T) \in p\mathcal{O}_F[\![T]\!]$, this means that $P(s_k) = P(s_{k-1})+(s_k-s_{k-1})h_k$, with $h_k \in p\Atplus$. But then this means that 
$$s_{k+1}-s_k = \phi_q^{-1}(s_k-s_{k-1})\phi_q^{-1}(h_k)$$
and thus $s_{k+1}-s_k \in p^{b+1}\Atplus$.

We already know that $s_1-s_0 \in p\Atplus$ (because $\overline{s_1}=\overline{s_0}=\overline{u}^p$ mod $p$) so that the sequence $(s_k)$ converges in $\Atplus$ to an element that we will denote by $s$.

Because both $\phi$ and $\theta$ are continuous for the $p$-adic topology, we know that $\theta \circ \phi_q^{-j}(s) = \lim\limits_{k \ra +\infty}P^{\circ k}(\theta \circ \phi_q^{-k}(u^p)) = P^{\circ k}(u_{j+k}^p)$. Therefore by lemma 5.3 of \cite{Ber14iterate}, $s$ is such that $\theta \circ \phi_q^{-j}(s) = \theta \circ \phi_q^{-j}(y_\sigma)$ for all $j \in \N$, so that $s = y_\sigma$.
\end{proof}

In particular, in lemma 5.3 of \cite{Ber14MultiLa}, we can actually take $x_n$ to be equal to $P^{\circ k}(\phi_q^{-k}(u^p))$ for some $k \gg 0$. In what follows, we let $h_{\ell}(u) := P^{\circ \ell}(\phi_q^{-\ell}(u^p))$.

Let $I=[0,r_0]$, let $m \geq 0$ and let $x \in (\Bt_F^I)^{\Gamma_m-\Qp-\an}$. Then there exists $n \geq 0$ such that $|\!|\partial_\sigma(x)|\!|_{\Gamma_m} \leq p^{nk}|\!|x|\!|_{\Gamma_m}$. Moreover, by \cite[Lemm. 2.4]{Ber14SenLa}, there exists $k_0 \geq m$ such that $|\!|x|\!|_{\Gamma_k}=|\!|x|\!|$ for all $k \geq k_0$. There exists $\ell \geq k_0$ such that $h_{\ell}(u)-y_\sigma \in p^n\At^I$, and there exists $m' \geq \ell$ such that $h_\ell(u), y_\sigma \in (\Bt_F^I)^{\Gamma_{m'}-\Qp-\an}$ and such that $|\!|h_{\ell}(u)-y_\sigma|\!|_{\Gamma_s} \leq p^{-n}$ for all $s \geq m'$. 

Then for $s \geq m'$, the series $x_i:=\frac{1}{i!}\sum_{k \geq 0}(-1)^k\partial_\sigma^{i+k}(x)\frac{(y_\sigma-h_{\ell}(u))^k}{k!}$ converges in $(\Bt_F^I)^{\Gamma_s-\Qp-\an}$, and we have 
$$x = \sum_{i \geq 0}x_i(y_\sigma-h_{\ell}(u))^i$$
in $(\Bt_F^I)^{\Gamma_s-\Qp-\an}$ (this is the same as the proof of theorem 5.4 of \cite{Ber14MultiLa}).

Now let 
$$X_{\ell,s}:= \left\{x \in (\Bt_F^I)^{\Gamma_{\ell}-\Qp-\an}, x = \sum_{i \geq 0}x_i(y_\sigma-h_{\ell}(u))^i \textrm{ and } x_i \in (\Bt_F^I)^{\Gamma_{s}-\an}\right\}.$$

The above shows that any $x \in (\Bt_F^I)^{\Qp-\la}$ belongs to some $X_{\ell,s}$, $s \geq \ell \geq 0$. 

We denote by $F[\![T_1,T_2]\!]$ the set of power series with coefficients in $F$ in two variables $T_1$ and $T_2$, endowed with an $F$-linear action of $\Gal(F_p^{\LT}/F)$ given by $g(T_1) = \chi_p(g)\cdot T_1$ and $g(T_2) = \sigma(\chi_p(g))\cdot T_2$. 

\begin{prop}
There is an injective, Galois-equivariant $F$-linear map $\iota_{\ell,s}: X_{\ell,s} \ra F[\![T_1,T_2]\!]$. 
\end{prop}
\begin{proof}
Let $x \in X_{\ell}$. We can write $x = \sum_{i \geq 0}x_i(y_\sigma-h_{\ell}(u))^i$, where $x_i \in (\Bt_F^I)^{\Gamma_s-\an}$. Note that $(\Bt_F^I)^{\Gamma_s-\an} \subset \B_{F,s}^I$ by \cite[Thm. 4.4]{Ber14MultiLa} so that we can write $x_i = f_i(\phi_q^{-s}(u))$, with $f_i \in F[\![u]\!]$. 

We can write $\phi_q^{s}(x) = \sum_{i \geq 0}f_i(u)(P^{\circ s}(y_\sigma)-P^{\circ s}(u^p))^i$, so that 
$$\phi_q^{s}(x)=\sum_{i \geq 0}f_i(u)\sum_{k = 0}^i\binom{i}{k}(P^{\circ s}(y_\sigma))^k(-P^{\circ s}(u^p))^{k-i}$$
and this is equal to (if everything converges)
$$\sum_{k \geq 0}(P^{\circ s}(y_\sigma))^k\sum_{j \geq 0}f_{j+k}(u)(-P^{\circ s}(u^p))^j.$$
Let $A_k:=\sum_{j \geq 0}f_{j+k}(u)(-P^{\circ s}(u^p))^j \in F[\![u]\!]$. This is a well defined element of $F[\![u]\!]$ since $P^{\circ s}(u^p) \in u\cdot F[\![u]\!]$ and since the $f_{j+k}(u)$ belong to $F[\![u]\!]$. 
Since $P^{\circ s}(y_\sigma) \in y_\sigma\cdot F[\![y_\sigma]\!]$ (because $s \geq \ell$), the sum $\sum_{k \geq 0}(P^{\circ s}(y_\sigma))^kA_k$ defines an element of $F[\![y_\sigma,u]\!]$. Now because $t_\sigma \in y_\sigma\cdot F[\![y_\sigma]\!]$ and $t_p \in F[\![u]\!]$, this can be rewritten as an element of $F[\![T_1,T_2]\!]$. It remains to check that the map we have just constructed is well defined relative to the Galois action, which is straightforward (because $\phi_q^{-s}$ is $\Gamma_K$-equivariant and then the rest is just rewriting power series in $F[\![Y_1,Y_2]\!]=F[\![T_1,T_2]\!]$). 

In order to see that the map that we obtain this way is injective, we can see that at every step the operations we make are injective. To see that we indeed have defined a map this way, we have to prove that the $x_i$ coming from $x$ are uniquely defined, which follows from the formula given in the proof of \cite[Thm. 5.4]{Ber14MultiLa}.
\end{proof}

\begin{coro}
\label{thm ext anticyclo bad}
We have $(\Bt_F^{[0,r_0]})^{\mathrm{ac},\la}=F$.
\end{coro}
\begin{proof}
Let $\nabla_1, \nabla_2$ denote respectively the maps $T_1\cdot\frac{d}{dT_1}$ and $T_2\cdot\frac{d}{dT_2}$ on $F[\![T_1,T_2]\!]$. It is clear from the definition of the Galois action of $\Gal(F_p^{\LT}/F)$ on $F[\![T_1,T_2]\!]$ that this action is locally analytic, and that the corresponding operators $\nabla_\id$ and $\nabla_\sigma$ coincide respectively with $\nabla_1$ and $\nabla_2$. 

By the previous proposition, it therefore suffices to prove that $F[\![T_1,T_2]\!]^{\nabla_1+\nabla_2=0}=K$. This is straightforward because 
$$(\nabla_1+\nabla_2)(\sum_{i,j}a_{ij}T_1^iT_2^j) = \sum_{i,j}(i+j)a_{ij}T_1^iT_2^j$$
which is equal to $0$ if and only if $a_{ij}=0$ for all $i,j \neq 0$.
\end{proof}

Those results highlight that in general, the rings $(\Bt_{K}^I)^{\la}$ could be really small, even if we restrict ourselves to the case of $p$-adic abelian extensions. However, if we assume that $K_\infty/K$ contains a cyclotomic extension, then most of those problems should disappear. Note that the case of the anticyclotomic extension is precisely a case where we removed the cyclotomic extension contained inside the Lubin-Tate extension. 

It would be interesting to know if ``containing a cyclotomic extension'' is the key component for the theory to behave properly:

\begin{question}
Are there Galois $p$-adic Lie extensions $K_\infty/K$ almost totally ramified, not containing any cyclotomic extension, such that for all compact subinterval $I$ of $[0;+\infty[$, $(\Bt_K^I)^{\la} \neq K$?
\end{question}

Note that, if we do not assume that $K_\infty/K$ is Galois but that its ramification is infinite and its Galois closure $L_\infty:=K_\infty^\Gal$ is such that $L_\infty/K$ is a $p$-adic Lie extension with Galois group $\Gamma_L:=\Gal(L_\infty/K)$, there is still a way to define locally analytic vectors attached to the extension $K_\infty/K$, in the following way: if $W$ is a $p$-adic Banach representation of $\Gamma_L$, we define the ''$K_\infty$-locally analytic vectors of $W$'' by $W^{K_\infty-\la}:=(W^{\Gamma_L-\la})^{\Gal(L_\infty/K_\infty)}$. Kummer extensions are particular cases of this setting, and in this case the theory does behave properly \cite{GP18}. It is therefore not clear what to expect if we generalize the theory to ``non Galois $p$-adic Lie extensions''.

\subsection{A particular case of the conjecture}
\label{subsection particular case}
We now explain how to prove conjecture \ref{conjecture structure good case} in a very particular case, which is already nontrivial and is a generalization of the Kummer case.

In this section, we assume that $K_\infty/K$ is a $p$-adic Lie extension which is a successive extension of $\Zp$-extensions: there exist $(K_{\infty,i})_{i \in \{0,\ldots,d\}}$ such that for all $i$, $K_{\infty,i}/K$ is Galois, $K_\infty=K_{\infty,d}$, $K_{\infty,0}=K$, and $\Gal(K_{\infty,i+1}/K_{\infty,i}) \simeq \Z_p$. We also assume that there exists $\eta : \G_K \to \Z_p^\times$ an unramified character such that $K_{\infty,1}=K_\infty^{\eta}$. In particular, this implies that $\Gamma_K$ is isomorphic to a semi-direct product $\Z_p \rtimes \cdots \rtimes \Zp$. We write $g \mapsto (c_d(g),\ldots,c_1(g))$ for the isomorphism $\Gamma_K \simeq \Z_p \rtimes \cdots \rtimes \Zp$, where if $1 \leq j \leq d$ and $g \in \Gamma_K$, $(c_j(g),\cdots,c_1(g))$ denotes the image of $g$ in the quotient $\Gal(K_{\infty,j}/K) \simeq \Z_p \rtimes \cdots \rtimes \Zp$. 

For any $i \in \{1,\cdots,d\}$, we let $g_i \in \Gal(K_{\infty}/K_{\infty,i-1})$ be such that $c_i(g_i)=1$, so that its image in the quotient $\Gal(K_{\infty,i}/K_{\infty,i-1})$ is a topological generator, and we let $\nabla_i \in \Lie(\Gamma_K)$ denote the operator corresponding to $\log g_i$. Note that since $g_i \in \Gal(K_{\infty}/K_{\infty,i-1})$, we have $c_j(g_i) = 0$ for $j < i$. Since it is clear that the $g_i$ generate $\Gamma_K$ topologically, the operators $\nabla_i$ define a basis of the Lie algebra of $\Gamma_K$. We also let $G_i = \Gal(K_{\infty,i}/K)$. 

\begin{lemm}\label{lemm tué par nabla stable}
If $x$ is a locally analytic vector of a $p$-adic Banach representation of $\Gamma_K$ such that there exists $j \geq 2$, such that for all $k \geq j$, $\nabla_k(x)=0$, then for all $\ell < j$ and for all $k \geq j$, $\nabla_k \circ \nabla_{\ell}(x) =\nabla_\ell \circ \nabla_k(x)=0$. 
\end{lemm}
\begin{proof}
Let $W$ be a $p$-adic Banach representation of $\Gamma_K$. Let $x$ be a locally analytic vector of $W$ which is killed by $\nabla_k$, for all $k \geq j$. By definition of the $\nabla_k$ operators, this implies that there exist $f_d,\ldots,f_j$ integers such that we have $g_k^{p^{f_k}}(x)=x$ for all $k \in \{j,\ldots,d\}$. Therefore, $x$ belongs to $W^{\Gal(K_\infty/M)}$ for some finite extension $M$ of $K_{\infty,j}$, which we can assume to be Galois over $K$. But then $g_\ell(x) \in W^{\Gal(K_\infty/M)}$ for all $\ell < j$, so that $\nabla_k \circ \nabla_\ell(x)=0$. 
\end{proof}

\begin{prop}
For any $i \in \{2,\cdots,d\}$, there exists $r_i \geq 0$ and $b_i \in \B_{K_{\infty,i}}^{\dagger,r_i}$ such that $(g_i-1)(b_i)=1$ and the image of $b_i$ in $ \Bt_{K_{\infty,i}}^I$, for $\min(I) \geq r_i$, is a locally analytic vector of $\Bt_{K_{\infty,i}}^I$ for the action of $G_i$.
\end{prop}
\begin{proof}
We only prove it for $i=d$, the proof for $i < d$ is the same replacing $\Gamma_K$ by $G_i$.

Let $V$ denote the $2$-dimensional $p$-adic representation of $\G_K$ given by
\begin{equation*}
g \mapsto 
\begin{pmatrix}
1 & c_d(g)  \\
0 & 1 
\end{pmatrix}
\end{equation*}

By the theorem of Cherbonnier-Colmez, the $(\phi,\Gamma)$-module attached to $V$ is overconvergent, so that it admits a basis on $(\B_K^{\eta})^{\dagger,r}$. If $(e_1,e_2)$ was the basis of $V$ giving rise to the matrix representation above, we see that a basis of the attached $(\phi,\Gamma)$-module on $(\B_K^{\eta})^{\dagger,r}$ is given by $(e_1 \otimes 1,e_2 \otimes 1 - e_1\otimes b)$ for some $b \in \B_K^{\dagger,r}$. The fact that this basis is invariant by the action of $\Gal(\Qpbar/K_{\infty,1})$ means that it also is invariant by the action of  $\Gal(\Qpbar/K_{\infty,d-1})$ and thus we get that $g_d(b)=b+c_d(g_b)=b+1$ by our choice of $g_d$. The fact that we can find such an element $b$ which is a locally analytic vector of $\Bt_{K_{\infty,i}}^I$ follows from corollary \ref{coro construire vla}.
\end{proof}

We let $r_b = \max(r_i)$ so that the $(b_i)$ can all be seen as elements of $(\Bt_K^{\dagger,r_b})$.

Recall that if $M_\infty^{\eta}$ is a finite extension of $K_\infty^{\eta}$ then there corresponds a finite unramified extension $\B_{M,\eta}^{\dagger}/\B_{K,\eta}^{\dagger}$ of degree $[M_\infty^{\eta}:K_\infty^{\eta}]$, and there exists $r(M) > 0$ and elements $x_1,\hdots,x_f$ in $\A_{M,\eta}^{\dagger,r(M)}$ such that $\A_{M,\eta}^{\dagger,s} = \oplus_{i=1}^f \A_{K,\eta}^{\dagger,s} \cdot x_i$ for all $s \geq r(M)$.

\begin{lemm}
\label{approx killed by nabla}
Let $M_\infty^{\eta} \subset K_\infty$ be a finite extension of $K_\infty^{\eta}$. If $r_\ell \geq r(M)$ then the $x_i$ defined above belong to $(\Bt_{\rig,K}^{\dagger,r_\ell})^{\pa}$ and are killed by $\nabla_i$ for all $i > 1$.
\end{lemm}
\begin{proof}
The fact that the $x_i$ are pro-analytic is a consequence of the proof of item $2$ of \cite[Thm. 4.4]{Ber14MultiLa}. The second part is straightforward as $M_\infty^{\eta}/K_\infty^{\eta}$ is finite. 
\end{proof}

If $M_\infty^{\eta}$ is a finite extension of $K_\infty^{\eta}$, and if $I$ is a compact subinterval of $[0;+\infty[$ such that $\min(I) \geq r(M)$,  we let $\A_{M,\eta}^I$ be the completion of $\A_{M,\eta}^{\dagger,r(M)}$ for $V(\cdot,I)$, and we let $\B_{M,\eta}^I = \A_{M,\eta}^I[1/p]$.

\begin{lemm}
\label{lemm approx}
If $x \in \A_K^{\dagger,r}$ and if $k,n \in \N$ then there exists $M_\infty^{\eta}$ a finite extension of $K_\infty^{\eta}$, $m \geq 0$ and $y \in \phi^{-m}(\A_{M,\eta}^{\dagger,p^mr})$ such that 
$x-y \in \pi^j\At^{\dagger,r}+u^k\Atplus$.
\end{lemm}
\begin{proof}
By reducing mod $\pi$, we obtain that $\overline{x} \in \E_K$. But $\E_K = \bigcup \E_M$ where $M$ goes through the set of finite extensions of $K_{\infty,\eta}$ contained in $K_\infty$. In particular, there exists a finite extension $M_{0}$ of $K_{\infty,\eta}$, contained in $K_\infty$, and $y_0 \in \A_{M_0}$ such that $x-y_0 \in p\A_K$, since $\A_{M_0,\eta} \subset \A_K$. Since $\frac{x-y_0}{p} \in \A_K$, the same arguments show that there exists a finite extension $M_1$ of $K_{\infty,\eta}$, contained in $K_\infty$, and $y_1 \in \A_{M_1,\eta}$ such that 
$\frac{x-y_0}{p}-y_1 \in \pi\A_K$, so that $x-y_0-py_1 \in \pi\A_K$, and we can without loss of generality assume that $M_0 \subset M_1$. By induction, we find $y_0,y_1,\cdots,y_n$ in $\A_{M_n,\eta}$, with $M_n$ finite extension of $K_{\infty,\eta}$ contained in $K_\infty$, such that $x-y_0-py_1-\cdots-p^ny_n \in \pi^{n+1}\A_K$. Let $z_n = y_0+\cdots+\pi^ny_n$. Let $\sum_{i \geq 0}p^i[x_i]$ be the way $x$ is written in $\At_K = W(\Et_K)$. Then $x^{(n)}:= \sum_{i=0}^np^i[x_i]$ is such that $x-x^{(n)} \in \pi^{n+1}\At_K$, and thus $x^{(n)}-z_n \in \pi^{n+1}\At_K$. In particular, since $z_n \in \At_{M_n,\eta}$ by construction, we deduce that the $x_i$ all belong to $\Et_{M_n,\eta}$ for $i \leq n$, and thus $x^{(n)} \in \At_{M_n,\eta}$. 

Since $x \in \At_K^{\dagger,r}$, we have in particular that $x^{(n)} \in \At_K^{\dagger,r}$. By corollary 8.11 of \cite{colmez2008espaces}, $\A_{M_n,\infty,\eta}^{\dagger,r}$ is dense in $\At_{M_n,\eta}^{\dagger,r}$ for the topology induced by $V(\cdot,r)$, so that we can find $y \in \A_{M_n,\infty,\eta}^{\dagger,r}$ such that $x^{(n)}-y \in \pi^n\At^{\dagger,r}+u^k\Atplus$. We thus have $x-y = (x-x^{(n)})+(x^{(n)}-y) \in \pi^n\At^{\dagger,r}+u^k\Atplus$.
\end{proof}

Lemma \ref{lemm approx} shows that for any $I=[r;s]$ with $r \geq r_b$, and any integer $n$ we can find elements $b_n^{\ell}$ such that $b_\ell-b_n^\ell \in p^n\At^I$ for all $\ell \in \{2,\ldots,d\}$, which by lemma \ref{approx killed by nabla} belong to $(\Bt_K^I)^{\la}$ and are killed by $\nabla_j$, for all $j \in \{2,\ldots,d\}$. Since they are locally analytic vectors, we let $m = m(n,I)$ be such that all these elements, along with the elements $b_\ell$, belong to $(\Bt_K^I)^{\Gamma_m-\an}$.

\begin{prop}
\label{proof theo induction}
Let $I = [r;s]$ with $r \geq r_b$ and let $m \geq m(n,I)$. Let $\ell \in \{2,\ldots,d\}$ and let $x \in (\Bt_K^I)^{m-\an}$ be such that for all $k > \ell$, $\nabla_k(x)=0$. Then there exist $(x_j)_{j \in \N} \in (\Bt_K^I)^{\Gamma_m-\an}$ such that $|\!|x_jp^{nj}|\!| \ra 0$, for all $k \geq \ell, \nabla_k(x_j)=0$ and $x=\sum_{j \geq 0}x_j(b_\ell-b_n^{\ell})^j$.
\end{prop}
\begin{proof}
Let $x \in (\Bt_K^I)^{\Gamma_m-\an}$. By \cite[Lemm. 2.6]{Ber14SenLa}, there exists $n \geq 1$ such that for all $j \geq 1$, $|\!|\nabla_\ell^j(x)|\!|_{\Gamma_m} \leq p^{nj}|\!|x|\!|$ for all $\ell \in \{2,\ldots,d\}$. Let $$x_j = \frac{1}{j!}\sum_{k \geq 0}(-1)^k\frac{(b_\ell-b_n^\ell)}{k!}\nabla_\ell^j(x).$$
Similarly to the proof of \cite[Thm. 5.4]{Ber14MultiLa}, the series converges in $(\Bt_K^I)^{\Gamma_m-\an}$ to an element $x_j$ such that $\nabla_\ell(x_j)=0$. Moreover, by construction of the $b_\ell$ and $b_n^\ell$, we have $\nabla_{k}(b_\ell-b_n^\ell)=0$ for all $k > \ell$, and thus using lemma \ref{lemm tué par nabla stable}, the $x_j$ are killed by $\nabla_k$, $k > \ell$. 
\end{proof}

\begin{prop}
\label{prop locana killed by all nabla is cyclo}
Let $I = [r,s]$ and let $m \geq m(I,n)$. Then there exists $M$ a finite extension of $K_{\infty}^\eta$, depending only on $m$, and $k \geq 0$ depending only on $s$, $m$ and $M$, such that 

$$(\Bt_K^I)^{\Gamma_m-\an, \nabla_d = \cdots = \nabla_2 = 0} \subset \phi^{-k}(\B^{p^kI}_{M,\eta}).$$
\end{prop}
\begin{proof}
Note that if $x \in (\Bt_K^I)^{\Gamma_m-\an}$ and is killed by $\nabla_i$ for some $i \in \{2,\cdots,d\}$, then this means that the orbit map $g \mapsto g(x)$, from the $p$-adic Lie group $G_i:=g_i^{\Zp}$ to $\Bt_K^I$, is an analytic function on $G_i \cap \Gamma_m$ which becomes constant on some compact open subgroup of $G_i$ since $\nabla_i(x) = 0$. It is therefore constant on $G_i \cap \Gamma_m$. It follows that if we let $M_\infty^\eta= (K_\infty)^{g_d^{p^m}=1,\cdots,g_2^{p^m}=1}$, then $x$ is invariant by $\Gal(K_\infty/M_\infty^\eta)$. Note that since $\Gamma_K$ is topologically generated by the $g_i$ for $i=1 \ldots d$, $M_\infty^\eta$ is a finite extension of $K_\infty^\eta$. 

Now we can conclude using the same argument as in the proof of theorem 4.2.9 of \cite{GP18}. We let $f = [M_\infty^\eta:K_\infty^\eta]$ and we let $r(M) > 0$ and $x_1,\cdots,x_f$ in $\A_{M,\eta}^{\dagger,r(M)}$ be such that $\A_{M,\eta}^{\dagger,s} = \oplus_{i=1}^f\A_{K,\eta}^{\dagger,s} \cdot x_i$ for all $s \geq r(M)$. Note that the proof of item 2 of \cite{Ber14MultiLa} shows that the $x_i$ are locally analytic for the action of $\Gal(K_\infty/K)$, so that there exists $k \geq m$ such that the $x_i$ are all $\Gamma_{m'}$-analytic. Then we have 

\begin{eqnarray*}
(\Bt_K^I)^{\Gamma_{k}-\an,\Gal(K_\infty/M_\infty^\eta)} &=& (\Bt_{M,\eta}^I)^{(\Gamma_{K,\eta})_{k}-\an} \\
 &=&     (\oplus_{i=1}^{f}\Bt_{K,\eta}^I \cdot x_i)^{(\Gamma_{K,\eta})_{k}-\an}\\
 &=&  \oplus_{i=1}^{f}(\Bt_{K,\eta}^I)^{(\Gamma_{K,\eta})_{k}-\an}\cdot x_i\\
 &\subset & \oplus_{i=1}^f\phi^{k}(\B_{K,\eta}^{p^kI})\cdot x_i
\end{eqnarray*}
where the second equality follows from proposition \ref{lainla and painpa}, and the last inclusion follows from the specialization of \cite[Thm.4.4]{Ber14MultiLa} to the twisted cyclotomic case.

and since $k \geq m$, we obtain that $(\Bt_K^I)^{\Gamma_{m}-\an,\Gal(K_\infty/M_\infty^\eta)=1} \subset \oplus_{i=1}^f\phi^{k}(\B_{K,\eta}^{p^kI})\cdot x_i \subset \phi^{k}(\B_{M,\eta}^{p^kI})$, which is what we wanted.
\end{proof}

We can now prove the conjecture:

\begin{theo}
\label{theo conjecture in d variables}
Let $K_\infty/K$ be a $p$-adic Lie extension of rank $d$ which is a successive extension of $\Zp$-extensions over a cyclotomic extension. Then for $n \gg 0$, there exist $d$ elements $x_{1,n},\ldots,x_{d,n}$ in $(\Bt_K^I)^{\Gamma_n-\an}$ such that 
$(\Bt_K^I)^{\Gamma_n-\an}$ is the set of power series $\sum_{\i = (i_1,\dots,i_d) \in \N^d}a_{\i}x_{j,n}^{i_j}$ in the variables $(x_{i,n})_{i \in \{1,\cdots,d\}}$ with coefficients in $K$ such that the series $\sum_{\i = (i_1,\dots,i_d) \in \N^d}a_{\i}x_{j,n}^{i_j}$ converge in $(\Bt_K^I)^{\Gamma_n-\an}$.
\end{theo}
\begin{proof}
Let $I = [r,s]$ with $r \geq r_b$ and let $m \geq m(n,I)$. Let $x \in (\Bt_K^I)^{\Gamma_m-\an}$. 

We start by applying proposition \ref{proof theo induction} with $\ell=d$, so that there is no condition on the nabla operators. Therefore, there exist $(x_j)_{j \in \N} \in (\Bt_K^I)^{\Gamma_m-\an}$ such that $|\!|x_jp^{nj}|\!| \ra 0$, $\nabla_d(x_j)=0$ and $x=\sum_{j \geq 0}x_j(b_\ell-b_n^{\ell})^j$.

Now each $x_j$ belongs to $(\Bt_K^I)^{\Gamma_m-\an}$, and is such that $\nabla_d(x_j)=0$ so that we can apply proposition \ref{proof theo induction} to each $x_j$ with $\ell=d-1$, so that there exist $(x_{jk})_{j,k \in \N}$ in $(\Bt_K^I)^{\Gamma_m-\an}$ such that $|\!|x_jp^{n(j+k)}|\!| \ra 0$, for all $k \geq d-1, \nabla_k(x_{jk})=0$ and $x_j=\sum_{k \geq 0}x_{jk}(b_{d-1}-b_n^{d-1})^k$.

We thus have $x = \sum_{k,j \in \N}x_{jk}(b_d-b_n^d)^j(b_{d-1}-b_n^{d-1})^k$, where $|\!|x_jp^{n(j+k)}|\!| \ra 0$, for all $k \geq d-1, \nabla_k(x_{jk})=0$. We can now apply proposition \ref{proof theo induction} to each $x_{jk}$ with $\ell = d-2$. 

Inductively, we find $(x_{\i})_{\i \in \N^{\{2,\cdots,d\}}}$ such that $|\!|x_{\i}^{n|\i|}|\!| \ra 0$, where $|\i| = \sum_{j=2}^{d}i_j$, for all $k > 1, \nabla_k(x_{\i})=0$ and

$$x = \sum_{\i \in \N^{\{2,\cdots,d\}}}x_{\i}\prod_{j=2}^{d}(b_{j}-b_n^{j})^{i_j}.$$ 

By proposition \ref{prop locana killed by all nabla is cyclo}, the elements $x_{\i}$ all belong to $\phi^{-k}(\B^{p^kI}_{M,\eta})$ for some finite extension $M_\infty^\eta$ of $K_{\infty}^\eta$, depending only on $m$, and $k \geq 0$ depending only on $s$, $m$ and $M$. 

Proposition 7.5 of \cite{colmez2008espaces} and its analogue in the twisted cyclotomic case show that the elements of $\B^J_{M,\eta}$ are power series in one variable over a subfield of $K$ which converge on some annulus depending on $J$, and so we are done. 
\end{proof}

\section{Locally analytic vectors for Robba rings}
In this section we study what happens when we take the locally analytic vectors attached to Lubin-Tate extensions in the corresponding Robba rings $\Bt_{\rig}^\dagger$. In particular, we show that those locally analytic vectors recover objects that were defined by Colmez using completely different methods in \cite{colmez2014serie}. We then use this result to explain why locally analytic vectors are usually not the right object to consider when working with Fréchet rings, and recall some properties satisfied by the corresponding attached modules. Finally, we prove that if $V$ is an $F$-analytic representation of $V$, then $(\Btrigplus \otimes V)^{H_K,\Gamma-\la} \simeq (\Bt_{\rig}^\dagger \otimes V)^{H_K,\Gamma-\la}$.

\subsection{Locally analytic vectors in Robba rings}
If $T$ is a variable and $L$ is a finite extension of $\Qp$, we let $L\langle \langle T \rangle \rangle$ denote the set of power series in $T$ with coefficients in $L$ and with infinite radius of convergence. 

\begin{prop}
\label{prop F-la for Robba}
We have $(\Bt_{\rig,K}^\dagger)^{\la} = (\Bt_{\rig}^+)^{\la} = K\langle \langle t_\pi \rangle \rangle$.
\end{prop}
\begin{proof}
Let $r \geq 0$ and let $z \in (\Bt_{\rig}^{\dagger,r})^{\la}$. It is therefore $\Gamma_n$-analytic for some $n \geq 0$, so that for any $s \geq r$, $z$ is a $\Gamma_n$-$F$-analytic vector of $\Bt^{[r;s]}$. By applying enough times $\phi_q$ to item 1. of \cite[Thm. 4.4]{Ber14MultiLa}, we have that the images of $z$ in $\Bt^{[r;s]}$ all belong to $\B_K^{[r;s]}$ as long as $s$ is such that $r_n \leq s$, where $r_n = p^{nh-1}(p-1)$ was defined in \S 1.3. Taking the inverse limit, this implies that $z \in \B_{\rig,K}^{\dagger,r}$. 

Since $\phi$ commutes with the Galois action, the reasoning above also applies to $\phi_q^{-1}(z)$. Therefore, for all $k \geq 0$, $\phi_q^{-k}(z) \in \B_{\rig,K}^{\dagger}$. This implies that $z$ belongs to the ring $K \langle \langle t_\pi \rangle \rangle$. Indeed, in the cyclotomic case this is \cite[Prop. 3.9]{colmez2014serie}, and for the general case this follows from the same arguments, using the dictionary developped by Colmez in the Lubin-Tate case in \cite[\S 2]{colmezLT}.

To finish the proof, it suffices to notice that any element of $K\langle \langle t_\pi \rangle \rangle$ is indeed $F$-locally analytic (and is actually $\Gamma_0$-analytic). 
\end{proof}

Proposition \ref{prop F-la for Robba} already shows that the set of $F$-analytic vectors of $\Bt_{\rig,K}^\dagger$ is really small compared with the set of $F$-pro-analytic vectors of $\Bt_{\rig,K}^\dagger$. 

We now explain what is $\D_{\rig,K}^{\dagger}(V)^\la$ and prove that its rank as an $E\brax$-module is too small in general. 

Given a $(\phi,\Gamma)$-module $\D$ over $E$ (in the cyclotomic setting), Colmez has defined \cite[\S 3.3]{colmez2014serie} a module denoted by $\D \boxtimes  \{0\}$ by $\cap_{n \geq 0}\phi^n(\D)$, which is a free $(\phi,\Gamma)$-module over $E\langle \langle t \rangle \rangle$ of rank $\leq \dim(V)$ by \cite[Thm. 3.20]{colmez2014serie}. 

\begin{prop}
\label{prop Bcrochu computes D box 0}
Let $V$ be an $F$-analytic representation of $\G_K$. Then 
$$(\Bt_\rig^\dagger \otimes_{\Qp}V)^{H_K,\Gamma_K-\la} = \D_{\rig,K}^\dagger(V)^\la = \cap_{n \geq 0}\phi_q^n(\D_{\rig,K}^\dagger(V)).$$
\end{prop}
\begin{proof}
Let $r \geq 0$ be such that $\D_{\rig}^\dagger(V)$ and all its structures are defined over $\B_{\rig,K}^{\dagger,r}$ and let $I$ be a compact subinterval of $[r,+\infty[$ such that $I \cap qI \neq \emptyset$. Let $x \in (\D_{\rig,K}^{\dagger,r}(V))^{\Gamma_n-\an}$.

Let $m \geq n$ be such that there exists a basis $(b_1,\ldots,b_d)$ of $\D^I(V)$ of $\Gamma_m$-analytic vectors of $\D^I(V)$ (this is possible by the same argument as in \S 2.1 of \cite{KR09}). Then by proposition \ref{lainla and painpa}, this implies that
$$(\D^I(V))^{\Gamma_m-\an} = \oplus_{i=1}^d(\B_K^I)^{\Gamma_m-\an}\cdot b_i.$$

Since $\D_{\rig}^\dagger(V)$ is a $\phi_q$-module, the elements $\phi_q(b_1),\ldots,\phi_q(b_d)$ form a basis of $\D^{qI}(V)$, and are $\Gamma_m$-analytic vectors of $\D^{qI}(V)$ since $\phi_q$ commutes with the Galois action. In particular, we get that 
$$(\D^{qI}(V))^{\Gamma_m-\an} = \oplus_{i=1}^d(\B_K^{qI})^{\Gamma_m-\an}\cdot \phi_q(b_i).$$
Applying inductively the same argument, we see that for $\ell \gg 0$, we have 
$$(\D^{q^{\ell}I}(V))^{\Gamma_m-\an} = \oplus_{i=1}^d(\B_K^{q^{\ell}I})^{\Gamma_m-\an}\cdot \phi_q^{\ell}(b_i).$$

But by applying $\phi_q$ enough times to item 1. of \cite[Thm. 4.4]{Ber14MultiLa}, we see that for $\ell$ big enough, $(\B_K^{q^{\ell+1}I})^{\Gamma_m-\an} = \phi_q((\B_K^{q^{\ell}I})^{\Gamma_m-\an})$, so that for $\ell \gg 0$, $(\D^{q^{\ell+1}I}(V))^{\Gamma_m-\an} = \phi_q((\D^{q^{\ell}I}(V))^{\Gamma_m-\an}$. Therefore, the image of the element $x \in (\D_{\rig,K}^\dagger(V))^{\Gamma_n-\an}$ in $(\D^{q^{\ell+1}I}(V))$ is also in $\phi_q((\D^{q^{\ell}I}(V))$, for $\ell \gg 0$, and thus $x \in \phi(\D_{\rig,K}^\dagger(V))^{\Gamma_n-\an}$. This proves that 
$$\D_{\rig,K}^\dagger(V)^\la = \cap_{n \geq 0}\phi_q^n((\D_{\rig,K}^\dagger(V))^{\la}) \subset \cap_{n \geq 0}\phi_q^n(\D_{\rig,K}^\dagger(V)).$$

Note that $(\Bt_\rig^\dagger \otimes_{\Qp}V)^{H_K} = \Bt_{\rig,K}^\dagger \otimes_{\B_{\rig,K}^\dagger}\D_{\rig,K}^\dagger(V)$.
Using item 3. of \cite[Thm. 4.4]{Ber14MultiLa} and the proof of Theorem 10.4 of ibid, we know that  
$$(\tilde{\D}_{\rig,K}^\dagger(V))^{\Gamma_K-\pa}=\bigcup_{m \geq 0}\phi_q^{-m}(\D_{\rig,K}^\dagger(V)).$$

Since $\Gamma_n$-analytic vectors are in particular also pro-analytic vectors of $\Gamma_K$, this means that 
$$(\tilde{\D}_{\rig,K}^\dagger(V))^{\Gamma_n-\an}=\bigcup_{m \geq 0}(\phi_q^{-m}(\D_{\rig,K}^\dagger(V)))^{\Gamma_n-\an}),$$
by taking the $\Gamma_n$-analytic vectors, and thus $(\Bt_\rig^\dagger \otimes_{\Qp}V)^{H_K,\Gamma_K-\la} = \D_{\rig,K}^\dagger(V)^\la$ since the latter is stable by taking inverse powers of $\phi_q$. 

It remains to prove that $\D_{\rig,K}^\dagger(V)^\la \supset \cap_{n \geq 0}\phi_q^n(\D_{\rig,K}^\dagger(V))$. 

If $x \in \cap_{n \geq 0}\phi_q^n(\D_{\rig,K}^\dagger(V))$, then $x$ belongs to a free $K\langle \langle t_\pi \rangle \rangle$-module which is $\Gamma_K$-stable so that for $g \in \Gal(K_\infty/K)$, $g(x) = \Mat(g)\cdot x$ where $\Mat(g) \in \GL_d(K\langle \langle t_\pi \rangle \rangle)$, so that the Galois action on $x$ is locally analytic. This finishes the proof.
\end{proof}

In particular, the following result of Colmez shows that in general the module $\D_{\rig,K}^\dagger(V)^\la$ is too small:

\begin{prop}
\label{prop colmez too small}
Let $V$ be a two dimensional irreducible representation of $\G_{\Qp}$. If $V$ is not trianguline then $\D_{\rig,K}^\dagger(V)\boxtimes  \{0\} = 0$. 
\end{prop}
\begin{proof}
This is item (i) of theorem 3.23 of \cite{colmez2014serie}. 
\end{proof}

\begin{rema}
\label{remark colmez semistable nonadmiss}
Theorem 3.23 of \cite{colmez2014serie} also says that if $V$ is a semistable, noncrystalline $2$-dimensional representation, then $\D_{\rig,K}^\dagger(V)\boxtimes  \{0\}$ is a $(\phi,\Gamma)$-module of rank $1$ over $E\langle \langle t \rangle \rangle$.  
\end{rema}

\subsection{$\phi$-modules on $L\langle \langle t_\pi \rangle \rangle$}
By proposition \ref{prop Bcrochu computes D box 0}, to any $E$-representation $V$ of $\G_K$ we can attach a module on $E\langle \langle t_\pi \rangle \rangle$, which is endowed with a Frobenius $\phi_q$ and an operator $\nabla$ coming from the action of the Lie algebra of $\Gamma_K$. Note that $\phi_q$ and $\Gamma_K$ act on $E\langle \langle t_\pi \rangle \rangle$ by
$$\phi_q(t_\pi)= \pi t_\pi, \quad g(t_\pi) = \chi_\pi(g)t_\pi.$$

We can also define an operator $\nabla$ on $E\langle \langle t_\pi \rangle \rangle$ by $\nabla_u = t_\pi\frac{d}{dt_\pi}$.

As a matter of fact, $\phi$-modules on $E\langle \langle t_\pi \rangle \rangle$ were already studied by Colmez in \cite[3.1]{colmez2014serie} and the results proved by Colmez show that $\phi$-modules on $E\langle \langle t_\pi \rangle \rangle$ are not as bad as one may think. Be careful that what we call $E\langle \langle t_\pi \rangle \rangle$ corresponds in the notations of Colmez to $E{\{\{t_\pi\}\}}$. In this section, we recall Colmez's results on $\phi$-modules on $E\langle \langle t_\pi \rangle \rangle$. 

\begin{defi}
A $(\phi_q,\Gamma_K)$-module on $E\langle \langle t_\pi \rangle \rangle$ is a finite free $E\langle \langle t_\pi \rangle \rangle$-module, endowed with semilinear actions of $\phi_q$ and $\Gamma_K$ which commute one to another and such that $\phi_q$ is an isomorphism. 

A $(\phi_q,\nabla)$-module on $E\langle \langle t_\pi \rangle \rangle$ is a finite free $E\langle \langle t_\pi \rangle \rangle$-module, endowed with semilinear actions of $\phi_q$ and $\nabla$ which commute one to another and such that $\phi_q$ is an isomorphism. 
\end{defi}

A $(\phi_q,\Gamma_K)$-module on $E\langle \langle t_\pi \rangle \rangle$ gives rise to a $(\phi_q,\nabla)$-module on $E\langle \langle t_\pi \rangle \rangle$ by taking the same $\phi$-structure and taking $\nabla$ to be the operator $\frac{\log(g)}{\log\chi_\pi(g)}$ for $g$ close enough to $1$. 

The ring $E\langle \langle t_\pi \rangle \rangle$ can be interpreted \textit{via} analytic functions, as it is the projective limit of the rings of analytic functions on the disks $v_p(x) \geq -ne$ for $n \in \N$. Those rings are principal Banach rings and therefore $E\langle \langle t_\pi \rangle \rangle$ is a Fréchet-Stein ring, which in particular implies that any closed submodule of a free module of rank $d$ is free of rank $\leq d$ and that a submodule of finite type of a free finite type module is closed and thus free. Moreover, Newton polygons theory show that an element $f \in E\langle \langle t_\pi \rangle \rangle$ does not vanish if and only if $f \in E^\times$, so that $(E\langle \langle t_\pi \rangle \rangle)^\times = E^\times$.

\begin{lemm}
\label{lemm descent M to M/vEbrax}
Let $M$ be a rank $d$ $\phi_q$-module on $E\langle \langle t_\pi \rangle \rangle$ and let $v \in M$ be such that there exists $\alpha \in E^\times$ such that $\phi_q(v) = \alpha v$. Then there exists $k \in \N$ such that $t_\pi^{-k}v \in M$ and $M/E\langle \langle t_\pi \rangle \rangle t_\pi^{-k}v$ is free of rank $d-1$ on $E\langle \langle t_\pi \rangle \rangle$.
\end{lemm}
\begin{proof}
See \cite[Lemm. 3.4]{colmez2014serie}.
\end{proof}

Let $M$ be a $\phi_q$-module on $E\langle \langle t_\pi \rangle \rangle$ and let $\overline{M} = M/t_\pi M$. If $P \in E[X]$ is unitary of degree $d$ and irreducible, then we let $M_P$ (resp. $\overline{M}_P$) denote the set of elements $v \in M$ (resp. in $\overline{M}$) such that $P(\phi_q)^n\cdot v = 0$ for $n \gg 0$ and if $k \in \N$, we let $P[k]$ be the polynomial $\pi^{kd}P(X/\pi^k)$.

\begin{theo}
\label{theo structure phi modules on Ebrax}
If $M$ is a $\phi_q$-module of rank $d$ on $E\langle \langle t_\pi \rangle \rangle$, there exists a basis $e_1,\ldots,e_d$ of $M$ in which the matrix of $\phi$ is $A+N$, where $A \in \GL_d(L)$ is semisimple and invertible, and $N$ is nilpotent and commutes with $A$. Moreover, $N$ splits into $N= N_0+t_\pi N_1+\ldots$, where $N_i \in M_d(L)$ sends the kernel $M_P$ of $P(A)$ into the one $M_{P[-i]}$ of $P(\pi^iA)$ for all $P$ (and thus in particular the sum is finite).
\end{theo}
\begin{proof}
See \cite[Thm. 3.6]{colmez2014serie}.
\end{proof}

Given a $(\phi_q,\Gamma_K)$-module $D$ on $E\langle \langle t_\pi \rangle \rangle$, we say that an element $v$ of $D$ is proper for the action of $\phi_q$ and $\Gamma_K$ if there exists an $F$-analytic character $\delta : K^\times \to E^\times$ such that $\phi(v)= \delta(\pi)v$ and $g(v)=\delta(\chi_\pi(g))v$ for all $g \in \Gamma_K$. 

\begin{lemm}
\label{lemm delta proper}
Given a $(\phi_q,\Gamma_K)$-module $D$ of rank $1$ on $E\langle \langle t_\pi \rangle \rangle$, $F$-analytic, with basis $e$, then there exists an $F$-analytic character $\delta : K^\times \ra E^\times$ such that $e$ is proper for $\delta$.
\end{lemm}
\begin{proof}
This just follows from the fact that a rank $1$ $(\phi_q,\Gamma_n)$-module on $E\langle \langle t_\pi \rangle \rangle$ has a unique basis $e$, up to multiplication by an element of $(E\langle \langle t_\pi \rangle \rangle)^{\times} = E^\times$.
\end{proof}

\subsection{Frobenius regularity}
We now explain how to use the fact that our rings are embedded with a Frobenius in order to show some regularity property. This section is in the same spirit as \cite[\S 3.1 and \S 3.2]{Ber02}.

\begin{lemm}
\label{lemm regfrob}
Let $h$ be a positive integer. Then
$$\bigcap_{s=0}^{+\infty}\pi^{-hs}\At^{\dagger,q^{-s}r}=\Atplus \quad \textrm{and} \quad \bigcap_{s=0}^{+\infty}\pi^{-hs}\At^{\dagger,q^{-s}r}_{\rig} \subset \Btrigplus.$$
\end{lemm}
\begin{proof}
This is \cite[Lemm. 3.1]{Ber02} when $K=\Qp$. The generalization when $K$ is a finite extension of $\Qp$ is straightforward.
\end{proof}

\begin{prop}
\label{prop frobreg crochu}
Let $r,v$ be two positive integers, and let $A \in \M_{v \times r}(\Bt_{\rig}^\dagger)$. If $P \in \GL_v(K\brax)$ is such that $P \in \M_v(K[t_\pi])$ and such that $A = P\phi_q^{-1}(A)$ then $A \in \M_{v \times r}(\Btrigplus)$. 
\end{prop}
\begin{proof}
Write $A$ as $(a_{ij})$ and $P$ as $(p_{ij})$. Let $h_0$ be such that the $\pi^{h_0}p_{ik}$ belong to $\O_F[t_\pi]$, and let $n$ be the highest degree of the $p_{ik}$ as polynomials in $t_\pi$. The assumption on the relation between $P$ and $A$ can be translated as:
$$p_{i1}\phi^{-1}(a_{1j})+\cdots+p_{iv}\phi^{-1}(a_{vj}) = a_{ij} \quad \forall i \leq v, j \leq r.$$
Let $c > 0$ and $r \geq 0$ be such that the $a_{ij}$ belong to $p^{-c}\At_{\rig}^{\dagger,r}$. Using the relation between $P$ and $A$, this implies that the $a_{ij}$ belong to $p^{-h_0-c}\At_{\rig}^{\dagger,r/q}\cdot\O_K[t_\pi]_n$, where $\O_K[t_\pi]$ denotes the ring of polynomials in $t_\pi$ with degree $\leq n$. Note that if $t_\pi \in \pi^{-\beta}\At_{\rig}^{\dagger,r}$, then since $\phi_q^{-1}(t_\pi) = \frac{1}{\pi}t_\pi$, we get that $t_\pi \in \pi^{1-\beta}\At_{\rig}^{\dagger,r/q}$ and thus $t_\pi \in \pi^{-\beta}\At_{\rig}^{\dagger,r/q^{\ell}}$ for all $\ell \geq 0$. We therefore have that the $a_{ij}$ belong to $\pi^{-h_0-c-n\beta}\At_{\rig}^{\dagger,r/q}$, and applying the result inductively, the $a_{ij}$ belong to $\pi^{-c-h_0\ell-n\beta\ell}\At_{\rig}^{\dagger,rq^{-\ell}}$. We can thus apply lemma \ref{lemm regfrob} to the $p^{c}a_{ij}$, which shows that they belong to $\Btrigplus$, as we wanted.
\end{proof}
 
\begin{prop}
\label{prop Btn equiv Btrigdaggercrochu}
Let $V$ be an $F$-analytic $E$-representation of $\G_K$. Then the morphism  
$$(\Btrigplus \otimes V)^{H_K,\Gamma_K-\la} \rightarrow (\Bt_{\rig}^\dagger\otimes V)^{H_K,\Gamma_K-\la}$$
induced by the inclusion $\Btrigplus \subset \Bt_{\rig}^{\dagger}$, is an isomorphism of $(\phi_q,\nabla)$-modules on $E\brax$. 
\end{prop}
\begin{proof}
Let $(v_1,\ldots,v_r)$ and $(d_1,\ldots,d_v)$ be respectively an $E$-basis of $V$ and an $E\brax$-basis of $(\Bt_{\rig}^{\dagger}\otimes_{E}V)^{\la}$. There exists $A \in \M_{r \times v}(\Bt_{\rig}^\dagger)$ such that $(d_i)=A(v_i)$. Let $P \in \GL_v(E\brax)$ be the matrix of $\phi_q$ in the basis $(d_i)$. By theorem \ref{theo structure phi modules on Ebrax}, we can assume that the basis $(d_i)$ of $(\Bt_{\rig}^{\dagger}\otimes_{E}V)^{\la}$ is such that $P \in \M_v(K[t_\pi])$. We then have $\phi_q(A) = PA$ and thus $A = \phi_q^{-1}(P)\phi_q^{-1}(A)$. By proposition \ref{prop frobreg crochu}, we have $A \in M_{r \times v}(\Btrigplus)$ and hence 
$$(\Bt_{\rig}^{\dagger}\otimes_{E}V)^{\la} \subset (\Btrigplus\otimes_{E}V)^{\la}.$$
\end{proof}

\section{Applications to trianguline representations}
We now explain how some of the rings previously introduced provide some results towards the question of the existence of a ring of periods for trianguline representations. We will start by recalling the notions of trianguline representations and refinements.

In a previous version of this paper, we claimed that trianguline representations of $\G_{\Qp}$ were admissible for the ring $\cal{C}^{\la}(\Gamma_K,\Btrigplus)_1$ but there was a gap in the proof and the claim is actually not true. We do expect though that if such a ring exists then it has to be some intermediate ring $\B$ between $\cal{C}^{\la}(\Gamma_K,\Btrigplus)_1$ and the rings $\cal{C}^{\la}(\Gamma_K,\Bt^I)_1$, but it is not clear at all ``how many periods we have to add'' to $\cal{C}^{\la}(\Gamma_K,\Btrigplus)_1$. These constructions naturally extend to the $F$-analytic Lubin-Tate case.

\subsection{Trianguline representations and refinements}
We start by recalling the definitions of trianguline representations and some associated properties. The notion of trianguline representations was introduced by Colmez in \cite{colmez2008representations}. Here  we choose to follow Berger's and Chenevier's definitions \cite{BerChe10} instead of Colmez's.

\begin{defi}
\label{defi trianguline}
We say that an $E$-representation $V$ of $\G_K$ is split trianguline if $\D_{\rig}^\dagger(V)$ is a successive extension of $(\phi,\Gamma_K)$-modules of rank $1$ over $E \otimes_{\Qp}\B_{\rig,K}^\dagger$. 

We say that an $L$-representation $V$ of $\G_K$ is trianguline if there exists a finite extension $E$ of $L$ such that the $E$-representation $E \otimes_L V$ is split trianguline.

We say that an $E$-representation $V$ of $\G_K$ is potentially split trianguline (resp. potentially trianguline) if there exists a finite extension $K'$ of $K$ such that $V_{|\G_{K'}}$ is split trianguline (resp. trianguline).
\end{defi}

Definition \ref{defi trianguline} can be equivalently stated in terms of $B$-pairs. We quickly recall that a $B$-pair is a pair $W = (W_e,W_{\dR}^+)$, where $W_e$ is a free $\B_e := (\Btrigplus[1/t])^{\phi=1}$-module of finite rank endowed with a continuous semi-linear action of $\G_K$, and $W_{\dR}^+$ is a $\G_K$-stable lattice in $W_{dR}:= \Bdr \otimes_{\B_e}W_e$. To a $p$-adic representation $V$ of $\G_K$ one can attach a $B$-pair $W(V)$ by $W(V) = (\B_e \otimes_{\Qp}V,\Bdrplus \otimes_{\Qp}V)$. If $E$ is a finite extension of $\Qp$, one extends the definition of $B$-pairs to $E$-linear objects, and one gets objects called $\BEK$-pairs in \cite{BerChe10} or $E-B$-pairs of $\G_K$ in \cite{Nakapieds}. Those objects are pairs $W = (W_e,W_{\dR}^+)$, where $W_e$ is a free $\B_{e,E}:=E \otimes_{\Qp}\B_e$-module of finite rank endowed with a continuous semi-linear action of $\G_K$, and $W_{\dR}^+$ is a $\G_K$-stable lattice in $W_{dR}:= (E \otimes_{\Qp} \Bdr) \otimes_{\B_{e,E}}W_e$.

The category of $\BEK$-pairs is equivalent to the one of $(\phi,\Gamma_K)$-modules over $E \otimes_{\Qp}\B_{\rig,K}^\dagger$, and thus an $E$-representation $V$ of $\G_K$ is split trianguline if the attached $\BEK$-pair is a successive extension of rank $1$ $\BEK$-pairs. 

\begin{lemm}
Let $V$ be an $F$-analytic representation. Then the following are equivalent:
\begin{enumerate}
\item $V$ is split trianguline.
\item The Lubin-Tate $(\phi_q,\Gamma_K)$-module $\D_{\rig}^\dagger(V)$ is a successive extension of $F$-analytic Lubin-Tate $(\phi_q,\Gamma_K)$-modules of rank $1$.
\end{enumerate}
\end{lemm}
\begin{proof}
See \cite[Thm. 4.11]{PnoteBpairs}.
\end{proof}

For a $\B_{e,E}$-representation, we say that it is split triangulable if it is a successive extension of rank $1$ $\B_{e,E}$-representations.

\begin{lemm}
An $E$-representation $V$ of $\G_K$ is split trianguline if and only if the corresponding $\B_{e,E}$-representation is split triangulable as a $\B_{e,E}$-representation of $\G_K$.
\end{lemm}
\begin{proof}
See \cite[Coro. 3.2]{BMTensTriang}.
\end{proof}

\begin{prop}
\label{prop triang cat stable by}
The categories of split trianguline representations and of trianguline representations are stable by subobjects, quotients, direct sums and tensor products.
\end{prop}
\begin{proof}
The fact that it is stable by quotients and subobjects follows from \cite[Prop. 3.3]{BMTensTriang}. For direct sums and tensor products it is a straightforward consequence of definition \ref{defi trianguline}.
\end{proof}

Let $D$ be a $(\phi,\Gamma_K)$-module of rank $d$ over $E \otimes_{\Qp}\B_{\rig,K}^\dagger$ and equipped with a strictly increasing filtration $(\Fil_i(D))_{i=0\dots d}$
:
$$\Fil_0(D):=\{0\} \subsetneq \Fil_1(D) \subsetneq \cdots \subsetneq \Fil_i(D) \subsetneq \cdots \subsetneq \Fil_{d-1}(D)
\subsetneq \Fil_d(D):=D,$$
of $(\phi,\Gamma_K)$-submodules which are direct summand as $E \otimes_{\Qp}\B_{\rig,K}^\dagger$-modules. We call 
such a $D$ a triangular $(\phi,\Gamma_K)$-module  over $E \otimes_{\Qp}\B_{\rig,K}^\dagger$, and the 
filtration $\cal{T} := (\Fil_i(D))$ a triangulation of $D$ over $E \otimes_{\Qp}\B_{\rig,K}^\dagger$.

Let $D$ be a triangular $(\phi,\Gamma_K)$-module. 
By proposition 3.1 of \cite{colmez2008representations}, each 
$$\mathrm{gr}_i(D):=\Fil_i(D)/\Fil_{i-1}(D), \, \, \, 1\leq i \leq d,$$ 
is isomorphic to the $(\phi,\Gamma_K)$-module on $E \otimes_{\Qp}\B_{\rig,K}^\dagger$ attached to a character $\delta_i$ for some unique $\delta_i: K^\times \ra E^\times$. Following \cite[2.3.2]{bellaichechenevier}, we define the parameter of the triangulation to be the continuous
homomorphism $$\delta:=(\delta_i)_{i=1,\cdots,d}: K^\times \ra (E^\times)^d.$$

When $K=\Qp$, the parameter of a triangular $(\phi,\Gamma_K)$-module refines the data of its Sen
polynomial:

\begin{prop} Let $D$ be a triangular $(\phi,\Gamma)$-module over $E \otimes_{\Qp}\B_{\rig,\Qp}^\dagger$ and $\delta$ the parameter of a triangulation of $D$. Then the Sen polynomial of $D$ is 
$$\prod_{i=1}^d(T-w(\delta_i)).$$
\end{prop}
\begin{proof}
See \cite[Prop. 2.3.3]{bellaichechenevier}.
\end{proof}

We now recall the notion of refinements for crystalline trianguline representations of $\G_{\Qp}$ as in \cite[\S 2.4]{bellaichechenevier}. Let $V$ be finite, $d$-dimensional, continuous, $E$-representation of $\G_{\Qp}$. 
We will assume that $V$ is crystalline and that the crystalline Frobenius $\phi$ acting on 
$\D_{\crys}(V)$ has all its eigenvalues in $E^\times$. 

By a refinement of $V$, using the definition of \cite[\S3]{mazur2000theme}, we mean the data of a full $\varphi$-stable
$E$-filtration $\Ref=(\Ref_i)_{i=0,\dots,d}$ of $\D_{\crys}(V)$:
$$\Ref_0=0 \subsetneq \Ref_1 \subsetneq \cdots \subsetneq \Ref_d=\D_{\crys}(V).$$
As in \cite[2.4.1]{bellaichechenevier}, we remark that any refinement $\Ref$ determines two orderings: 
\begin{enumerate}
\item It determines an ordering 
$(\varphi_1,\cdots,\varphi_d)$ of the eigenvalues of $\varphi$, defined by the formula $$\det(T-\varphi_{|\Ref_i})=\prod_{j=1}^i (T-\varphi_j).$$
If all these eigenvalues are distinct then such an ordering conversely determines $\Ref$. 
\item It determines also an ordering $(s_1,\cdots,s_d)$ on the set of Hodge-Tate weights of $V$, defined by the property that the jumps of the weight filtration of $\D_{\crys}(V)$ induced on $\Ref_i$ are $(s_1,\cdots,s_i)$
\end{enumerate}

The theory of refinements has a simple interpretation in terms of 
$(\phi,\Gamma)$-modules: let $D$ be a crystalline $(\phi,\Gamma)$-module as above and
let $\Ref$ be a refinement of $D$. We can construct from $\Ref$ a
filtration $(\Fil_i(D))_{i=0,\cdots,d}$ of $D$ by setting
	$$\Fil_i(D):=(E \otimes_{\Qp}\B_{\rig,\Qp}^\dagger[1/t]\Ref_i)\cap D,$$
which is a finite type saturated $E \otimes_{\Qp}\B_{\rig,{\Qp}}^\dagger$-submodule of $D$.

\begin{prop} \label{ref=tri} The map defined above $(\Ref_i) \mapsto
(\Fil_i(D))$ induces
a bijection between the set of refinements of $D$ and the set of
triangulations of $D$, whose inverse is
$\Ref_i:=\Fil_i(D)[1/t]^{\Gamma}$.
	In the bijection above, for $i=1,\dots, d$, the graded piece
$\Fil_i(D)/\Fil_{i-1}(D)$ is
isomorphic to the $(\phi,\Gamma)$-module on $E \otimes_{\Qp}\B_{\rig,{\Qp}}^\dagger$ attached to $\delta_i$ where $\delta_i(p)=\varphi_i p^{-s_i}$ and 
${\delta_i}_{|\Gamma}=\chi^{-s_i}$, where the $\varphi_i$ and $s_i$ are defined by items $1$ and $2$ above. 
\end{prop}
\begin{proof}
See \cite[Prop. 2.4.1]{bellaichechenevier}. 
\end{proof}

\begin{rema}
In particular, Proposition \ref{ref=tri} shows that crystalline representations are trianguline, and that the set of their triangulations is in natural bijection with the set of their refinements.
\end{rema}

We now finish this section with a result regarding trianguline representations that we were not able to find in the litterature.

\begin{prop}
\label{prop Qp same as L rep for triang}
Let $V$ be an $L$-representation of $\G_K$. Then $V$ is trianguline if and only if the underlying $\Qp$-representation of $V$ is trianguline.
\end{prop}
\begin{proof}
Let $V$ be an $L$-representation of $\G_K$ and let $E$ be a finite extension of $L$, containing all the images of the embeddings $\tau : L \to \overline{K}$ and such that $E \otimes_L V$ is split trianguline. Then $E \otimes_{\Qp}V = (E \otimes_{\Qp}L) \otimes_L V = \oplus_{\tau \in \Sigma}(E \otimes_L V)_\tau$ where $\Sigma = \Emb(L,\overline{K})$.

In particular, $E \otimes_L V$ is a subrepresentation of $E\otimes_{\Qp}V$ and this concludes the first half of the proof by proposition \ref{prop triang cat stable by}. For the other direction, let $W = W_e(E \otimes_L V)$ the corresponding $\B_{e,E}$-representation and let $W_0 = 0 \subset W_1 \subset \ldots W_d = W$ a triangulation of $W$. For $\tau \in \Sigma$, let $\B_{e,E,\tau} = E \otimes_{L,\tau}\B_{e,E}$. For $\tau \in \Sigma$ and $1 \leq i \leq d$, let $W_{i,\tau} = \B_{e,E,\tau} \otimes_{\B_{e,E}}W_{i}$. By construction 
$$0 \subset W_{1,\tau} \subset \ldots \subset W_{d,\tau}$$ 
is a triangulation of $W((E \otimes_L V)_\tau)$ and thus $E \otimes_{\Qp}V$ is trianguline.
\end{proof}

\subsection{Discussion on a ring of periods for trianguline representations}
By proposition \ref{prop triang cat stable by}, we know that the category of (split) trianguline representations of $\G_K$ is a Tannakian category. Because of this and because of proposition \ref{prop Qp same as L rep for triang}, it appears reasonable to look for a ring $\B$ such that trianguline representations are exactly the representations which are $\B$-admissible in the sense of Fontaine.

Recall that the notion of admissibility in the sense of Fontaine is defined for what he called regular rings and is as follows (we only recall the definitions of \cite{fontaine1994representations} in the particular case of $\Qp$-representations because that's all we need here). 

Let $\B$ be a topological $\Qp$-algebra endowed with an action of a group $G$. For any $\Qp$-representation of $G$, we let $\D_{\B}(V):=(\B\otimes_{\Qp}V)^G$. We let $\alpha_{\B}(V)$ denote the $\B$-linear map $\B \otimes_{\B^G}\D_{\B}(V) \longrightarrow \B \otimes_{\Qp}V$ deduced from the inclusion $\D_{\B}(V) \subset \B\otimes_{\Qp}V$ by extending the scalars to $\B$. The ring $\B$ is said to be $G$-regular if the following hold:
\begin{enumerate}
\item $\B$ is reduced;
\item for any $p$-adic representation $V$ of $G$, the map $\alpha_V$ is injective;
\item any element $b$ of $\B$ which is nonzero and is such that the $\Qp$-line generated by $\B$ is $G$-stable is invertible.
\end{enumerate}
The last condition implies in particular that $\B^G$ is a field. If $\B$ is $G$-regular, a representation $V$ of $G$ is said to be $\B$-admissible if $\alpha_{\B}(V)$ is an isomorphism, which is equivalent as saying that $\dim_{\B^G}\D_{\B}(V)=\dim_{\Qp}V$. 

Unfortunately, it seems to us that in the case we consider, the last condition is too strong and thus we extend the notion of $G$-regularity as follows: we say that $\B$ is $G$-regular if the following conditions are met:
\begin{enumerate}
\item $\B$ is reduced;
\item for any $p$-adic representation $V$ of $G$, $\D_{\B}(V)$ is a free $\B^G$-module;
\item the map $\alpha_V$ is injective.
\end{enumerate}

It is clear that $G$-regular rings in the sense of Fontaine are $G$-regular for us, but that the converse does not hold. 

In the rest of the paper, $G$-regularity and admissibility are to be understood in our sense.

We now explain exactly what we mean by a ring of trianguline periods. 

\begin{defi}
A $\G_K$-regular ring $\B$ is said to be a trianguline periods ring for $\G_K$ if trianguline representations of $\G_K$ are $\B$-admissible, and if $\B$-admissible representations of $\G_K$ are trianguline.
\end{defi}

\begin{prop}
Let $\B$ be a $\G_K$-regular ring and let $V$ be an $L$-representation of $\G_K$. Then $V$ is $\B$-admissible if and only if there exists a finite extension $E$ of $L$ such that $V \otimes_L E$ is $\B$-admissible.
\end{prop}
\begin{proof}
It's clear that if $V$ is $\B$-admissible, then there exists a finite extension $E$ of $L$ such that $V \otimes_L E$ is $\B$-admissible. To show the reverse, first note that the admissibility of an $E$ representation $V$ does not depend on wether one considers it as a $\Qp$-representation or as an $E$-representation (the $\B^{\G_K}$-module $\D_{\B}(V) = (V \otimes_{\Qp}\B)^{\G_K}$ is always the same). Now because the category of $\B$-admissible representations is clearly stable by subobjects, it suffices to note that $V$ is a sub-$\Qp$-representation of $V \otimes_L E$.
\end{proof}

Unlike in the crystalline or semistable case, if such a ring exists, it has to depend on $K$:

\begin{prop}
\label{B depend de K}
There is no ring $\B$ satisfying the properties above such that, for any finite extension $K$ of $\Qp$, $\B$ is a trianguline periods ring for $\G_K$.
\end{prop}

Proposition \ref{B depend de K} is a consequence of the following result:

\begin{prop}
\label{prop construct pot triang not triang}
Let $L/K$ be any finite extension. Then there exists a representation $V$ of $\G_K$ such that $V$ is trianguline as a representation of $\G_L$ but is not trianguline as a representation of $\G_K$.
\end{prop}
\begin{proof}
Let $\eta : \G_L \to L^\times$ be a character such that there exists $\tau_1 \neq \tau_2 \in \Emb(L,\Qpbar)$ with $(\tau_1)_{|K}=(\tau_2)_{|K}$, and such that $\eta$ is $\tau_1$-de Rham but not $\tau_2$-de Rham in the sense of \cite{Dingpartially}. Our claim is that such a character can't possibly extend to $\G_K$ and neither can any of its conjugate, i.e. there is no character $\rho : \G_K \to L^\times$ such that $\rho_{|\G_L}=\sigma(\eta)$ for some $\sigma \in \Emb(L,\Qpbar)$. Indeed, if such a $\rho$ existed, then the dimension of $\D_{\dR,\sigma}(\sigma(\eta))$ would only depend of the dimension of $\D_{\dR,\sigma_{|K}}(\rho)$, which is not the case because of the assumption on $\tau_1$ and $\tau_2$. 

We now let $V=\mathrm{ind}_{\G_L}^{\G_K}\eta$. This is a $p$-adic representation of $\G_K$, whose restriction to $\G_L$ is the sum of the conjugates of $\eta$, so that it clearly is trianguline as a representation of $\G_L$. Let us assume that it also is trianguline as a representation of $\G_K$. Let $W$ be the $\B_{|K}^{\otimes L}$-pair attached to $V$. As a $\B_{|L}^{\otimes L}$-pair, we can write 
$$W=\oplus_{\sigma}W(\sigma(\eta)).$$
Since we assumed that $W$ is trianguline as a $\B_{|K}^{\otimes L}$-pair, there exists $W_1 \subset W$ a direct summand of rank $1$. For $\tau \in \Emb(L,\Qpbar)$, we have the following exact sequence
$$0 \ra \oplus_{\tau \neq \sigma}W(\sigma(\eta)) \ra W \ra W(\tau(\eta)) \ra 0$$
so that, since $W_1$ is a direct summand of rank $1$ of $W$ and by proposition 2.4 of \cite{BMTensTriang}, we either have $W_1 = W(\tau(\eta))$ or $W_1 \subset \oplus_{\tau \neq \sigma}W(\sigma(\eta))$. By induction, $W_1$ has to be equal to one of the $W(\tau(\eta))$. Therefore, one of the conjugates of $\eta$ has to extend to $\G_K$, which we have proven is not possible.  
\end{proof} 

We can now give a proof of proposition \ref{B depend de K}:

\begin{proof}
By the results of \S 5, the periods of any representation live in $\cal{C}^{\la}(\Gamma_K,\Bt^I)_1$, for $I$ any compact subinterval of $[r_0;+\infty[$, and in particular so do the periods of any trianguline representation $V$ of $\G_K$, so that we can assume that $\B \subset \cal{C}^{\la}(\Gamma_K,\Bt^I)_1$. Since every unramified representation of $\G_K$ is trianguline, we can assume that $\B \supset \hat{\Qp^{unr}}$. Moreover, if $L/K/\Qp$ are unramified then it is easy to see that $(\cal{C}^{\la}(\Gamma_K,\Bt^I)_1)^{\G_L} = L \otimes_K (\cal{C}^{\la}(\Gamma_K,\Bt^I)_1)^{\G_K}$, and thus we can assume that $\B^{\G_L} = L \otimes_K \B^{\G_K}$.

Assume that $\B$ is a ring satisfying the properties, and such that for any finite extension $K/\Qp$, $\B$ is a trianguline periods ring for $\G_K$. Let $K$ be a finite unramified extension of $\Qp$, let $L$ be a finite unramified extension of $K$. Let $V$ be a $p$-adic representation of $\G_K$ which is not trianguline as a representation of $\G_K$ but becomes trianguline over $\G_L$, which exists by the previous proposition. Let $\D_L = (\B \otimes V)^{\G_L}$. It is a $\B^{\G_L}$-module, endowed with a semilinear action of $\Gal(L/K)$. Let $\D = \D_L^{\Gal(L/K)}$. By Speiser's lemma, $\D_L \simeq L \otimes_{K}\D$ and thus $\D_L = \B^{\G_L}\otimes_{\B^{\G_K}}\D$. Thus, $V$ is $\B$-admissible as a representation of $\G_K$.
\end{proof}

\subsection{$F$-analytic $\B_{\tri,K}^{\an}$-admissible representations}
Since we are now thinking of rings of periods for trianguline representations, we let $\B_{\tri,K}^{\an}=\cal{C}^{\la}(\Gamma_K,\Btrigplus)_1$. We now explain why $\B_{\tri,K}^{\an}$ is a good starting candidate as a ring of trianguline periods. Note that by proposition \ref{prop Bcrochu computes D box 0} and remark \ref{remark colmez semistable nonadmiss} we already know that there are $F$-analytic trianguline representations of $\G_K$ which are not $\B_{\tri,K}^{\an}$-admissible. To put some emphasis on the point of view of rings of periods, we write $\D_{\tri,K}^{\an}(V)$ for $(\D_{\rig}^{\dagger}(V))^{\la}$ and $\D_{\tri,K}^{n,\an}(V)$ for $(\D_{\rig}^{\dagger}(V))^{\Gamma_n-\an} = (\cal{C}^{\an}(\Gamma_n,\Btrigplus) \otimes_{\Qp}V)^{\G_{K_n}}$.

\begin{prop}
The ring $\B_{\tri,K}^{\an}$ is $\G_K$-regular for $F$-analytic representations.
\end{prop}
\begin{proof}
By proposition \ref{prop Btn equiv Btrigdaggercrochu}, it suffices to prove that the ring $\cal{C}^{\la}(\Gamma_K,\Bt_{\rig}^\dagger)_1$ is $\G_K$-regular for $F$-analytic representations. But now this follows from proposition \ref{prop Bcrochu computes D box 0} and in the cyclotomic case from lemma 3.19 and theorem 3.20 of \cite{colmez2014serie}. The proof of those results extend to the Lubin-Tate case verbatim. 
\end{proof}

We now define a notion of refinements for $F$-analytic representations of $\G_{K}$ which are $\B_{\tri,K}^{\an}$-admissible. We let $V$ be a $\B_{\tri,K}^{\an}$-admissible $L$-representation of dimension $d$ of $\G_{K}$. By a refinement of $V$, we mean the data of a full $\phi_q$- and $\Gamma_K$-stable $L\langle \langle t_\pi \rangle \rangle$-filtration $\Ref=(\Ref_i)_{i=0,\dots,d}$ of $(\D_{\rig}^\dagger(V))^{\la}$:
$$\Ref_0=0 \subsetneq \Ref_1 \subsetneq \cdots \subsetneq \Ref_d=(\D_{\rig}^{\dagger}(V))^{\la}.$$

Note that, as in the crystalline case studied in \cite{bellaichechenevier}, the theory of refinements has a simple interpretation in terms of 
$(\phi_q,\Gamma_K)$-modules: let $\D_{\rig}^\dagger(V)$ be the triangulable $(\phi_q,\Gamma_K)$-module over $L \otimes_{K}\B_{\rig,K}^\dagger$ attached to $V$ and let $\Ref$ be a refinement of $\D_{\rig}^{\dagger}(V)\boxtimes \{0\}$. We can construct from $\Ref$ a
filtration $(\Fil_i(D))_{i=0,\cdots,d}$ of $\D_{\rig}^\dagger(V)$ by setting
	$$\Fil_i(\D_{\rig}^\dagger(V)):=((L \otimes_{K}\B_{\rig,K}^\dagger)\otimes_{L\langle \langle t_\pi \rangle \rangle}\Ref_i)$$
which is a finite type saturated $E \otimes_{K}\B_{\rig,K}^\dagger$-submodule of $\D_{\rig}^\dagger(V)$.

\begin{prop} \label{gen refinements = trianguline} The map defined above $(\Ref_i) \mapsto
(\Fil_i(\D_{\rig}^\dagger(V)))$ induces
a bijection between the set of refinements of $V$ and the set of
triangulations of $\D_{\rig}^\dagger(V)$, whose inverse is
$\Ref_i:=\left((\Fil_i(\D_{\rig}^{\dagger}(V)))^{\la}\right)$.
\end{prop}
\begin{proof}
This is exactly as in the crystalline case.
\end{proof}

\begin{prop}
\label{exists refinement}
Let $M$ be a $(\phi_q,\Gamma_K)$-module of rank $d$ on $L\langle \langle t_\pi \rangle \rangle$. Then, up to extending the scalars to some finite extension $E$ of $L$, there exists a filtration 
$$M_0 = 0 \subsetneq M_1 \subsetneq \ldots \subsetneq M_d = M$$
of $M$ by saturated sub-$(\phi_q,\Gamma_K)$-modules.
\end{prop}
\begin{proof}
We prove the result by induction. If $d=1$ there is nothing to prove. Assume now that $d \geq 2$ and that the result holds for $d-1$.

By theorem \ref{theo structure phi modules on Ebrax} and lemma \ref{lemm descent M to M/vEbrax}, upto replacing $L$ by a finite extension $E'$ of $L$, there exists $e_1$ proper for the action of $\phi$ and $\Gamma_K$ such that $E'\langle \langle t_\pi \rangle \rangle \cdot e_1$ is saturated in $M$. By induction, $M/(E'\langle \langle t_\pi \rangle \rangle\cdot e_1)$ admits a full $(\phi_q,\Gamma_K)$-stable filtration $(\Ref_i)_{i=1}^{d-1}$. We let $M_{i+1}$ be a lift of $\Ref_i$ containing $E'\langle \langle t_\pi \rangle \rangle\cdot e_1$ and we put $M_1 = E'\langle \langle t_\pi \rangle \rangle\cdot e_1$. We then have that $(M_i)_{i=1}^d$ is a full $(\phi_q,\Gamma_K)$-stable filtration of $M$. 
\end{proof}

\begin{theo}
\label{theo Btri-adm is triang}
Let $V$ be a $\B_{\tri,K}^{\an}$-admissible $F$-analytic $p$-adic representation $V$. Then $V$ is trianguline. 
\end{theo}
\begin{proof}
By proposition \ref{exists refinement} and proposition \ref{gen refinements = trianguline}, there exists a finite extension $L$ of $K$ such that $\D_{\rig}^\dagger(V \otimes_{K}L)$ is a triangulable $(\phi_q,\Gamma_K)$-module over $L \otimes_{K}\B_{\rig,K}^\dagger$. 

Moreover, we see from lemma \ref{lemm delta proper} that the characters appearing in the triangulation are $F$-analytic.
\end{proof}

\begin{lemm}
Let $V$ be an $F$-analytic $E$-representation of $\G_{K}$ such that the attached $(\phi_q,\Gamma_K)$-module $\D_{\rig}^\dagger(V)$ is triangulable, and let $\delta : (K^\times)^d \ra (E^\times)^d$ be the parameter of a triangulation of $\D_{\rig}^\dagger(V)$. Then in an adapted basis for the refinement of $\D_{\tri,K}^{\an}(V)$ corresponding to $\delta$ by proposition \ref{gen refinements = trianguline}, the matrices of $\nabla$ and $\phi_q$ are respectively of the form:
\[
\begin{pmatrix}
w(\delta_1) & * & \cdots & * \\
0 & w(\delta_2) & \cdots & * \\
\vdots  & \vdots  & \ddots & \vdots  \\
0 & 0 & \cdots & w(\delta_d)
\end{pmatrix}
\quad \textrm{and }
\begin{pmatrix}
\delta_1(\pi) & * & \cdots & * \\
0 & \delta_2(\pi) & \cdots & * \\
\vdots  & \vdots  & \ddots & \vdots  \\
0 & 0 & \cdots & \delta_d(\pi)
\end{pmatrix}.
\]
\end{lemm}
\begin{proof}
We prove it by induction on $d$. For $d=1$, by proposition 3.1 of \cite{colmez2008representations} there exists a basis $e_\delta$ of $\D_{\rig}^\dagger(V)$ in which $g(e_\delta) = \delta(\chi_{\pi}(g))e_\delta$ and $\phi_q(e_\delta)=\delta(\pi)e_\delta$, and the action if $F$-analaytic, so that $e_\delta$ is a basis of $\D_{\tri,K}^{n,\an}(V)$ which satisfies the result of the lemma. To see that it is unique note that since $(L\brax)^\times = L^\times$, the matrices of $\nabla$ and $\phi_q$ in an other basis of $\D_{\tri,K}^{n,\an}(V)$ would be the same.

Assume now that $d \geq 2$ is such that the result holds for $d-1$ and let $(\Fil_i(\D_{\rig}^\dagger(V)))_{i=0,\ldots,d}$ be the filtration of $\D_{\rig}^\dagger(V)$ corresponding to the triangulation. Since our constructions are stable by saturated sub-objects, we get by induction that in an adapted basis for the refinement of $\D_{\tri,K}^{n,\an}(\Fil_{d-1}(\D_{\rig}^\dagger(V)))$, the matrices of $\nabla_u$ and $\phi$ are respectively of the form:
\[
\begin{pmatrix}
w(\delta_1) & * & \cdots & * \\
0 & w(\delta_2) & \cdots & * \\
\vdots  & \vdots  & \ddots & \vdots  \\
0 & 0 & \cdots & w(\delta_{d-1})
\end{pmatrix}
\quad \textrm{and }
\begin{pmatrix}
\delta_1(\pi) & * & \cdots & * \\
0 & \delta_2(\pi) & \cdots & * \\
\vdots  & \vdots  & \ddots & \vdots  \\
0 & 0 & \cdots & \delta_{d-1}(\pi)
\end{pmatrix}.
\]
Since our constructions are also stable by quotients by saturated sub-objects and using the proof in the rank $1$ case, we know that the matrices of $\nabla$ and $\phi_q$ in a basis of $\D_{\rig}^\dagger(V)/\Fil_{d-1}(\D_{\rig}^\dagger(V)) \simeq E\otimes_{K}\B_{\rig,K}^\dagger(\delta_d)$ are respectively of the form $(w(\delta_d))$ and $(\delta_d(\pi))$. Therefore, in an adapted basis for the refinement of $\D_{\tri,K}^{\an}(V)$ corresponding to $\delta$, the matrices of $\phi_q$ and $\nabla_u$ are as we wanted.
\end{proof}

In particular, as in the crystalline case, a refinement defines an ordering on both the eigenvalues of $\phi_q$ and on the set of Hodge-Tate weights of $V$, and encodes the data of the Hodge-Tate weights of its parameter.  

Finally, we just remark that given the construction of our rings of periods, it is quite obvious that the modules $\D_{\tri,K}^{n,\an}(V)$ attached to $F$-analytic $p$-adic representations of $\G_K$ contain its crystalline periods:

\begin{prop}
Let $V$ be an $F$-analytic $p$-adic representation. Then 
\begin{enumerate}
\item $\D_{\crys}^+(V) \subset \D_{\tri,K}^{\an}(V)^{\nabla=0}$;
\item $\D_{\crys}(V) \subset (\D_{\tri,K}^{\an}(V)[1/t_\pi])^{\nabla=0}$.
\end{enumerate}
\end{prop}
\begin{proof}
This just follows from the fact that $\D_{\crys}^+(V)=(\Bt_{\rig}^+ \otimes_{\Qp}V)^{\G_K}$ and that $\D_{\crys}(V)=(\Bt_{\rig}^+[1/t_\pi] \otimes_{\Qp}V)^{\G_K}$ by lemma 3.8 of \cite{Porat}.
\end{proof}

\bibliographystyle{amsalpha}
\bibliography{bibli}
\end{document}